\documentclass[aap]{imsart}

\RequirePackage{amsthm,amsmath,amsfonts,amssymb,mathtools}
\RequirePackage[authoryear]{natbib}%% uncomment this for author-year citations
\RequirePackage[colorlinks,citecolor=blue,urlcolor=blue]{hyperref}%% uncomment this for coloring bibliography citations and linked URLs

\usepackage{xurl}
\hypersetup{breaklinks=true}

\startlocaldefs
\theoremstyle{plain} 
\newtheorem{theorem}{Theorem}
\newtheorem*{theorem*}{Theorem}
\newtheorem{lemma}[theorem]{Lemma}
\newtheorem{corollary}[theorem]{Corollary}
\newtheorem{proposition}[theorem]{Proposition}

\theoremstyle{remark}
\newtheorem{remark}{Remark}

\usepackage{macros}

\usepackage{amsmath,amssymb,amsthm,graphicx,enumitem}
\usepackage{color}
\usepackage{xcolor}
\usepackage{enumitem}

\usepackage{ulem}
\usepackage[english]{babel}

\usepackage{amsmath}
\makeatletter
\let\save@mathaccent\mathaccent
\newcommand*\if@single[3]{%
  \setbox0\hbox{${\mathaccent"0362{#1}}^H$}%
  \setbox2\hbox{${\mathaccent"0362{\kern0pt#1}}^H$}%
  \ifdim\ht0=\ht2 #3\else #2\fi
  }
\newcommand*\rel@kern[1]{\kern#1\dimexpr\macc@kerna}
\newcommand*\widebar[1]{\@ifnextchar^{{\wide@bar{#1}{0}}}{\wide@bar{#1}{1}}}
\newcommand*\wide@bar[2]{\if@single{#1}{\wide@bar@{#1}{#2}{1}}{\wide@bar@{#1}{#2}{2}}}
\newcommand*\wide@bar@[3]{%
  \begingroup
  \def\mathaccent##1##2{%
    \let\mathaccent\save@mathaccent
    \if#32 \let\macc@nucleus\first@char \fi
    \setbox\z@\hbox{$\macc@style{\macc@nucleus}_{}$}%
    \setbox\tw@\hbox{$\macc@style{\macc@nucleus}{}_{}$}%
    \dimen@\wd\tw@
    \advance\dimen@-\wd\z@
    \divide\dimen@ 3
    \@tempdima\wd\tw@
    \advance\@tempdima-\scriptspace
    \divide\@tempdima 10
    \advance\dimen@-\@tempdima
    \ifdim\dimen@>\z@ \dimen@0pt\fi
    \rel@kern{0.6}\kern-\dimen@
    \if#31
      \overline{\rel@kern{-0.6}\kern\dimen@\macc@nucleus\rel@kern{0.4}\kern\dimen@}%
      \advance\dimen@0.4\dimexpr\macc@kerna
      \let\final@kern#2%
      \ifdim\dimen@<\z@ \let\final@kern1\fi
      \if\final@kern1 \kern-\dimen@\fi
    \else
      \overline{\rel@kern{-0.6}\kern\dimen@#1}%
    \fi
  }%
  \macc@depth\@ne
  \let\math@bgroup\@empty \let\math@egroup\macc@set@skewchar
  \mathsurround\z@ \frozen@everymath{\mathgroup\macc@group\relax}%
  \macc@set@skewchar\relax
  \let\mathaccentV\macc@nested@a
  \if#31
    \macc@nested@a\relax111{#1}%
  \else
    \def\gobble@till@marker##1\endmarker{}%
    \futurelet\first@char\gobble@till@marker#1\endmarker
    \ifcat\noexpand\first@char A\else
      \def\first@char{}%
    \fi
    \macc@nested@a\relax111{\first@char}%
  \fi
  \endgroup
}
\makeatother
\endlocaldefs

\begin{document}
\allowdisplaybreaks

\begin{frontmatter}
%%%%%%%%%%%%%%%%%%%%%%%%%%%%%%%%%%%%%%%%%%%%%%
%%                                          %%
%% Enter the title of your article here     %%
%%                                          %%
%%%%%%%%%%%%%%%%%%%%%%%%%%%%%%%%%%%%%%%%%%%%%%
\title{Sharp Convergence Rates for Empirical  \\ Optimal Transport with Smooth Costs}
%\title{A sample article title with some additional note\thanksref{T1}}
\runtitle{Empirical  Optimal Transport with Smooth Costs}
%\thankstext{T1}{A sample of additional note to the title.}
 
\begin{aug}
%%%%%%%%%%%%%%%%%%%%%%%%%%%%%%%%%%%%%%%%%%%%%%%
%% Only one address is permitted per author. %%
%% Only division, organization and e-mail is %%
%% included in the address.                  %%
%% Additional information can be included in %%
%% the Acknowledgments section if necessary. %%
%%%%%%%%%%%%%%%%%%%%%%%%%%%%%%%%%%%%%%%%%%%%%%%
\author[A]{\fnms{Tudor} \snm{Manole}\ead[label=e1]{tmanole@andrew.cmu.edu}},
\and
\author[B]{\fnms{Jonathan} \snm{Niles-Weed}\ead[label=e2]{jnw@cims.nyu.edu}}
%%%%%%%%%%%%%%%%%%%%%%%%%%%%%%%%%%%%%%%%%%%%%%
%% Addresses                                %%
%%%%%%%%%%%%%%%%%%%%%%%%%%%%%%%%%%%%%%%%%%%%%%
\address[A]{Department of Statistics and Data Science, 
Carnegie Mellon University, \printead{e1}}
%\vspace{-0.25in}
\address[B]{Courant Institute of Mathematical Sciences, New York University, \printead{e2}}
\end{aug}

\begin{abstract}
We revisit the question of characterizing the convergence rate of plug-in estimators of optimal transport costs.
It is well known that an empirical measure comprising independent samples from an absolutely continuous distribution on $\bbR^d$ converges to that distribution at the rate $n^{-1/d}$ in Wasserstein distance, which can be used to prove that plug-in estimators of many optimal transport costs converge at this same rate.
However, we show that when the cost is smooth, this analysis is loose: plug-in estimators based on empirical measures converge quadratically faster, at the rate $n^{-2/d}$.
As a corollary, we show that the Wasserstein distance between two distributions 
is significantly easier to estimate when the measures are well-separated.
We also prove lower bounds, showing not only that our analysis of the plug-in estimator is tight, but also that no other estimator can enjoy significantly faster rates of convergence uniformly over all pairs of measures.
Our proofs rely on empirical process theory arguments based on tight 
control of $L^2$ covering numbers for locally Lipschitz and semi-concave functions.
As a byproduct of our proofs, we derive $L^\infty$ estimates on the displacement induced
by the optimal coupling between any two measures satisfying suitable concentration and anticoncentration conditions,
for a wide range of cost functions.  
\end{abstract}



\end{frontmatter}

\section{Introduction}
\label{sec:introduction}
Optimal transport costs have received a recent surge of interest in applied probability and statistics.
Arising from the classical optimal transport problem~\citep{villani2003}, 
this family of divergences measures the work required to couple two probability distributions
in terms of 
a cost function over the space upon which they are defined. This fact makes them a powerful tool
for comparing measures in a manner which is sensitive to the geometry of the underlying space, and has motivated their   use in areas such as computer vision~\citep{rubner2000}, generative modeling~\citep{arjovsky2017}, 
and  computational biology~\citep{orlova2016}, among many others. 
We refer the reader to the monographs of
\cite{panaretos2019}, \cite{santambrogio2015}, and \cite{peyre2019}
for surveys of their respective applications in  statistics,  applied mathematics, and machine learning.

In many of these applications, it is necessary to estimate the optimal transport cost between two measures
on the basis of independent observations.
This raises the fundamental question of characterizing the expected convergence rate of empirical estimators of these costs.
Though this question has been studied in great generality in the literature, the goal of this paper is to highlight some unexpected phenomena that arise when the cost function is smooth.

For concreteness, we focus throughout on the optimal transport problem
over the Euclidean space $\bbR^d$ for some integer $d \geq 1$. 
Let $\calX, \calY \subseteq \bbR^d$, and let $\calP(\calX)$ denote the set
of Borel probability measures with support contained in $\calX$. 
Let $\mu \in \calP(\calX)$ and $\nu \in \calP(\calY)$. 
Given a nonnegative cost function $c:\calX \times \calY \to \bbR$, the {\it optimal transport cost} 
based on $c$ is defined~by
$$\calT_c(\mu,\nu) = \inf_{\pi \in \Pi(\mu,\nu)} \int c(x,y)d\pi(x,y).$$
Here, $\Pi(\mu,\nu)$
denotes the set of {\it couplings}, that is, joint Borel probability measures over $\calX \times \calY$
with respective marginals $\mu$ and $\nu$.

In statistical contexts, the measures $\mu$ and $\nu$ are typically unknown, 
and it is necessary to estimate the optimal transport
cost between them on the basis of i.i.d. observations $X_1, \dots, X_n \sim \mu$ and $Y_1, \dots, Y_n \sim \nu$.
A canonical choice is the plug-in estimator $\calT_c(\mu_n,\nu_n)$, obtained by replacing $\mu$ and $\nu$ by their corresponding empirical measures:
$$\mu_n = \frac 1 n \sum_{i=1}^n \delta_{X_i}, \quad \nu_n = \frac 1 n \sum_{i=1}^n \delta_{Y_i}.$$
We call this quantity the {\it empirical optimal transport cost}, and we seek sharp upper and lower bounds on the expected gap between the empirical optimal transport cost and its population counterpart:
\begin{equation}
\label{eq:Delta_n_c}
\Delta_n(c) = \bbE \big| \calT_c(\mu_n,\nu_n) - \calT_c(\mu,\nu)\big|.
\end{equation}
We highlight the dependence of $\Delta_n$ on $c$ because a key finding of our work is that the rate of decay of $\Delta_n$ is driven by properties of the cost, and can improve significantly when $c$ is smooth.

To illustrate the phenomena we have in mind, we turn to perhaps the most widely-used cost functions: those of the form $c_p(x,y) = \norm{x-y}^p$, $p\geq 1$, where $\| \cdot \|$ denotes the Euclidean metric on $\bbR^d$.
These costs give rise to the $p$-Wasserstein distances,  defined by $W_p = \calT_{c_p}^{1/p}$.
The convergence rate of the empirical $p$-Wasserstein distance $W_p(\mu_n,\nu_n)$ to its population counterpart $W_p(\mu,\nu)$ 
is a well-studied problem; for instance,
assuming for simplicity of exposition that $\calX=\calY$ is a compact set,
\cite{fournier2015} prove 
that there exists a constant
$C_d > 0$, depending only on $d$ and $\calX$, such that
\begin{equation}
\label{eq:fournier_bound}
\bbE W_p(\mu_n,\mu)  \leq \big[\bbE W_p^p(\mu_n,\mu)\big]^{\frac 1 p}
 \leq C_d n^{-1/d}, 
\end{equation}
whenever $d > 2p$. Since $W_p$ is a metric, it follows that 
\begin{equation}
\label{eq:naive_Wpp_bound} 
 \bbE \big|W_p(\mu_n,\nu_n) - W_p(\mu,\nu)\big|
 \leq \bbE W_p(\mu_n, \mu) + \bbE W_p(\nu_n,\nu) \leq 2C_d n^{-1/d}.
 \end{equation}
The $n^{-1/d}$ rate in equation~\eqref{eq:naive_Wpp_bound} is well known 
to be inherent to statistical optimal transport problems.
In particular, it was shown by \cite{niles-weed2019} that,
up to polylogarithmic factors, no estimator of $W_p(\mu, \nu)$ 
 improves on the rate in equation~\eqref{eq:naive_Wpp_bound} uniformly over all pairs of measures $(\mu,\nu)$.
Nevertheless, one of the main contributions of this paper is to show that this bound
is only tight when $\mu = \nu$, and 
can otherwise be improved up to quadratically.
Indeed, our results imply the bound
\begin{equation}
\label{eq:p_transport_location}
\Delta_n(c_p) = \bbE \big| W_p^p(\mu_n,\nu_n) - W_p^p(\mu,\nu)\big|
 \lesssim \begin{cases}
   n^{-p/d}, & 1 \leq p \leq 2 \\
   n^{-2/d}, & 2 \leq p < \infty,
   \end{cases}
\end{equation}
which, as we shall see, entails
\begin{equation}
\label{eq:wasserstein_location}
\bbE \big| W_p(\mu_n,\nu_n) - W_p(\mu,\nu)\big|
 \lesssim \delta_0^{1-p}\begin{cases}
   n^{-p/d}, & 1 \leq p \leq 2 \\
   n^{-2/d}, & 2 \leq p < \infty
   \end{cases},\quad \text{if  }\   W_p(\mu,\nu) \geq \delta_0 > 0.
\end{equation}
Whenever $p > 1$, equations~\eqref{eq:p_transport_location} and~\eqref{eq:wasserstein_location}
provide a significant sharpening of the naive
estimate in equation~\eqref{eq:naive_Wpp_bound}. We show 
that such improvements arise due to the H\"older smoothness of the cost $c_p$, and in fact, 
similar rates of convergence for $\Delta_n(c)$ are enjoyed by a much broader
collection of smooth cost functions. 
Beyond smoothness assumptions on $c$, we establish our main results 
under the following broad structural condition, which is presumed throughout the sequel,
\begin{enumerate}[leftmargin=1.65cm,listparindent=-\leftmargin,label=(\textbf{H\arabic*}),start=0]   %[label=\textbf{(A1($\delta$)}]      %[\textbf{A1($\delta$)}]
\item  \label{assm:global} The cost function $c:\calX \times \calY \to \bbR$
is nonnegative, and takes the form $c(x,y) = h(x-y)$ where $h: \bbR^d \to \bbR_+$
is convex, even, and lower semi-continuous.
\end{enumerate}
Before summarizing our main results and comparing them to further existing literature, 
we begin with an idealized example which illustrates the role of these conditions.

\subsection{Example: Location Families}
\label{sec:example_location} 
For simplicity, we limit this example to upper bounding the following
one-sample analogue of $\Delta_n(c)$,
$$\bbE |\calT_c(\mu_n,\nu) - \calT_c(\mu,\nu)|.$$ 
We also continue to assume for simplicity that $\calX=\calY$ is a convex and compact set.
Let $c$ be any cost function such 
that condition~\ref{assm:global} holds, and assume
there exists $\alpha \in (1,2]$ such that 
$h \in \calC^\alpha(\calX)$. Here, $\calC^\alpha(\calX)$
denotes the H\"older space over $\calX$
with regularity $\alpha$, which is defined in Section~\ref{sec:background}
together with all other notational conventions used in the sequel.  

Let $\mu,\nu \in \calP(\calX)$ be any two measures differing 
only by a location transformation with respect to a fixed vector $z_0 \in \bbR^d$, 
in the sense that $\nu = {T_0}_\# \mu := \mu(T_0^{-1}(\cdot))$, where $T_0(z) = z+z_0$
and $\#$ denotes the pushforward operator. 
In this example, it is simple to find an optimal
coupling between $\mu$ and $\nu$. 
Indeed, recall that $h$ is convex and even under~\ref{assm:global},
thus for all $\pi\in \Pi(\mu,\nu)$, Jensen's inequality implies
$$\int h(x-y)d\pi(x,y) \geq h\left( \int xd\mu(x) - \int yd\nu(y)\right) = h(z_0).$$
Thus, $\calT_c(\mu,\nu) \geq h(z_0)$, and the lower bound is achieved by the coupling
$\pi=(Id, T_0)_\# \mu$, implying that $T_0$ is an optimal transport map from $\mu$ to $\nu$. 
On the other hand, for all couplings $\pi_n \in \Pi(\mu_n,\mu)$,
$\gamma_n = (Id, T_0)_\# \pi_n$
is a (typically suboptimal) coupling between $\mu_n$ and $\nu$, whence
\begin{align*}
\calT_c(\mu_n,\nu)
 \leq \int c(z,y) d\gamma_n(z,y)=\int c(z, T_0(x)) d\pi_n(z,x).
\end{align*}
Since we assumed that the H\"older norm $\Lambda :=  \norm h_{\calC^\alpha(\calX)}$  is finite, 
$h$ is close to its first-order 
Taylor expansion. Specifically, we obtain from the above display,
\begin{align}
\label{eq:location_shift_taylor}
%\nonumber 
\calT_c(\mu_n,\nu)
 &\leq  \int \Big[h(x-T_0(x)) +  \langle \nabla h(x- T_0(x)), z-x\rangle  
  +  \Lambda   \norm{z-x}^\alpha \Big] d\pi_n(z,x).
\end{align} 
Due to the marginal constraints in the definition of $\pi_n$, 
equation~\eqref{eq:location_shift_taylor} is tantamount to  
\begin{align}
%\label{eq:location_shift_step}
\calT_c(\mu_n,\nu)
 &\leq  \calT_c(\mu,\nu) + \int \langle \nabla h(z_0), \cdot\rangle d(\mu_n-\mu) 
  + \Lambda  \int \norm{x-z}^\alpha  d\pi_n(x,z).
\end{align}
The final term of the above display is manifestly the $\norm\cdot^\alpha$-transport cost between $\mu_n$ and $\mu$,
with respect to a possibly
suboptimal coupling $\pi_n \in \Pi(\mu_n,\mu)$. 
Since it holds for any choice of $\pi_n$, taking the infimum over such couplings 
leads to 
\begin{align}
\label{eq:example_location_bound}
\calT_c(\mu_n,\nu)
  &\leq \calT_c(\mu,\nu) + \int \langle \nabla h(z_0), \cdot\rangle d(\mu_n-\mu)
  + \Lambda W_\alpha^\alpha(\mu_n, \mu).
\end{align}
The second term on the right-hand side of the above display is  
a mean-zero sample average, and hence typically decays at the rate $n^{-1/2}$ in probability. 
Equation~\eqref{eq:example_location_bound} thus provides an upper bound on $\calT_c(\mu_n,\nu) - \calT_c(\mu,\nu)$
which is primarily driven by the rate of convergence of the empirical measure under the 
optimal transport  cost with respect to $\norm\cdot^\alpha$, which we refer to as 
the $\alpha$-transport cost in the sequel. 
By equation~\eqref{eq:fournier_bound}, we arrive at the following
one-sided estimate whenever $d \geq 5$,
\begin{equation}
\label{eq:one_sided_location}
\bbE\Big[ \calT_c(\mu_n,\nu) - \calT_c(\mu,\nu)\Big] \leq 
\Lambda  \bbE \big[W_\alpha^\alpha(\mu_n,\mu)\big]
\leq \Lambda C_d n^{-\alpha/d}.
\end{equation}
Although equation~\eqref{eq:one_sided_location} does not imply
an upper bound in expected absolute value, 
it captures the main features of our problem; as we shall see,
a simple extension of the above derivations leads to the bound
 \begin{equation}
\label{eq:two_sided_location}
\bbE \big|\calT_c(\mu_n,\nu) - \calT_c(\mu,\nu)\big| \leq C_0 n^{-\alpha/d}, \quad
\text{for all }  d \geq 5,
\end{equation}
for a large enough constant $C_0 > 0$ depending on $d,\calX$ and $\Lambda$.
Equation~\eqref{eq:two_sided_location} shows that, for the class of cost functions under consideration, 
the rate of convergence of the empirical optimal transport
cost is largely driven by the smoothness of $c$  when $\mu$
and $\nu$ differ merely in mean. In particular, notice that the cost $h(x) = \norm x^p$
satisfies $h \in \calC^p$ for all $p \geq 1$, 
which implies the previously announced result~\eqref{eq:p_transport_location} in the special
case of one-sample location families. 

This fast rate of convergence in equation~\eqref{eq:two_sided_location} arose in the present example because 
the first-order term in the Taylor expansion~\eqref{eq:location_shift_taylor} is negligible, leading
to a rate driven only by its remainder. While
this argument cannot easily be extended to general measures
$\mu$ and $\nu$,  
its conclusion turns out to be generic, as we now describe.

\subsection{Our Contributions}
\label{sec:contributions}
The primary contribution
of this paper is to provide sharp upper and lower bounds on $\Delta_n(c)$ 
for smooth costs satisfying condition~\ref{assm:global}.  
In this setting, our main result informally states that 
whenever $h \in \calC^\alpha$ for some $\alpha > 0$,   
\begin{equation}
\label{eq:informal_statement}
\Delta_n(c) \lesssim \begin{cases}
n^{-\alpha/d}, & 0 \leq \alpha \leq 2 \\
n^{-2/d}, & 2 \leq \alpha < \infty 
\end{cases},\qquad \text{for all } d \geq 5.
\end{equation}
This upper bound is stated formally in Theorem~\ref{thm:ub_general} under the assumption that $\mu$ and $\nu$
admit bounded support. Under additional conditions on $c$, we extend this result
to measures $\mu$ and $\nu$ with unbounded support, satisfying appropriate tail assumptions, 
in Theorem~\ref{thm:main_unbounded} and 
Corollary~\ref{cor:norm_p_unbounded}. As in equation~\eqref{eq:wasserstein_location}, 
our results have natural implications for the convergence rate of empirical Wasserstein distances, 
which we discuss in Corollary~\ref{cor:wasserstein}.  
In view of Section~\ref{sec:example_location},
the convergence rate~\eqref{eq:informal_statement} admits a natural
interpretation:
the first order term in a formal expansion of the empirical
optimal transport cost is typically negligible when $d \geq 5$, 
leading to a rate that improves with the smoothness parameter $\alpha \in (0,2)$.
When $\alpha \geq 2$, the quadratic term in this expansion is not negligible, 
thus faster rates do not occur without stronger conditions. 

At the heart of our proofs is the Kantorovich dual formulation of the optimal
transport problem---summarized in Section~\ref{sec:background}---which
allows us to reduce the problem
of bounding $\Delta_n(c)$ to that of bounding the expected suprema
of empirical processes indexed by collections of 
sufficiently regular Kantorovich potentials. 
While characterizing the regularity of these potentials is routine
when $\mu$ and $\nu$ are compactly supported~(\cite{gangbo1996}, Appendix C), 
the bulk of our efforts lies in the case where they admit unbounded support.
In this setting, one of our key technical contributions is to provide
quantitative $L^\infty$ estimates on the displacement induced by the optimal coupling between any two 
 measures satisfying appropriate tail 
 conditions (Theorem~\ref{thm:coupling_quantitative}).
For instance, the following is a special case of our result for the $p$-transport cost. 
\begin{theorem*}[Informal]
Let $\mu,\nu \in \calP(\bbR^d)$ and $p > 1$.  
Let $\nu$ be a $\sigma^2$-sub-Gaussian measure~\citep{boucheron2013} and $\mu$ have finite $p$-th moment, 
and assume there exist constants $c_1,c_2 > 0$ such that $\mu(B_{x,1}) \geq c_1 \exp(-c_2 \norm x^2)$ 
for all $x \in \bbR^d$. 
Then,  
for any optimal coupling $\pi$ between
$\mu$ and $\nu$ with respect to the cost $c_p(x,y)=\norm{x-y}^p$, 
\begin{equation}
\label{eq:informal_regularity}
%{\esssup_{\substack{y \in \bbR^d \\ (x,y) \in \supp(\pi)}}}  
\norm{y} \lesssim \sigma (\norm{x}+1),
\quad \text{ for  } \pi\text{-a.e. } (x,y).
\end{equation}
In particular, if there exists an optimal transport map $T$ from $\mu$ to $\nu$
with respect to $c_p$, then
$$\norm{T(x)} \lesssim \sigma (\norm{x}+1),
\quad \text{ for  } \mu\text{-a.e. } x.$$
\end{theorem*}
Analogues of equation~\eqref{eq:informal_regularity} have  previously been derived 
by~\cite{colombo2021} in the special case where $\mu$ is a Gaussian measure and $p=2$, and we 
further discuss these results below the statement of Theorem~\ref{thm:coupling_quantitative}. 
As we shall see, equation~\eqref{eq:informal_regularity} leads to estimates on the local Lipschitz constants of 
Kantorovich potentials between any two, possibly atomic probability measures, and forms the basis
of our main results when $\mu$ and $\nu$ have unbounded support. 
These results are quantitative analogues of the fact, proved by~\cite{gangbo1996}, that Kantorovich potentials
are locally Lipschitz under mild smoothness conditions on $c$.
  
In Section~\ref{sec:lower_bounds_empirical} we explicitly construct measures $\mu$ and $\nu$
for which  inequality~\eqref{eq:informal_statement} is achieved up to universal constants, 
inspired by the example in Section~\ref{sec:example_location}. 
While this result proves that our upper bounds 
cannot generally be improved, it does not preclude the 
possibility that there exists another estimator $\hat \calT_n$, i.e. a measurable function of $X_1, Y_1, \dots, X_n, Y_n$,
for which the quantity $\bbE|\hat \calT_n - \calT_c(\mu,\nu)|$ scales at a faster rate than
that of equation~\eqref{eq:informal_statement}, uniformly
over pairs of measures $\mu,\nu$. 
We prove in Section~\ref{sec:lower_bounds_minimax} that, in an information theoretic sense, such an improvement
is not possible up to polylogarithmic factors.

Though we prove inequality~\eqref{eq:informal_statement} for all $d \geq 5$, 
notice that it does not generally hold for all $d\geq 1$. Indeed, it is a simple observation
that the empirical optimal transport cost cannot generally 
achieve a faster rate of convergence than $n^{-1/2}$~\citep{niles-weed2019}.
	The probabilistic behavior of the empirical costs is therefore qualitatively different in low dimension.
While our proof techniques for bounded measures can be extended to the case $d \leq 4$,
they do not appear to yield  tight results for certain values of $\alpha > 0$; see
Remark~\ref{rem:low_dim} below. 
	Similarly, our techniques for unbounded measures do not generally appear to be tight in the low-dimensional case.
Since our goal in this paper is to obtain sharp convergence rates, we assume in what follows that $d \geq 5$, where we are able to establish exact results.

\paragraph*{Outline of the remainder of the paper}
In Sections~\ref{sec:past_work} and~\ref{sec:background}, we review prior work and recall some important preliminary results on the duality theory of transport costs.
Section~\ref{sec:upper_bounds_compact} contains our main results for compactly supported measures.
In Section~\ref{sec:upper_bounds_unbounded}, we extend these results to the unbounded case.
Lower bounds appear in Section~\ref{sec:lower_bounds}.
The proofs of certain intermediary results 
from Sections~\ref{sec:upper_bounds_compact}--\ref{sec:lower_bounds}
are respectively deferred to Appendices~\ref{app:proofs_compact}--\ref{app:proof_minimax}.  

\subsection{Related Work}
\label{sec:past_work}
Upper bounds on the expected deviation $\Delta_n(c)$ are available
in the literature for several special cases.
The closest to our setting is the quadratic cost $c_2(x,y) = \norm{x-y}^2$,
for which~\cite{chizat2020} prove that 
$\Delta_n(c_2) \lesssim n^{-2/d}$ when $\mu$ and $\nu$ are compactly supported.
Their proof hinges upon the Knott-Smith optimality criterion, 
which allows them to relate
$\Delta_n(c_2)$ to  suprema of empirical processes indexed by convex potentials,
which are in fact globally Lipschitz since $\mu$ and $\nu$ are assumed compact.
Empirical processes indexed by globally Lipschitz convex functions are well-studied~\citep{bronshtein1976, guntuboyina2012b},
and lead to their result.
Our results extend theirs in two directions: we replace $c_2$ by any smooth cost, and we remove the condition that the measures be compactly supported.
When $\mu$ and $\nu$ are compactly supported and $\alpha=2$, 
our proof strategy mirrors that of \cite{chizat2020}: though the potentials arising for other costs are not necessarily convex, it is still possible to use existing empirical process theory bounds to obtain sharp rates.
On the other hand, when $\mu$ and $\nu$ have unbounded support, the relevant potentials may not even be globally Lipschitz, and our proof requires significant new techniques.

Faster rates of convergence for estimating optimal transport costs
are achievable under strong conditions on $\calX$ and $\calY$. 
For instance, when $c$ is a metric
raised to a power $p \geq 1$, 
the bound $\Delta_n(c) \lesssim n^{-1/2}$ is known to hold when $\calX$ and $\calY$ are 
one-dimensional~\citep{munk1998, freitag2005, bobkov2019, delbarrio2019c, manole2019} or countable~\citep{sommerfeld2018, tameling2019}.
In both of these cases,
the corresponding empirical $p$-Wasserstein distance is known to exhibit distinct convergence rates depending on whether $\mu$ and $\nu$
are vanishingly close or not, similar to our findings in equation~\eqref{eq:wasserstein_location}.
While these two examples form 
important special cases, their underlying proof techniques are closely tied to 
characterizations of the optimal transport problem which are only available for discrete and one-dimensional measures, 
and do not shed more general light on the behaviour of $\calT_c(\mu_n,\nu_n)$.

Though the naive bound in equation~\eqref{eq:naive_Wpp_bound} 
is loose for $ p > 1$ when $W_p(\mu,\nu)$ is bounded away from zero, 
\cite{liang2019} and \cite{niles-weed2019} show that it cannot 
generally be improved  by more than a polylogarithmic factor when no separation conditions are placed on $\mu$ and $\nu$.
Recall that this  upper bound arose from the convergence rate of $\mu_n$ under the $p$-Wasserstein distance in equation~\eqref{eq:fournier_bound}. 
The study of such convergence rates was initiated by \cite{dudley1969} in the special case $p=1$, 
who also used arguments from empirical process theory, due to the dual characterization of 
$W_1$ as a supremum over Lipschitz functions~\citep{villani2003}.
For $p > 1$, distinct techniques have been used to study this problem 
in great generality by \cite{boissard2014a, fournier2015, bobkov2019, weed2019, singh2019, lei2020}, 
and references therein. \cite{dudley1969} also derived deterministic lower bounds on the quality of approximating
$\mu$ by any discrete measure supported on $n$ points under $W_1$---we build upon these results to obtain our lower bounds
on $\Delta_n(c)$ in Section~\ref{sec:lower_bounds_empirical}.

Another line of work has sought to understand optimal rates of estimation for Wasserstein distances when the {\it densities}---rather than the cost---are smooth.
These works~\citep{Lia17,singh2018,WeeBer19} show that the plug-in empirical estimator $W_p(\mu_n, \nu_n)$ for $W_p(\mu, \nu)$ is suboptimal if $\mu$ and $\nu$ have smooth densities, but that replacing $\mu_n$ and $\nu_n$ by appropriate nonparametric density estimators suffices to obtain optimal rates of estimation.
Under similar conditions on $\mu$ and $\nu$, it is also possible to construct appropriate smooth estimators of the optimal map between $\mu$ and $\nu$~\citep{hutter2021}.
Our work takes a quite different perspective: rather than adding additional conditions on~$\mu$ and~$\nu$, we show that the rates of convergence of empirical estimators improve under additional smoothness conditions on the cost.

\subsection{Notation and Further Background on the Optimal Transport Problem}
\label{sec:background}
Our proofs make repeated use of the  Kantorovich dual formulation 
of the optimal transport problem (\cite{villani2008}, Theorem 5.10), which we now describe. 
Define for all $\varphi \in L^1(\mu)$
and $\psi \in L^1(\nu)$ the functional
$$J_{\mu,\nu}(\varphi,\psi) = \int \varphi d\mu + \int \psi d\nu.$$
The regularity condition~\ref{assm:global} is sufficient to imply 
\begin{equation}
\label{eq:kantorovich}
\calT_c(\mu,\nu) = \sup_{(\varphi,\psi) \in \Phi_c(\mu,\nu)} J_{\mu,\nu}(\varphi,\psi),
\end{equation}
where $\Phi_c(\mu,\nu)$ denotes the set of pairs $(\varphi,\psi) \in L^1(\mu) \times L^1(\nu)$ such that
$\varphi(x) + \psi(y) \leq c(x,y)$ for all $x \in \calX$ and
$y \in \calY.$
If we further assume $c(x,y) \leq c_1(x) + c_2(y)$	
for some $c_1 \in L^1(\mu)$ and $c_2\in L^1(\nu)$, then the supremum in equation~\eqref{eq:kantorovich}
is achieved. Any pair $(\varphi,\psi)$ achieving the supremum is called a pair of {\it optimal (Kantorovich) potentials}. 

We shall say that a function $g: \bbR^d \to \widebar \bbR$ taking values in the extended real
line $\widebar \bbR = \bbR \cup \{-\infty\}$ 
is {\it $c$-concave}~\citep{gangbo1996} if it is not identically $-\infty$, and if there exists a nonempty
set $\calA \subseteq \bbR^d \times \bbR$
such that
$$g(y) = \inf_{(x,\lambda)\in \calA} \big\{ c(x,y) - \lambda \big\}.$$
The canonical example of $c$-concave functions are {\it $c$-conjugates}, for which $\calA$
is the graph of a map $f$. Specifically,
given $f:\calX \to \widebar \bbR$ not identically $-\infty$, the $c$-conjugate of $f$ is given by 
\begin{equation}
\label{eq:c_conj}
f^c:\calY \to \widebar\bbR, \quad f^c(y) = \inf_{x \in \calX} \Big\{c(x,y) - f(x)\Big\}.
\end{equation}
Whenever a map $f$ is merely defined over a nonempty subset $\calX_0 \subseteq \calX$ in the sequel, 
we extend its definition to $\calX$ by setting $f(\calX\setminus\calX_0) = \{-\infty\}$. In this case, 
it is clear that the infimum in the above display can be restricted to $\calX_0$. 

When the supremum in the Kantorovich duality~\eqref{eq:kantorovich} is achieved by a pair $(\varphi,\psi) \in \Phi_c(\mu,\nu)$,
it is easy to see that $(\varphi,\varphi^c)$ also
lies in $\Phi_c(\mu,\nu)$ and can only increase the value of $J_{\mu,\nu}(\varphi,\psi)$.
% $J_{\mu,\nu}(\varphi,\varphi^c) \geq J_{\mu,\nu}(\varphi,\psi)$ while $(\varphi,\varphi^c) \in \Phi_c(\mu,\nu)$.
It follows that $(\varphi,\varphi^c)$ is itself a pair of optimal Kantorovich potentials, and 
one obtains~\citep{villani2008},
\begin{equation}
\label{eq:kantorovich_c_conj}
\calT_c(\mu,\nu) = \sup_{\varphi\in L^1(\mu)} \int \varphi d\mu + \int \varphi^c d\nu.
\end{equation}
Notice that if $c(x,y) = -x^\top y$, then the definition~\eqref{eq:c_conj} of $c$-conjugate
reduces to $-(-f)^*$, where for any convex function $h$ on $\bbR^d$, 
$h^*(y) = \sup_{x \in \bbR^d} \{ x^\top y - h(x)\}$ denotes its Legendre-Fenchel transform. 
It is well known that if $h$ is also lower semi-continuous and not identically infinite,  
the supremum in the definition of $h^*(y)$ is achieved
by a point in its subdifferential; specifically, one has the relation
$$y \in \partial h(x) ~~ \Longleftrightarrow ~~ x \in \partial h^*(y) ~~ \Longleftrightarrow ~~ x^\top y = h(x) + h^*(y).$$
To derive analogous notions for $c$-concave functions, define the {\it $c$-superdifferential}
of a $c$-concave function $f:\calX \to\widebar\bbR$ by
$$\partial^c f = \left\{(x,y) \in \calX \times \calY: c(v,y) - f(v) \geq c(x,y) - f(x), \ \forall v \in \calX\right\}.$$
Furthermore, let $\partial^c f(x) = \{y \in \bbR^d: (x,y) \in \partial^c f\}$ and $\partial^c f(B) = \bigcup_{x \in B} \partial^c f(x)$,
for all $B \subseteq \calX$. The following Lemma summarizes the main properties of $c$-concave functions 
which we shall require.
Some of the statements which follow are weaker than necessary, but 
sufficient for our purposes.
\begin{lemma}
\label{lem:kantorovich_background}
Let $f: \calX \to \widebar\bbR$ be $c$-concave, and assume condition~\ref{assm:global}.
\begin{enumerate}
\item[(i)] (\cite{villani2008}, Proposition 5.8) We have, $f^{cc} = f$. 
\item[(ii)] (\cite{villani2003}, Remark 1.13) 
Assume the cost $c$ is bounded. Then, the supremum
in equation~\eqref{eq:kantorovich_c_conj} is achieved
by a $c$-concave function $\varphi \in L^1(\mu)$ such that 
$0 \leq \varphi \leq \norm{c}_\infty$ and $-\norm c_\infty \leq \varphi^c \leq 0.$ 
\item[(iii)] (\cite{gangbo1996}, Theorem 2.7) $\partial^c f$ is $c$-cyclically monotone, in the sense that
for any permutation $\sigma$ on $k\geq 1$ letters and any $(x_1, y_1), \dots, (x_k, y_k) \in  \partial^c f$,
%$$c(x,y) + c(x',y') \leq c(x,y') + c(x',y),\quad \text{ for all } (x,y),(x',y') \in \partial^c f.$$
$$\sum_{j=1}^k c(x_j, y_j) \leq \sum_{j=1}^k c(x_{\sigma(j)}, y_j).$$
\item[(iv)] (\cite{gangbo1996}, Proposition C.4) Assume further that $h$ is superlinear. 
For any given $x \in \bbR^d$, assume there exists a neighborhood
of $x$ over which $f$ is bounded. Then, the $c$-superdifferential $\partial^c f(x)$ is nonempty. Therefore, if $f$ is locally bounded over $\bbR^d$, 
it holds that for all $x,y \in \bbR^d$,
$$y \in \partial^c f(x) ~~ \Longleftrightarrow ~~ x \in \partial^c f^c(y) ~~ \Longleftrightarrow ~~ c(x,y) = f(x) + f^c(y).$$
In particular,  
$$f(x) =\inf_{y \in \partial^c f(x)} \big\{ c(x,y) - f^c(y)\big\},\quad
  f^c(y) =\inf_{x \in \partial^c f^c(y)} \big\{ c(x,y) - f(x)\big\}.$$% =  \inf_{x \in \partial f^c(y)} \big\{ c(x,y) - f(x)\big\}.$$
Furthermore, if $f$ is in fact an optimal Kantorovich potential for the optimal transport 
problem~\eqref{eq:kantorovich_c_conj} from $\mu$ to $\nu$, 
and if $\calT_c(\mu,\nu) < \infty$, 
then for any optimal coupling $\pi \in \Pi(\mu,\nu)$, $\supp(\pi) \subseteq \partial^c f$.
\end{enumerate}
\end{lemma}

We close this section with a summary of notational conventions used in the sequel.

{\bf Notation.} Given $a,b \in \bbR$, we write $a\vee b = \max\{a,b\}$ and $a\wedge b = \min\{a,b\}$. 
Given a   set $\Omega\subseteq \bbR^d$ and a function $f:\Omega \to \bbR$ which is differentiable to order $k \geq 1$, 
and a multi-index $\beta \in \bbN_0^d$, we write $|\beta| = \sum_{i=1}^d \beta_i$, 
and for all $|\beta|\leq k$, %\com{Reviewer's Minor Comment \#6}
$D^\beta f = \partial^{|\beta|} f/\partial x_1^{\beta_1}\dots \partial x_d^{\beta_d}.$
Given $\alpha > 0$, let $\underline{\alpha}$ denote the largest integer strictly less
than $\alpha$ (for instance, $\underline{\alpha} = 0$ when $\alpha=1$). 
The H\"older space $\calC^\alpha(\Omega)$ is defined as the set
of functions $f:\Omega \to \bbR$ admitting at least $\underline{\alpha}$
continuous derivatives over the interior $\Omega^\circ$ of $\Omega$, which extend continuously  
up to the boundary of $\Omega$, and are such that the H\"older norm 
$$\norm f_{\calC^\alpha(\Omega)} = \sum_{j=0}^{\underline{\alpha}} \sup_{|\beta| = j} \|D^\beta f\|_{L^\infty(\Omega^\circ)}
 + \sum_{|\beta|=\underline{\alpha}} \sup_{\substack{x,y\in\Omega^\circ \\ x\neq y}} \frac{|D^\beta f(x) - D^\beta f(y)|}{\norm{x-y}^{\alpha-\underline{\alpha}}}$$
is finite. Notice that when $k \geq 1$ is an integer, $\calC^k(\Omega)$
is also commonly denoted $\calC^{k-1,1}(\Omega)$.  
Given a measure space $(\Omega, \calF, \nu)$, we interchangeably use the symbols
$\norm f_{L^p(\Omega)}$ and $\norm f_{L^p(\nu)}$  
to denote the $L^p$ norm $(\int_{\Omega} |f(x)|^p d\nu(x))^{1/ p}$ for any $1 \leq p \leq \infty$
and any Borel-measurable function $f:\Omega\to\bbR$. 
We drop the suffix $\Omega$ and simply
write $\calC^\alpha$ and $L^p$ when $\Omega$ is clear from context. 
Given a map $T:\Omega \to \Omega$,  the pushforward of $\nu$
under $T$ is denoted by $T_\# \nu = \nu(T^{-1}(\cdot))$. 
Given a vector $x \in \bbR^d$ and $1 \leq p \leq \infty$, 
we write $\norm x_{\ell_p} = (\sum_{i=1}^d |x_i|^p)^{1/ p}$.
When $p=2$, we drop the subscript $\ell_2$ and simply write $\norm\cdot$.
Given a square matrix $A=(a_{ij}) \in \bbR^{d  \times d}$, 
%$\tr(A)$ denotes the trace of $A$, 
$\norm A_{\mathrm{op}} = \sup\{\norm{Ax}: x \in \bbR^d, \norm x=1\}$
denotes its operator norm, and $\norm A_\infty = \max_{1 \leq i,j \leq d} |a_{ij}|$ its entrywise $\ell_\infty$ norm.
The convolution of two functions $f,g:\bbR^d \to \bbR$ is denoted $(f\star g)(x) = \int f(y) g(x-y) dy$,
and that of $f$ and a Borel probability measure $\mu$ on $\bbR^d$ is given
by $(\mu\star f)(x) = \int f(x-y)d\mu(y)$, for all $x \in \bbR^d$.
The closed ball centered at $x \in \bbR^d$ of radius $r \geq 0$ is denoted
$B_{x, r} = \{x \in \bbR^d: \norm x \leq r\}$, and we drop the subscript $r$ when it equals 1: $B_x = B_{x,1}$. 
$\calL$ denotes 	the Lebesgue measure on $\bbR^d$. 

A {\it universal} constant is any constant $C > 0$
which is tacitly permitted to depend on $d, \calX, \calY$, and 
may also depend on additional data when specified. 
%Universal constants will never depend
%on the measures $\mu$ and $\nu$, however. 
Given sequences of nonnegative real numbers 
$(a_n)_{n=1}^\infty$ and $(b_n)_{n=1}^\infty$, we write $a_n \lesssim b_n$ if there exists a universal
constant $C > 0$ such that $a_n \leq C b_n$ for all $n \geq 1$. We also write $a_n \asymp b_n$ if $b_n \lesssim a_n \lesssim b_n$. 

Finally, as discussed in Section~\ref{sec:contributions}, we assume $d \geq 5$
throughout the remainder of this manuscript, unless otherwise stated.

\section{Upper Bounds for Compactly Supported Measures}
\label{sec:upper_bounds_compact}
We begin by bounding the rate of convergence of the empirical optimal transport
cost in the special case where $\calX$, $\calY$
and $\calZ = \calX - \calY  = \{x-y: x\in \calX, y \in \calY\}$ satisfy the following condition. 
\begin{enumerate}[leftmargin=1.65cm,listparindent=-\leftmargin,label=(\textbf{S\arabic*})]   %[label=\textbf{(A1($\delta$)}]      %[\textbf{A1($\delta$)}]
\item \label{assm:S1} 
$\calX, \calY \subseteq \bbR^d$ are convex and compact sets with nonempty interior. 
Furthermore, we have $\calX, \calY,\calZ % = \{x-y: x\in \calX, y \in \calY\} 
\subseteq B_{0,1}$. 
\end{enumerate}
The assumption of compactness of $\calX$ and $\calY$ will be relaxed
in the following section, under concentration and anticoncentration conditions on the measures.
Once $\calX$ and $\calY$ are assumed compact, notice that the final assumption of~\ref{assm:S1} can always be satisfied 
up to rescaling and recentering. Furthermore, we recall that the supports of $\mu$ and $\nu$
are merely assumed to be contained in $\calX$ and $\calY$, and thus need not be convex themselves. 

We shall also assume throughout this section that the cost $c$
satisfies condition~\ref{assm:global} and the following
smoothness condition.
\begin{enumerate}[leftmargin=1.65cm,listparindent=-\leftmargin,label=(\textbf{H\arabic*}),start=1] 
\item \label{assm:bded_smoothness} 
There exists $\alpha \in (0,2]$ and a convex open set $\calZ_1$ such that $\calZ \subseteq \calZ_1 \subseteq B_{0,2}$,
and $h \in \calC^\alpha(\calZ_1)$. 
Furthermore, we have $0 \leq h \leq 1$ on $\calZ_1$.
We write $\Lambda := 1 \vee \norm h_{\calC^\alpha(\calZ_1)}<\infty$. 
\end{enumerate} 
For any measures $\mu \in \calP(\calX)$ and 
$\nu \in \calP(\calY)$, recall that
$X_1, \dots, X_n \sim \mu$ and $Y_1, \dots, Y_n \sim \nu$ denote i.i.d.
samples,  with corresponding empirical measures
$\mu_n = \frac 1 n \sum_{i=1}^n \delta_{X_i}$ and 
$\nu_n = \frac 1 n \sum_{i=1}^n \delta_{Y_i}$.
The main result of this section is now stated as follows.
\begin{theorem} 
\label{thm:ub_general}
%Let $d \geq 5$. 
Assume conditions \ref{assm:S1}, \ref{assm:global}, and \ref{assm:bded_smoothness}. Then, 
there exists a constant $C > 0$ depending only on $d, \alpha,\calX, \calY, \calZ_1$ such that
$$\sup_{\substack{\mu\in\calP(\calX) \\ \nu \in \calP(\calY)}} \bbE_{\mu,\nu}\big| \calT_c(\mu_n,\nu_n) - \calT_c(\mu,\nu)\big| \leq
C\Lambda n^{-\alpha/d}.$$
\end{theorem}
Theorem~\ref{thm:ub_general} proves that the convergence rates anticipated in Section~\ref{sec:example_location},
for measures differing only in mean, 
in fact hold for all compactly supported measures.
In particular, the $n^{-\alpha/d}$ rate of convergence is achievable
as soon as $h \in \calC^\alpha$, for $\alpha \in (0,2]$, though is not generally claimed to improve further when $\alpha > 2$. 
For instance, the quadratic cost $\norm\cdot^2$ lies in $\calC^\infty$, 
but one cannot hope for a faster convergence rate than $n^{-2/d}$. Indeed, 
we derive matching lower bounds in Section~\ref{sec:lower_bounds} under closely related assumptions on the cost function $c$,
which imply that the upper bound of Theorem~\ref{thm:ub_general} is generally unimprovable.

A careful investigation of our proof reveals that Theorem~\ref{thm:ub_general}
in fact continues to hold for nonconvex costs $h$. We nevertheless
prefer to retain the assumption of convexity in condition~\ref{assm:global}
since it is required for the remainder of our main results; in particular, 
we do not claim that the convergence rate in Theorem~\ref{thm:ub_general}
is sharp when $h$ is not convex.

By letting $\alpha$ vanish, Theorem~\ref{thm:ub_general} suggests that 
the empirical optimal transport cost does not generally converge at any polynomial
rate for cost functions which fail to be uniformly H\"older continuous. 
Indeed, absent any smoothness assumptions on $c$, 
$\calT_c(\mu_n,\nu_n)$ may not even
converge in $L^1(\bbP)$, as can be seen by taking $c$ to be
the Hamming metric. In this case, 
 $\calT_c(\mu_n,\nu_n)$ is simply the Total Variation 
distance between $\mu_n$ and $\nu_n$, which almost surely
equals unity when $\mu$ and $\nu$ are absolutely continuous with 
respect to the Lebesgue~measure.

As discussed in Section~\ref{sec:introduction},
perhaps the most widely-used cost functions satisfying conditions~\ref{assm:global} and \ref{assm:bded_smoothness}
are norms over $\bbR^d$ raised to a power greater than one. 
We illustrate the conclusion of Theorem~\ref{thm:ub_general}
for such an example. 
\begin{corollary} [Powers of $\ell_r$ Norms]
\label{cor:lp_norms} 
Let $\calX, \calY$ satisfy condition~\ref{assm:S1}, and define the cost 
$c_{p,r}(x,y) = \norm{x-y}_{\ell_r}^p$ for all $x \in \calX, y \in \calY$
and $p,r \geq 1$. Let $\mu \in \calP(\calX)$, $\nu \in \calP(\calY)$.
\begin{enumerate} 
\item[(i)] We have for all $p,r \geq 1$, $\bbE \big| \calT_{c_{p,r}}(\mu_n,\nu_n) - \calT_{c_{p,r}}(\mu,\nu)| \lesssim n^{-(2\wedge p \wedge r)/d}$. 
In particular, specializing to  $r=2$, 
$$\bbE \big| W_p^p(\mu_n,\nu_n) -W_p^p(\mu,\nu)\big| \lesssim 
\begin{cases}
n^{-p/d}, & 1 \leq p < 2 \\
n^{-2/d}, & 2 \leq p < \infty.
\end{cases}$$
\item[(ii)] If $\calX,\calY \subseteq \bbR^d$ are disjoint, then for all $p \geq 1$ and $r \geq 2$,
\begin{equation*} 
\bbE \big| \calT_{c_{p,r}}(\mu_n,\nu_n) - \calT_{c_{p,r}}(\mu,\nu)\big| \lesssim n^{-2/d}.
\end{equation*}
\end{enumerate}
\end{corollary} 
The proof is deferred to Appendix~\ref{app:lr_norms}. 
Corollary~\ref{cor:lp_norms}(i) follows from the fact that $\norm\cdot^p_{\ell_r} \in \calC^{2\wedge p\wedge r}(\calZ_1)$ 
for any bounded open set $\calZ_1$, for all $p,r \geq 1$. 
When $r \geq 2$, notice that $\norm\cdot_{\ell_r}$ is smooth away
from the origin, so that the condition $\norm\cdot_{\ell_r}^p \in \calC^2(\calZ_1)$
can be satisfied for any $p \geq 1$
whenever the closed set $\calZ \subseteq \calZ_1$ 
does not contain the point zero. This observation leads to Corollary~\ref{cor:lp_norms}(ii).
This last point implies the rather surprising fact that for all measures $\mu$ and~$\nu$ admitting
disjoint and compact support, one has
\begin{equation}
\bbE \big| W_1(\mu_n,\nu_n) - W_1(\mu,\nu) \big| \lesssim n^{-2/d}.
\end{equation}
When the measures $\mu$ and $\nu$ are not vanishingly close, Corollary~\ref{cor:lp_norms} also translates
into convergence rates for empirical Wasserstein distances.
\begin{corollary}[Wasserstein Distances]
\label{cor:wasserstein}
Let $p \geq 1$. Let $\calX, \calY$ satisfy condition~\ref{assm:S1}, and let $\mu \in \calP(\calX)$, $\nu \in \calP(\calY)$.
Assume $W_p(\mu,\nu) \geq \delta_0$, for some constant $\delta_0 > 0$. Then,
\begin{equation}
\label{eq:cor_wasserstein}
\bbE \big| W_p(\mu_n,\nu_n) - W_p(\mu,\nu)\big| \lesssim \delta_0^{1-p} \begin{cases}
n^{-p/d}, & 1 \leq p < 2 \\
n^{-2/d}, & 2 \leq p < \infty.
\end{cases}
\end{equation}
 \end{corollary}
\begin{proof}
By the numerical inequality $|x-y| {\leq} y^{1-p}|x^p - y^p| $ for all $x,y \geq 0$, $p \geq 1$, one~has
$$\bbE |W_p(\mu_n,\nu_n) - W_p(\mu,\nu)| % \leq W_p^{1-p}(\mu,\nu) \bbE \big| W_p^p(\mu_n,\nu_n) - W_p^p(\mu,\nu)\big| 
\leq \delta_0^{1-p} \bbE \big| W_p^p(\mu_n,\nu_n) - W_p^p(\mu,\nu)\big|.$$
The claim thus follows from Corollary~\ref{cor:lp_norms}.
\end{proof}

\subsection{Proof of Theorem~\ref{thm:ub_general}}  
\label{sec:pf_thm_compact} 
We divide our argument into three cases.

\subsubsection{Case 1: $\alpha = 2$}
Under conditions~\ref{assm:S1},~\ref{assm:global} and~\ref{assm:bded_smoothness},
it follows from Lemma~\ref{lem:kantorovich_background}(ii)
that there exist Kantorovich potentials $\varphi_n:\calX \to \bbR$ and $\psi_n: \calY \to \bbR$
%and $g_n:\supp(\nu_n) \to \bbR$ 
such that 
% 
%Under condition~\ref{assm:global},
%there exists a pair $(f_{n},g_n)\in \Phi_c(\mu_n,\nu_n)$  of optimal Kantorovich potentials satisfying
$\calT_c(\mu_n,\nu_n) = J_{\mu_n,\nu_n}(\varphi_n, \psi_n)$ and $|\varphi_n|, |\psi_n| \leq 1$.
Furthermore, since $\mu$ and $\nu$ are compactly supported, it follows
immediately from the definition of the $c$-conjugate that
$(\varphi_n,\psi_n) \in \Phi_c(\mu,\nu)$, whence
\begin{align} 
\label{eq:pf_bded_supp_step}
\nonumber 
\calT_c(\mu,\nu)
 &= \sup_{(\varphi, \psi)\in\Phi_c(\mu,\nu)} J_{\mu,\nu}(\varphi, \psi)  \\
 &\geq  J_{\mu,\nu}(\varphi_n, \psi_n) = J_{\mu_n,\nu_n}(\varphi_n, \psi_n)
  + \int \varphi_nd(\mu-\mu_n) + \int \psi_nd(\nu-\nu_n).
\end{align}
On the other hand, recalling that $\calT_c(\mu_n, \nu_n) = J_{\mu_n,\nu_n}(\varphi_n,\psi_n)$, we derive
\begin{equation}
\label{eq:pf_main_emp_proc_red}
\calT_c(\mu_n,\nu_n) - \calT_c(\mu,\nu) \leq \int \varphi_nd (\mu_n-\mu) + \int \psi_n d(\nu_n-\nu).
\end{equation}
Our goal is now to bound the empirical processes arising on the right-hand side of the above display.
Due to the compactness of $\calX$ and $\calY$, it can be deduced from \cite{gangbo1996} that
$\varphi_n$ and $\psi_n$ are  Lipschitz and semi-concave. 
The following Lemma is an analogue of their results with explicit constants, 
whose proof is included in Appendix~\ref{app:semi_convex_compact} for completeness.
\begin{lemma}
\label{lem:semi_concave}
Assume conditions \ref{assm:S1} and \ref{assm:global}, and that
condition~\ref{assm:bded_smoothness} holds with $\alpha = 2$.
Then the maps
$$\tilde\varphi_n: x \in \calX \longmapsto  \varphi_n(x) - \frac{\Lambda}{2} \norm x^2 ,\quad
  \tilde\psi_n: y \in \calY \longmapsto \psi_n(y) - \frac{\Lambda}{2} \norm y^2$$
are concave and $(2\Lambda)$-Lipschitz.
Furthermore,  $|\tilde\varphi_n|, |\tilde\psi_n| \leq 2 \Lambda$. 
\end{lemma} 
For any $L,U > 0$, let $\calF_{L,U}(K)$ denote the set of $L$-Lipschitz convex functions $f : K \to \bbR$ over a convex set $K \subseteq \bbR^d$, such that 
$|f| \leq U$. Recalling the convexity of $\calX$ and $\calY$ under condition~\ref{assm:S1}, define 
$$\Delta_n = \sup_{f \in \calF_{1,1}(\calX)} \int f d(\mu_n-\mu) + \sup_{g \in\calF_{1,1}(\calY)} \int g d(\nu_n-\nu).$$
By Lemma~\ref{lem:semi_concave}, we have $(-\tilde\varphi_n/2\Lambda) \in \calF_{1,1}(\calX)$
and $(-\tilde\psi_n/2\Lambda) \in \calF_{1,1}(\calY)$,
thus together with equation~\eqref{eq:pf_main_emp_proc_red}  we obtain
\begin{align} 
\calT_c(\mu_n,\nu_n) - \calT_c(\mu,\nu) \leq 
  2\Lambda \Delta_n + \frac{\Lambda}{2}\int  \norm \cdot^2 d\big((\mu_n-\mu) + (\nu_n-\nu)\big).
 \label{eq:pf2_step1}
\end{align}  
On the other hand, lower bounds on $\calT_c(\mu_n,\nu_n) - \calT_c(\mu,\nu)$ are simple to obtain. 
As before, there exists a pair of optimal Kantorovich potentials 
$(\varphi_0, \psi_0) \in \Phi_c(\mu, \nu)$ such that $|\varphi_0| \vee |\psi_0| \leq 1$ and 
$\calT_c(\mu,\nu) = J_{\mu,\nu}(\varphi_0,\psi_0)$. Therefore, 
\begin{align*}
\calT_c(\mu_n,\nu_n)
%  = \sup_{(\varphi,\psi)\in\Phi_c(\mu_n,\nu_n)} J_{\mu_n,\nu_n}(\varphi,\psi)
  \geq J_{\mu_n,\nu_n}(\varphi_0,\psi_0)
 % =  \int \varphi_0 d\mu_n + \int \psi_0  d\nu 
  = \calT_c(\mu,\nu)  + \int \varphi_0 d(\mu_n-\mu) + \int \psi_0 d(\nu_n-\nu).
  \end{align*}
Combining the previous two displays, we deduce
\begin{align}
\nonumber
\bbE  \big| \calT_c(\mu_n,\nu_n)  - \calT_c(\mu,\nu) \big| 
\nonumber &\leq 2\Lambda\bbE[\Delta_n] +
      \frac \Lambda 2\bbE\left|\int \norm \cdot^2 d\big((\mu_n-\mu) + (\nu_n-\nu)\big) \right|  \\ &
     {+}   \bbE\left|\int  \varphi_0 d(\mu_n-\mu)\right| 
      {+}   \bbE\left|\int  \psi_0 d(\nu_n-\nu)\right| 
\lesssim\Lambda \Big( \bbE[\Delta_n] + n^{-1/2}\Big),
 \label{eq:emp_proc_red}
\end{align}
where the final bound is a straightforward consequence of Chebyshev's inequality,
and where we have used the fact that $|\varphi_0|,|\psi_0| \leq 1$. It thus remains to bound
$\bbE[\Delta_n]$. This last is a sum of expected suprema of  empirical processes indexed by
convex Lipschitz functions, upper bounds for which can be obtained
via Dudley's chaining technique~\citep{dudley2014} in terms of the metric entropy of the class $\calF_{L,U}(K)$. 
Specifically, recall that for all $\epsilon > 0$, the $\epsilon$-metric entropy of a set $A$ contained in a metric space $(\calX, \eta)$
is  the logarithm of the $\epsilon$-covering number $N(\epsilon, A,\eta)$ of $(A, \eta)$, defined by
$$N(\epsilon, A,\eta) = \inf\big\{ N\geq 1: \exists \{x_1, \dots, x_N\} \subseteq \calX, \forall x \in A, \exists 1 \leq i \leq N: \eta(x,x_i) \leq \epsilon\big\}.$$
The following version of Dudley's bound will be sufficient for our purposes, and can be deduced for instance from Lemma~16 of~\cite{vonluxburg2004} (see also Lemma~3.2 of~\cite{vandegeer2000}). 
\begin{lemma}[\cite{vonluxburg2004}] 
\label{lem:chaining}
Let $\calG$ be a set of real-valued measurable functions on $\bbR^d$. Then,  
\begin{align}
\label{eq:chaining}
\bbE \left[ \sup_{g \in \calG} \int g d(\mu_n-\mu)\right]
 &\lesssim \bbE\left[\inf_{\tau > 0}\left\{\tau  + \frac 1 {\sqrt n}\int_{\tau}^\infty \sqrt{\log N(\epsilon, \calG, L^2(\mu_n))} d\epsilon\right\}\right].
 \end{align} 
\end{lemma} 
Tight bounds on the metric entropy of the class $\calF_{L,U}(K)$ are well known, and were first obtained in general 
dimension $d$ by~\cite{bronshtein1976} (see also~\cite{dudley2014}).
The following is a version of Bronshtein's result stated in Theorem 1 of \cite{guntuboyina2012b} with explicit
dependence on the constants $L$ and $U$, which we
shall also use in Section~\ref{sec:upper_bounds_unbounded}.  
\begin{lemma}[\cite{bronshtein1976}]
\label{lem:bronshtein}
There exist universal constants $C,\epsilon_0 > 0$ such that
for every $L,U > 0$ and $b > a$, we have for all $\epsilon \leq \epsilon_0(U+L(b-a))$,
$$\log N \big(\epsilon, \calF_{L,U}([a,b]^d), L^\infty\big) \leq C \left(\frac{U + L(b-a)}{\epsilon}\right)^{\frac d 2}.$$
\end{lemma}
Notice that, by condition~\ref{assm:S1},  
$$N(\cdot, \calF_{1,1}(\calX), L^2(\mu_n)) \leq  
  N(\cdot, \calF_{1,1}(\calX), L^\infty) \leq 
  N(\cdot, \calF_{1,1}([-1,1]^d), L^\infty),$$
  and the same upper bound holds for the covering number $N(\cdot, \calF_{1,1}(\calY), L^2(\nu_n))$. 
Combine these facts with equation~\eqref{eq:emp_proc_red} and with Lemmas~\ref{lem:chaining}--\ref{lem:bronshtein} to deduce that for any $\tau > 0$,
\begin{align}
\label{eq:chaining_integral_calc}
\bbE &\big| \calT_c(\mu_n,\nu_n)  - \calT_c(\mu,\nu) \big|
 \lesssim  \Lambda \left[n^{-1/2} + \tau + \frac 1 {\sqrt n}\int_\tau^\infty \epsilon^{-\frac d 4} d\epsilon\right]
 \lesssim  \Lambda \left[n^{-1/2} + \tau + \frac {\tau^{1 - \frac d 4}} {\sqrt n}  \right],
\end{align}
where we have used the assumption $d \geq 5$. Choosing $\tau \asymp n^{-2/d}$ leads to the claimed bound,
\begin{align*}
\nonumber
\bbE \big| \calT_c(\mu_n,\nu_n)  - \calT_c(\mu,\nu) \big|
\lesssim \Lambda n^{-2/d}.
\end{align*}

\subsubsection{Case 2: $1 < \alpha < 2$} 

We prove the claim using a smooth approximation of the cost~$h$,   thereby appealing
to the result of Case 1.
Let $K: \bbR^d \to \bbR_+$ be an even, smooth mollifier with support lying in $B_{0,1}$.
For $\sigma > 0$, write $K_\sigma(x) = \sigma^{-d} K(x/\sigma)$.
Define the modified cost function $c_\sigma(x,y) = h_\sigma(x-y)$, where
$h_\sigma = h \star K_\sigma.$
\begin{lemma}
\label{lem:cost_convolution}
Assume conditions~\ref{assm:S1}, \ref{assm:global}--\ref{assm:bded_smoothness}
hold for some~$\alpha\in(1,2)$. Then, 
there exist universal constants $C> 0$ and $\epsilon \in (0,1)$  such that for all $\sigma \in (0, \epsilon)$, the following statements hold.
\begin{enumerate} 
\item[(i)] We have,
$\|h - h_\sigma\|_{L^\infty(\calZ)} \leq \Lambda\sigma^\alpha.$
\item[(ii)] The cost function $c_\sigma$ itself satisfies condition \ref{assm:global},
and satisfies condition~\ref{assm:bded_smoothness} 
in the sense that there exists an open set $\tilde\calZ_1\supseteq \calZ$ contained in $B_{0,2}$ 
such that $h_\sigma \leq 1$ on $\tilde \calZ_1$ and
$$\norm{h_\sigma}_{\calC^2(\tilde\calZ_1)} \leq \Lambda_\sigma := C\Lambda  \sigma^{\alpha-2}.$$

\end{enumerate}
\end{lemma}
The proof of Lemma~\ref{lem:cost_convolution} appears in Appendix~\ref{app:proof_lem_cost_conv}. Notice that
$$\sup_{\substack{\tilde\mu \in\calP(\calX) \\ \tilde\nu\in\calP(\calY)}}
		|\calT_c(\tilde\mu,\tilde\nu) - \calT_{c_\sigma}(\tilde\mu,\tilde\nu)| \leq \norm{h-h_\sigma}_{L^\infty(\calZ)},$$
thus Lemma~\ref{lem:cost_convolution}(i) implies
\begin{align*}
\bbE \big| &\calT_c(\mu_n,\nu_n) - \calT_c(\mu,\nu)\big| \\
 &\leq \bbE \big| \calT_c(\mu_n,\nu_n) - \calT_{c_\sigma} (\mu_n,\nu_n)\big| + \bbE \big| \calT_{c_\sigma}(\mu_n,\nu_n) - \calT_{c_\sigma}(\mu,\nu)\big|
  + \bbE \big| \calT_{c_\sigma}(\mu,\nu) - \calT_c(\mu,\nu)\big| \\
 &\leq 2 \Lambda\sigma^\alpha + \bbE \big| \calT_{c_\sigma}(\mu_n,\nu_n) - \calT_{c_\sigma}(\mu,\nu)\big|.
\end{align*}
On the other hand, by Lemma~\ref{lem:cost_convolution}(ii), 
%the cost $c_\sigma$
%satisfies conditions~\ref{assm:global} and~\ref{assm:bded_smoothness} with $\norm{h}_{\calC^2} = \sigma^{p-2}$, thus 
we may apply the result of Case 1 to obtain,
$$\bbE \big| \calT_{c_\sigma}(\mu_n,\nu_n) - \calT_{c_\sigma}(\mu,\nu)\big| %\lesssim \Lambda_\sigma n^{-2/d}
\lesssim   \Lambda  \sigma^{\alpha-2} n^{-2/d}.$$
Altogether, we deduce that for any $\sigma \in (0,\epsilon)$,
\begin{align*}
\bbE \big| \calT_{c}(\mu_n,\nu_n) - \calT_{c}(\mu,\nu)\big| 
 \lesssim \Lambda \left[\sigma^\alpha +  \sigma^{\alpha-2} n^{-\frac 2 d}\right].
\end{align*}
Choosing $\sigma \asymp n^{-1/d}$ leads 
to an upper bound scaling at the rate $\Lambda n^{-\alpha/d}$ on the right-hand side of the above display, as claimed.
 
\subsubsection{Case 3: $0 < \alpha \leq 1$}
When $\alpha \in (0,1]$, the claim follows by a simpler argument than 
that of Case 1. 
For any bounded set $K \subseteq \bbR^d$ and $L  > 0$, define
the $\alpha$-H\"older ball 
$\calC^\alpha(K; L) = \{f \in \calC^\alpha(K): \norm f_{\calC^\alpha(K)} \leq L\}$, and let  $\varphi_n$ and $\psi_n$ be defined as in Case 1. 
When $\alpha < 2$, Lemma~\ref{lem:semi_concave} no longer guarantees that these potentials are semi-concave, however 
the following  H\"older estimate  is easily derived, and stated without proof. 
\begin{lemma} 
Assume conditions~\ref{assm:S1} and~\ref{assm:global}, and that
condition~\ref{assm:bded_smoothness} holds with $\alpha \in (0,1]$. 
Then, $\varphi_n\in \calC^\alpha(\calX; \Lambda)$ and $\psi_n \in \calC^\alpha(\calY;\Lambda)$ for all $n \geq 1$. 
\end{lemma} 
By an analogous reduction as in Case 1, we therefore have for any $\tau > 0$,
\begin{align*}
\bbE \big| \calT_c(\mu_n,\nu_n) - \calT_c(\mu,\nu)\big|
% &\lesssim \Lambda \Big(n^{-1/2}+  \bbE [\tilde \Delta_n] \Big) \\
 &\lesssim \Lambda \left[n^{-1/2} + \tau + \frac 1 {\sqrt n} \int_\tau^\infty \sqrt{\log N(\epsilon, \calC^\alpha([-1,1]^d;1), L^\infty)} d\epsilon\right].
% &\lesssim \Lambda \left[n^{-1/2} + \tau + \frac 1 {\sqrt n} \int_\tau^\infty \epsilon^{-\frac d {2\alpha}} d\epsilon\right].  
\end{align*}
By Theorem 2.7.1 of~\cite{vandervaart1996}, one has
$$\log N(\epsilon, \calC^\alpha([-1,1]^d;1), L^\infty) \lesssim \epsilon^{-d/\alpha}, \quad \epsilon > 0,$$
%and this upper bound may clearly be replaced by zero for all
%$\epsilon \geq 2$. 
implying that the right-hand side of the penultimate display is of order $\Lambda n^{-\alpha/d}$ if $\tau \asymp n^{-\alpha/d}$ 
and $d \geq 5 > 2\alpha$.   The claim follows.
\qed

\begin{remark}
\label{rem:low_dim}
Though Theorem~\ref{thm:main_unbounded} is only stated when $d \geq  5$, 
a simple extension of our proof yields the rate 
$ n^{-1/2}$ whenever $d \leq 4$
and $\alpha = 2$, up to a logarithmic
factor when $d = 4$; this follows by taking $\tau = 0$ in equation~\eqref{eq:chaining_integral_calc}. 
A similar extension can be made when $\alpha \in (0,1]$. On the other hand, 
our mollification step for the case $\alpha \in (1,2)$
does not appear to yield a sharp convergence rate when $d < 2\alpha < 5$. 
After a preprint of our paper was made available, 
the work of~\cite{hundrieser2022} has extended Theorem~\ref{thm:ub_general},
by showing that in all dimensions $d \geq 1$, 
and for all $\alpha \in (0,2]$, the upper bound
\begin{align}
\bbE \big| \calT_c(\mu_n,\nu_n) - \calT_c(\mu,\nu)\big| 
\lesssim \begin{cases}
n^{-\alpha/d}, & 2\alpha < d \\ 
n^{-\alpha/d}\log n, & 2\alpha=d \\ 
n^{-1/2}, & 2\alpha > d
\end{cases}
\end{align}
holds
under the same assumptions as those of Theorem~\ref{thm:ub_general}.
In particular, this result recovers our Theorem~\ref{thm:ub_general} 
when $d \geq 5$. 
\end{remark}

\section{Upper Bounds for Unbounded Measures under Tail Conditions} 
\label{sec:upper_bounds_unbounded} 
We now turn to upper bounding the rate of convergence of the empirical
optimal transport cost for measures $\mu,\nu \in \calP(\bbR^d)$ with unbounded support, 
under suitable tail  conditions. We shall assume that 
the cost function $c$ satisfies the following smoothness assumption, which is a suitable generalization
of condition~\ref{assm:bded_smoothness} to the present setting. 
\begin{enumerate}[leftmargin=1.65cm,listparindent=-\leftmargin,label=(\textbf{H\arabic*}),start=2]  
\item \label{assm:unbded_smoothness}
$h \in \calC^2_{\mathrm{loc}}(\bbR^d)$, and there exist $p,\Lambda \geq 1$ such that $\norm{h}_{\calC^2(B_{0,r})} \leq \Lambda r^{p}$
for all $r \geq 1$. 
\end{enumerate}
Notice that unlike in Section~\ref{sec:upper_bounds_compact}, we limit our exposition
to costs lying in $\calC^2_{\mathrm{loc}}$ rather than $\calC^\alpha_{\mathrm{loc}}$ for
all $ \alpha \in (0,2]$. As we shall see, our main result is nevertheless sufficiently general
to cover the costs $h(x) = \norm x^p$ for all $p > 1$, via an approximation argument.

Our upper bounds in  Section~\ref{sec:upper_bounds_compact}
hinged upon Lemma~\ref{lem:semi_concave}, which provided
quantitative estimates on the  Lipschitz and semi-concavity constants 
of  optimal Kantorovich potentials, for any sufficiently smooth cost function over
a compact set. 
In contrast, we now only assume a local
H\"older estimate on $c$ in assumption~\ref{assm:unbded_smoothness}, 
thus the optimal Kantorovich potentials between two measures
on $\bbR^d$ will generally not be {\it globally} Lipschitz or semi-concave. 
While these properties nevertheless hold {\it locally} under rather general conditions
\citep{gangbo1996}, we are not aware of existing quantitative 
estimates  under the conditions required for our development. 
One of our key technical contributions in this section 
is to obtain such quantitative bounds, which we now describe before stating our main result.
We begin with the following straightforward generalization of Lemma~\ref{lem:semi_concave},
whose proof appears in Appendix~\ref{app:semi_concave_general}.
\begin{lemma} 
\label{lem:semi_concave_general}
Let $c$ be a cost function satisfying conditions~\ref{assm:global}
and~\ref{assm:unbded_smoothness} with $h$ superlinear. Given two measures 
$\mu,\nu\in \calP(\bbR^d)$, 
let $\pi \in \Pi(\mu,\nu)$ be an optimal coupling with respect to $c$, 
and assume there exists a locally bounded $c$-concave function $\varphi:\bbR^d \to \bbR$ such that
$\supp(\pi) \subseteq \partial^c\varphi$. 
Let $r \geq 1$, 
and let 
$$\Lambda_r  = \Lambda\sup\big\{\|x-y\|^p: x \in B_{0,r}, \ y \in \partial^c \varphi(B_{0,r})\big\}.$$
Then, there exist universal constants $c_1,c_2 > 0$ such that the map 
$$\phi:B_{0,r} \to \bbR, \quad \phi(x) = \varphi(x) - c_1\Lambda_r \norm x^2,$$
is concave, and Lipschitz with parameter $c_2r\Lambda_r$.
\end{lemma}
Lemma~\ref{lem:semi_concave_general} shows that
the local Lipschitz and semi-concavity
constants for a Kantorovich potential $\varphi$ are largely driven by the maximal displacement induced by the coupling
$\pi$ over points lying in $B_{0,r}$. 
We will show how $L^\infty$ estimates on these displacements can be obtained 
under the following conditions on $\mu,\nu$.

\begin{enumerate} 
\item[(i)] \textbf{\bf Super-Gaussian Anticoncentration.} 
We will say a measure $\mu$ is $(\gamma,b)$-super-Gaussian for some $\gamma,b > 0$ if 
for any $x \in \bbR^d$,  
$$ \mu(B_{x}) \geq b \cdot \bbP(Z \in B_{x}),\quad \text{where } Z \sim N(0,\gamma^2). $$
\item[(ii)] \textbf{Sub-Weibull Concentration.} A measure $\nu$ is said to be $(\sigma,\beta)$-sub-Weibull \citep{kuchibhotla2020a,vladimirova2020}
for some $\sigma > 0$ and $0 < \beta \leq 2$~if
$$\int \exp\left[\frac 1 {2} \left(\frac{\norm y}{\sigma}\right)^\beta\right] d\nu(y) \leq 2.$$
\end{enumerate}
The assumption of super-Gaussianity 
implies an {\it anticoncentration} bound for the underlying measure, in the sense
that it cannot place significantly less probability mass than a Gaussian distribution in any  
unit-radius ball. In Appendix~\ref{app:super_gaussian}, we show that
a measure is super-Gaussian whenever it admits a regular Lebesgue density in the sense of~\cite{polyanskiy2016}.
For instance, Polyanskiy and Wu show that for any probability measure $\mu$ with finite first moment, 
the mixture distribution $K_\tau \star \mu$  admits a regular density, where $K_\tau$ is the $N(0, \tau^2 I_d)$
density for some $\tau > 0$. Any such measure is thus also super-Gaussian. 
We also note that absolute continuity is not necessary for super-Gaussianity; for example, 
given any (possibly atomic) measure $\rho \in \calP(\bbR^d)$, any super-Gaussian
measure $\mu$, and any $\lambda \in (0,1)$, the measure $\lambda \rho + (1-\lambda) \mu$ is also super-Gaussian. 

On the other hand, the sub-Weibull condition is a {\it concentration} assumption which   
generalizes the well-known sub-Gaussian and sub-exponential conditions,
which respectively correspond to the cases
$\beta=2$ and $\beta=1$ up to rescaling of the constant $\sigma$
\citep{boucheron2013}. 
Indeed, it is a straightforward consequence of Markov's inequality
that if $Y$ has $(\sigma,\beta)$-sub-Weibull distribution
for some $\beta \in (0,2]$ and $\sigma > 0$,
then for all $u > 0$,
\begin{equation}
\label{eq:sub_weibull_concentration}
\bbP(\norm Y \geq u) \leq 2 \exp\left\{-\frac 1 {2} \left(\frac{u}{\sigma}\right)^\beta\right\}.
\end{equation}
Finally, we shall require the following condition on the cost function $c$. 
\begin{enumerate}[leftmargin=1.65cm,listparindent=-\leftmargin,label=(\textbf{H\arabic*}),start=3]  
\item \label{assm:unbded_growth}
We have $h(0)=0$. Furthermore, there exist constants $p > 1$, $\kappa \geq 1$, and a convex  differentiable function $\omega:(1,\infty) \to (1,\infty)$ such that 
$$h(z) = \omega(\norm z),\quad \text{and}\quad
\frac 1 \kappa \norm z^{p-1} \leq\omega'(\norm z) \leq \kappa \norm z^{ p  - 1}\quad \text{for all } \norm z > 1.$$ 
\end{enumerate}
Condition~\ref{assm:unbded_growth} implies that the function $h$ is superlinear, with order of growth 
comparable to that of $\norm \cdot^p$ for some $p > 1$.
Aside from the assumption $h(0)=0$, which can always be satisfied by translation,  
we emphasize that condition~\ref{assm:unbded_growth} does not constrain the behaviour of~$h$ near zero, but is nevertheless
stronger than the conditions assumed in
Section~\ref{sec:upper_bounds_compact}. 
Therefore, we provide several examples of cost functions satisfying these two conditions
before turning to our main results.

The most important example of a cost satisfying our assumptions is, of course, $\norm \cdot^p$ for $p \geq 2$.
However, when $p \in (1, 2)$, the cost $\norm \cdot^p$ does not satisfy condition~\ref{assm:unbded_smoothness}; to study this case, we 
will employ an approximation argument with the cost function $h_\epsilon(x) = (\norm x^2 + \epsilon^{2/p})^{p/2}-\epsilon$, 
which satisfies \ref{assm:unbded_smoothness}--\ref{assm:unbded_growth} for any $\epsilon > 0$ and $p > 1$.  
This cost function has been of interest in its own right in the optimal transport literature, 
as it forms an approximation of $\norm\cdot^p$
which satisfies the celebrated Ma-Trudinger-Wang regularity conditions even when $p \neq 2$~\citep{ma2005,li2014}. 
 
	Conditions~\ref{assm:unbded_smoothness}--\ref{assm:unbded_growth} also hold for costs that have different power-type behaviors at the origin and infinity, such as $h(z) = \lambda_p \|z\|^p + \lambda_q \|z\|^q$ for $p > q \geq 2$, which arise in the study of modified transport problems with congestion costs~\citep{BraCarSan10,CarJimSan08}.
	
	More generally, conditions~\ref{assm:unbded_smoothness}--\ref{assm:unbded_growth} are 
satisfied by any 
twice continuously differentiable
convex cost function of the form
$$h(x) \propto \begin{cases}
\|x\|^p, & \|x\| > 1 \\
h_0(x) ,  & \|x\| \leq 1,
\end{cases}$$
where $h_0(0)=0$ and $p \geq 2$. This family includes, for instance, smooth approximations
of the truncated cost $h(x) = \|x\|^p I(x \in B_{0}^c)$, and $\ell_p$ analogues of Huber's loss function.
   
While the smoothness condition~\ref{assm:unbded_smoothness} will be needed 
in order to appeal to Lemma~\ref{lem:semi_concave_general},
assumption~\ref{assm:unbded_growth} is sufficient to obtain the following result, which plays a central role in our~development.
\begin{theorem}
\label{thm:coupling_quantitative}
Let $\mu,\nu \in \calP(\bbR^d)$ and assume $\nu$ is $(\sigma,\beta)$-sub-Weibull. 
Let $c$ be a cost function satisfying conditions~\ref{assm:global} and~\ref{assm:unbded_growth}, 
and let $\pi \in \Pi(\mu,\nu)$ have $c$-cyclically monotone support.
Then, there exists a constant $C > 0$ depending on $d, p, \beta, \kappa$  such that for 
any $c$-concave function $\varphi$ satisfying $\supp(\pi) \subseteq \partial^c\varphi$,
\begin{equation} 
\label{eq:coupling_quantitative_part1}
\sup_{y \in \partial^c\varphi(x)}\norm{y} \leq C \sigma \left\{(\norm x+1) \vee \sup_{\norm{x-y}\leq 2} \left[\log\left(\frac 1 {\mu(B_y)}\right)\right]^{\frac 1 {\beta}}\right\}, \quad x \in \bbR^d.
\end{equation}
In particular, if 
$\mu$ is $(\gamma,b)$-super-Gaussian, $h$ is strictly convex, and $\calT_c(\mu, \nu) < \infty$,
then the unique optimal transport map $T$
pushing $\mu$ forward onto $\nu$  satisfies for  $\mu$-a.e. $x \in \bbR^d$,
\begin{equation}
\label{eq:coupling_quantitative_part2}
\norm{T(x)} \leq C' \sigma  (\norm x+1)^{\frac {2} {\beta}},
\end{equation}
for a constant $C' > 0$ depending on $d,\kappa,p,\beta, \gamma,b$.
\end{theorem}
We defer the proof to Section~\ref{sec:proof_coupling_quantitative}.
Theorem~\ref{thm:coupling_quantitative} implies that 
any optimal transport plan between $\mu$ and $\nu$ does not move
probability mass from any point $x \in \bbR^d$ by more than a polynomial of $\norm x$. 
To obtain this result, we required an anticoncentration assumption on the source
measure $\mu$ and a concentration assumption on the target measure $\nu$, ensuring
that their tails are sufficiently comparable
to avoid large transports of mass. It is easy 
to see that assumptions of this nature are necessary: for
instance, if $\mu$ were compactly supported and $\nu$ were
supported over $\bbR^d$, any transport plan in $\Pi(\mu,\nu)$
would couple a nonzero amount of mass from the bounded support of $\mu$ 
with points lying at an arbitrarily far distance.

In the special case where $\nu$ is sub-Gaussian, 
its tails are no heavier than those of a super-Gaussian measure~$\mu$. 
Equation~\eqref{eq:coupling_quantitative_part2}
 shows that the optimal transport map from $\mu$ to $\nu$
grows at most linearly in this regime, irrespective of the order of growth $p$
of the cost function.  
This bound is clearly unimprovable in general, as can be seen
by taking $\mu=\nu$. 

Our proof of Theorem~\ref{thm:coupling_quantitative} is inspired
by its non-quantitative analogues proven
by~\cite{gangbo1996}, and 
by~\cite{colombo2021} who derived analogous quantitative bounds
for the special case where $\mu$ is a Gaussian measure and $h(x) = \norm x^2$. 
Unlike \cite{colombo2021}, our result holds for any
cost function satisfying conditions~\ref{assm:global} and~\ref{assm:unbded_growth}, 
and for general measures $\mu,\nu$ which
are not presumed to be absolutely continuous with respect to the Lebesgue measure. 
We shall require this level of generality in the sequel, when Theorem~\ref{thm:coupling_quantitative}
will be invoked for $\mu$ and $\nu$ replaced by their empirical counterparts.

Equipped with Theorem~\ref{thm:coupling_quantitative},
we are ready to state the main result of this section.
\begin{theorem}
\label{thm:main_unbounded}
Assume conditions \ref{assm:global},~\ref{assm:unbded_smoothness} and~\ref{assm:unbded_growth} hold.
Assume further that $\mu,\nu\in \calP(\bbR^d)$
are both $(\sigma,\beta)$-sub-Weibull, and  $(\gamma,b)$-super-Gaussian. 
Then, there exists a constant $C > 0$ depending on $\sigma,\beta,\gamma,b,\kappa,p,d$ such that
$$\bbE \big| \calT_c(\mu_n, \nu_n) - \calT_c(\mu,\nu)\big| \leq C \Lambda n^{-\frac 2 d}.$$
\end{theorem} 	
Theorem~\ref{thm:main_unbounded} shows that, for $\calC^2_{\mathrm{loc}}$ convex costs
with polynomial rate of growth, 
the $n^{-2/d}$ rate of convergence obtained for compactly
supported measures in Section~\ref{sec:upper_bounds_compact} carries over to unboundedly supported measures,
with tails satisfying suitable concentration and anticoncentration 
conditions. While Theorem~\ref{thm:main_unbounded} does not provide upper bounds for $\calC^\alpha_{\mathrm{loc}}$ costs when $\alpha < 2$, it is sufficiently general   to deduce
the following special case.
\begin{corollary}
 \label{cor:norm_p_unbounded}
Assume $\mu,\nu \in \calP(\bbR^d)$ satisfy the same conditions as Theorem~\ref{thm:main_unbounded}.
Then, for all $p > 1$, 
$$\bbE \big| W_p^p(\mu_n, \nu_n) - W_p^p(\mu,\nu)\big| \lesssim 
\begin{cases}
n^{-p/d}, & 1 < p \leq 2\\
n^{-2/d}, & 2 \leq p < \infty.
 \end{cases}$$
\end{corollary}
For the regime $p \geq 2$, this result follows as a direct consequence
of Theorem~\ref{thm:main_unbounded}, while for the regime $1 < p < 2$, we achieve
the claim by using a smooth uniform approximation of $\norm\cdot^p$ which satisfies the conditions of
Theorem~\ref{thm:main_unbounded}. The proof is deferred to Appendix~\ref{pf:root_n_unbounded}.

By reasoning identically as in Corollary~\ref{cor:wasserstein}, one can also
deduce a convergence rate for  empirical Wasserstein distances between measures with unbounded support. In particular, equation~\eqref{eq:cor_wasserstein}
continues to hold for all $\mu,\nu \in \calP(\bbR^d)$ satisfying the conditions of Theorem~\ref{thm:main_unbounded}, 
and satisfying
$W_p(\mu,\nu) \geq \delta_0>0$.

\subsection{Proof of Theorem~\ref{thm:main_unbounded}}
 
As in the proof of
Theorem~\ref{thm:ub_general}, we shall reduce the problem of bounding the 
$L^1(\bbP)$ convergence rate of $\calT_c(\mu_n,\nu_n)$ 
to that of bounding the supremum of an empirical process. Unlike in Theorem~\ref{thm:ub_general},
the relevant empirical process in this section will be indexed by {\it locally} 
semi-concave Lipschitz functions, with quantitative local Lipschitz and semi-concavity moduli 
obtained in part by appealing to Theorem~\ref{thm:coupling_quantitative}.
Our proof proceeds with eight steps, the first five of which carry out this reduction, and the last 
three of which bound the resulting
empirical process. 
Throughout the proof, $C,C', C_i, c_i > 0$, $i\geq 0$,  
denote constants possibly depending on $\sigma,\beta,\gamma,b,\kappa,p,d$, 
which do not depend on $\Lambda$ or otherwise on $c$, $\mu$ and $\nu$, and 
whose value may change from line to line. 
Likewise, the symbols $\lesssim$ and $\asymp$ hide constants possibly
depending on the former quantities. 
All intermediary results appearing in the sequel are proven in 
Appendix~\ref{app:proofs_unbounded}. 

{\bf Step 0: Setup.} %We prove the claim by employing a peeling argument. 
Let $L_j = [- 3^{j},   3^{j}]^d$
for all $ j\geq 0$. For all $j \geq 1$, let $I_{j1}, \dots, I_{jm_d}$ denote the  
 $m_d:=3^d - 1$ %consisting of
cubes of side-length $2 \cdot 3^{j-1}$ forming the natural partition of $L_j \setminus L_{j-1}$.
For notational convenience, set $I_0 \equiv I_{0k} = L_0$ for all $k=1,\dots,m_d$, and define
$$I_j := L_j \setminus L_{j-1} = \bigcup_{k=1}^{m_d} I_{jk},\quad j\geq 1.$$
Note that $\{I_j : j \geq 0\}$ and $\{I_{jk}:j\geq 0, 1 \leq k \leq m_d\}$ are partitions of $\bbR^d$, up to measure-zero
intersections. We also write $\ell_j = \sup_{x \in I_j} \norm x = \sqrt d 3^j$ for all $j \geq 0$.

Let $\calF$ denote the set of convex functions over $\bbR^d$.
Recall that $\calF_{m,u}(I)$ denotes the set of $m$-Lipschitz convex functions over
a convex set $I \subseteq \bbR^d$, which are uniformly bounded over $I$ by $u > 0$.
 Let $M = (M_{j}: j \geq 0)$  and $U = (U_{j}: j \geq 0)$
 denote  sequences of nonnegative real numbers, and let 
\begin{equation}
\label{eq:defn_calKM}
\calK_{M,U} = \Big\{f:\bbR^d \to \bbR :  (-f)|_{I_{jk}} \in \calF_{M_{j},U_{j}}(I_{jk}), \ j\geq 0 , 1 \leq k \leq m_d\Big\}.
\end{equation}

{\bf Step 1: Extension of Kantorovich Potentials.}
Let
$$R_n = \max_{1 \leq i \leq n} \norm{X_i}\vee \norm{Y_i}.$$
Conditions~\ref{assm:global} and~\ref{assm:unbded_growth} imply that $h(z) \leq \kappa \norm z^p$
for all $z \in \bbR^d$, thus $h$ is bounded above by $\widebar R_n := \kappa (2R_n)^p$  over $B_{0,2R_n}$. 
It can then be deduced from Lemma~\ref{lem:kantorovich_background}(ii) that
there exist potentials 
$f_n : \supp(\mu_n) \to [-\widebar R_n, 0]$ and $g_n: \supp(\nu_n) \to [0,\widebar R_n]$
such that  $(f_n,g_n) \in \Phi_c(\mu_n,\nu_n)$ and $(f_n, g_n)$ is 
optimal for the optimal transport problem from $\mu_n$ to $\nu_n$.
We extend the domain of $f_n$ and $g_n$ to $\bbR^d$ using the following
construction. Define for all $y \in \bbR^d$,
$$\eta_n(y) = \inf_{x \in \supp(\mu_n)} \Big\{c(x,y) - f_n(x)\Big\} \wedge \widebar R_n,$$
and for all $x,y \in \bbR^d$,
\begin{align*}
\varphi_n(x) = \eta_n^{c}(x)= \inf_{y \in \bbR^d} \Big\{ c(x,y) - \eta_n(y)\Big\},\quad 
\psi_n(y) = \eta_n^{cc}(y) = \inf_{x \in \bbR^d} \Big\{ c(x,y) - \eta_n^c(x)\Big\}.
\end{align*}
\begin{lemma}
\label{lem:extension_potentials}
Given an optimal coupling $\pi_n$ between $\mu_n$ and $\nu_n$, the following hold.
\begin{enumerate}
\item[(i)] For all $x,y \in \bbR^d$, $\varphi_n(x) + \psi_n(y)\leq c(x,y)$. 
\item[(ii)] $\varphi_n(x) = f_n(x)$ for all $x \in \supp(\mu_n)$, and
$\psi_n(y) = g_n(y)$ for all $y \in \supp(\nu_n)$.
In particular,
$$\calT_c(\mu_n,\nu_n) = \int \varphi_n d\mu_n + \int \psi_n d\nu_n.$$
%\item[(iii)] The sets $\partial^c\varphi_n$ and $\partial^c\psi_n$ are $c$-monotone, in the sense
%that for every $(x,y),(x',y') \in \partial^c\varphi_n$,
%$$c(x,y) + c(x',y') \leq c(x,y') + c(x',y).$$
\item[(iii)] For all $x \in \bbR^d$,  $|\varphi_n(x)| \vee |\psi_n(x)| \leq \widebar R_n$.
\item[(iv)] For all $(x,y) \in \supp(\pi_n)$, $(x,y) \in \partial^c\varphi_n$ and $(y,x) \in \partial^c\psi_n$.
\end{enumerate}   
\end{lemma}
By Lemma~\ref{lem:extension_potentials}(iii), $\varphi_n$ and $\psi_n$ are bounded, so $\varphi_n \in L^1(\mu)$ and $\psi_n \in L^1(\nu)$. 
This fact combined with Lemma~\ref{lem:extension_potentials}(i)
guarantees that $(\varphi_n,\psi_n) \in \Phi_c(\mu,\nu)$, whence 
 \begin{align}
 \label{eq:reduction_unbded} 
\nonumber \calT(\mu,\nu)
 &= \sup_{(\varphi,\psi)\in \Phi_c(\mu,\nu)} \int \varphi d\mu + \int \psi d\nu \\
 &\geq \int \varphi_n d\mu + \int \psi_n d\nu
 = \calT_c(\mu_n, \nu_n) + \int \varphi_n d(\mu-\mu_n) + \int \psi_n d(\nu-\nu_n).
\end{align}
It remains to bound the last two terms on the right-hand side of the above display.
We shall do so by first proving that $\varphi_n,\psi_n \in \calK_{M,U}$
with high probability, for suitable sequences $M$ and $U$. 
We focus on $\varphi_n$, and a symmetric
argument can be used for $\psi_n$. 

By Lemma~\ref{lem:semi_concave_general}, recall that
the Lipschitz and semi-concavity moduli of $\varphi_n|_{I_{jk}}$
are largely driven by the magnitude of the $c$-superdifferential
$\partial^c \varphi_n(I_{jk})$. The bulk of our effort will go into bounding this quantity.
In fact, it will suffice to bound that of $\partial^c\varphi_n(L_j)$, 
for all $j \geq 0$.  
To do so, we proceed with the following step,
in view of invoking Theorem~\ref{thm:coupling_quantitative}
with the measures $\mu_n$ and $\nu_n$.
 
\textbf{Step 2: Global Concentration and Local Anticoncentration of $\mu_n, \nu_n$.} 
Fix $\rho = \frac {2p}{\beta} \vee \frac d 4$, and set
$$V_{1,n} =  \int \exp\left(\frac{\norm{x}^\beta}{2\rho \sigma^\beta}\right) d\mu_n(x),\quad 
  V_{2,n}   = \int \exp\left(\frac{\norm{y}^\beta}{2\rho\sigma^\beta}\right)d\nu_n(y).$$
By Jensen's inequality, notice that  
\begin{equation}
\label{eq:sub_weibull_red}
\int \exp\left(\frac{\norm{x}^\beta}{2 \rho V_{1,n}\sigma^\beta}\right)d\mu_n(x)
\leq V_{1,n}^{1/V_{1,n}} \leq 2,
\end{equation}
implying that $\mu_n$ is $(\sigma (\rho V_{1,n})^{1/\beta} , \beta)$-sub-Weibull. Similarly, 
$\nu_n$ is $(\sigma (\rho V_{2,n})^{1/\beta}, \beta)$-sub-Weibull, implying that both are $(\sigma (\rho V_n)^{1/\beta}, \beta)$-sub-Weibull
when $V_n = V_{1,n} + V_{2,n}$. 

We further show that $\mu_n$  satisfies a high-probability anticoncentration bound in a sufficiently small
region about the origin. Specifically, define the integer
\begin{equation}
\label{eq:Jn}
J_n = \left\lfloor \frac 1 2 \log_3\left(\frac { \gamma^2} {4d} \log n\right)  \right\rfloor,
\end{equation}
and the event 
$$A_n = \bigcap_{j=0}^{J_n} \left\{
\inf_{x \in L_{j}}\inf_{\norm{x-y}\leq 2} \mu_n(B_y) \geq \frac{C_1}{2} { \exp(- \ell_j^2/\gamma^2)}\right\},$$ 
where we recall that the parameter $\gamma$ arises from the super-Gaussianity assumption on~$\mu$ and~$\nu$.
 The following result is proven using elementary tools from empirical process theory.
\begin{lemma}
\label{lem:anticonc}
There exists a sufficiently large choice of the constant $C_1 > 0$, depending only on $\gamma$,
such that  $\bbP(A_n^c) \lesssim 1/n$.
\end{lemma}

\textbf{Step 3: Bounding $\partial^c\varphi_n$.}
Step 2 will allow us to bound $\partial^c\varphi_n(L_j)$ whenever $0 \leq j \leq J_n$
by invoking Theorem~\ref{thm:coupling_quantitative}.
On the other hand, $\mu_n$ may place insufficient mass outside the box $L_{J_n}$ to appeal to Theorem~\ref{thm:coupling_quantitative}
when $j > J_n$, thus we treat this case separately below.
\begin{itemize}
\item \textbf{Regime 1: $0 \leq j \leq J_n$.} 
By Step 2, $\nu_n$ is $(\sigma (\rho V_n)^{1/\beta}, \beta)$-sub-Weibull, and 
$\supp(\pi_n) \subseteq \partial^c \varphi_n$ by Lemma~\ref{lem:extension_potentials}(iv).
Therefore, by Theorem~\ref{thm:coupling_quantitative}, we have for all $0 \leq j \leq J_n$, 
\begin{align*}
\sup_{\substack{y \in \partial^c\varphi_n(L_j)}}\norm{y} &\lesssim 
V_n^{\frac 1 \beta} \left\{(\ell_j+1)  \vee \sup_{\norm{x-y}\leq 2} \left[\log\left(\frac 1 {\mu_n(B_y)}\right)\right]^{\frac 1 {\beta}}\right\}. 
\end{align*} 
Over the event $A_n$, we therefore have uniformly in $0 \leq j \leq J_n$,
$$
 \sup_{\substack{y \in \partial^c\varphi_n(L_j)}}\norm{y}  
  \lesssim V_n^{\frac 1 \beta} \left[\ell_j+1 + {\ell_j^2/\gamma^2}\right]^{\frac 1 {\beta\wedge 1}}
  \lesssim V_n^{\frac 1 \beta} \ell_j^{\frac{2}{\beta\wedge 1}}
  \lesssim V_n^{\frac 1 \beta}  3^{jq_1},$$
for a large enough exponent $q_1 \geq 1$. 
 
\item \textbf{Regime 2: $J_n < j <\infty$.} In this regime, it will suffice
to provide a crude bound on $\partial^c\varphi_n(L_j)$.  
We begin with the following result, which is a quantitative generalization of Proposition~C.4 of~\cite{gangbo1996}. 
\begin{proposition}
\label{prop:phi_bound_superdiff}
Let $R,r \geq 4$. Let $\varphi:\bbR^d \to \bbR$ be a  $c$-concave function such that $|\varphi|\leq R$ over $B_{0,r}$. Then,
under conditions \ref{assm:global} and~\ref{assm:unbded_growth}, we have 
$$\sup_{y \in \partial^c \varphi(B_{0,r/ 2})} \norm{y}^{p-1} \leq C_{p,\kappa}( r^{p-1} +  R),$$
for a constant $C_{p,\kappa} > 0$ depending only on $p$ and $\kappa$.
\end{proposition}
By Proposition~\ref{prop:phi_bound_superdiff}, it will suffice to bound
$|\varphi_n(x)|$ uniformly over $x\in L_j$, for all $j \geq J_n$. Recall from Lemma~\ref{lem:extension_potentials} that $|\varphi_n(x)|\leq \widebar R_n$, thus it suffices to bound $\widebar R_n$. 
Define $D_n = \sigma(4\log n)^{1 /\beta}$
and the event $A_n' = \{R_n \leq D_n\}$. Apply a union bound together with the sub-Weibull assumption
on $\mu$ and $\nu$ to deduce that  
\begin{equation}
\label{eq:Bn_highprob} 
\bbP((A_n')^\cp) \lesssim n\exp\left\{-\frac 1 2 \left(\frac{D_n}{\sigma}\right)^\beta\right\} \lesssim \frac 1 n.
\end{equation}
 Over the event $A_n'$, 
we therefore have  $|\varphi_n(x)| \leq \widebar R_n \lesssim D_n^p$ for all $x \in \bbR^d$.
Combined with
Proposition~\ref{prop:phi_bound_superdiff}, we deduce
\begin{align*}
\sup_{\substack{y \in \partial^c\varphi_n(L_j)}}\norm{y} \lesssim \left[\ell_j + D_n^{\frac p {p-1}}\right]
\lesssim (\log n)^{\frac p {\beta(p-1)}} 3^j
\lesssim (3^j\log n)^{q_2},
\end{align*} 
for a large enough exponent $q_2 > 0$.
\end{itemize}
In summary, the following holds over the event $A_n \cap A_n'$, 
uniformly in $j \geq 0$,
$$\sup_{\substack{y \in \partial^c\varphi_n(L_j)}}\norm{y}
 \leq H_{j} := C_2\begin{cases}
3^{jq_1}V_n^{\frac 1 \beta}, & 0 \leq j \leq J_n \\
(3^j\log n)^{q_2}, & j > J_n.
\end{cases}$$ 

{\bf Step 4: Bounding the Lipschitz and Semi-Concavity Moduli of $\varphi_n$.}
Define for a large enough constant $C_3 > 0$,
$$\xi_n(x) =  C_3 \norm x^2\sum_{j=0}^\infty  H_j^p I(x \in I_j),
$$
and set 
$\tilde \varphi_n(x) = \frac 1 \Lambda \varphi_n(x) -  \xi_n(x).$
Under condition~\ref{assm:unbded_growth}, it follows from Lemma~\ref{lem:semi_concave_general} that 
%that $\tilde \varphi_n|_{L_j}$ 
%and 
$\tilde\varphi_n|_{I_{jk}}$ is $C_4 H_j^{p}\ell_j$-Lipschitz
and concave 
for all $j \geq 0$ and $k=1, \dots, m_d$. 
Furthermore, the map $\tilde\varphi_n - \tilde\varphi_n(0)$ is bounded over $L_j$,
and hence also over $I_{jk}$, by $C_5 H_j^{p} \ell_j^2$. Thus, 
there exist sufficiently large exponents $r_i \geq 1$, $1 \leq i \leq 4$, such that if
$M{=}(M_j)_{j=0}^\infty$, $U {=} (U_j)_{j=0}^\infty$,~where
$$M_j = C\begin{cases}
{ 3^{jr_1}} , & 0 \leq j \leq J_n \\
(3^j\log n)^{r_2}, & j > J_n.
\end{cases}, 
\quad U_j = C'\begin{cases}
{ 3^{jr_3}} , & 0 \leq j \leq J_n \\
(3^j\log n)^{r_4}, & j > J_n.
\end{cases},$$
then ${ V_n^{-\frac{p}{\beta}}}\big(\tilde\varphi_n - \tilde\varphi_n(0) \big) \in \calK_{M,U}$, over the event $A_n \cap A_n'$. 

{\bf Step 5: Empirical Process Reduction.}
We deduce from Step 4 that, over $A_n \cap A_n'$,
\begin{align*}
\left|\int \varphi_n d(\mu_n-\mu)  \right|
 &= \Lambda \left| \int \tilde \varphi_n d(\mu_n-\mu) +  \int \xi_n d(\mu_n - \mu)\right|\\
 &\leq \Lambda  \left| \int (\tilde \varphi_n - \tilde\varphi_n(0))d(\mu_n-\mu)\right| + \Lambda \left|\int \xi_n d(\mu_n-\mu)\right|\\
 &\leq  \Lambda { V_n^{\frac{p}{\beta}}} \sup_{f \in \calK_{M,U}} \int f d(\mu_n-\mu)+ \Lambda \left| \int\xi_n d(\mu_n-\mu)\right|.
\end{align*}
Apply the same argument over the event 
$$E_n = A_n \cap A_n' \cap \bigcap_{j=0}^{J_n} \left\{
\inf_{x \in I_{j}}\inf_{\norm{x-y}\leq 2} \nu_n(B_y) \geq \frac{C_1}{2} { \exp(- \ell_j^2/\gamma^2)}\right\}$$
to deduce similarly that ${ V_n^{-p /\beta}}(\tilde \psi_n - \tilde\psi_n(0)) \in \calK_{M,U}$, 
where $\tilde\psi_n(y) = \frac 1 \Lambda \psi_n(y) - \xi_n(y)$, up to increasing the constants $C,C',C_3 > 0$, 
and that $\bbP(E_n^\cp) \lesssim 1/n$. We thus have, over $E_n$, 
\begin{align}
\left|\int \psi_n d(\nu_n-\nu) \right| \leq  \Lambda {V_n^{\frac p \beta}} \sup_{g \in \calK_{M,U}} \int g d(\nu_n-\nu) 
											+\Lambda \left| \int \xi_n d(\nu_n-\nu)\right|.
\end{align}
In the sequel, we write
\begin{align*}
\Delta_n {=} \sup_{f \in \calK_{M,U}} \int f d(\mu_n-\mu) {+} \sup_{g \in \calK_{M,U}} \int g d(\nu_n-\nu),
~
  \calX_n {=} \left|\int \xi_n d(\mu_n-\mu)\right| {+} \left| \int \xi_n d(\nu_n-\nu)\right|,
\end{align*}
so that, 
\begin{align*}
\bbE\left\{I_{E_n} \left[\left|\int \varphi_n d(\mu_n-\mu)\right| + \left| \int \psi_n d(\nu_n-\nu)\right| \right]\right\}
 &{ \lesssim     \Lambda \bbE\left[V_n^{\frac p \beta}\Delta_n\right] + \Lambda \bbE[\calX_n]} \\
 &{ \leq   \Lambda \left(\bbE\Big[V_n^{\frac {2p} \beta}\Big] \bbE\left[\Delta_n^2\right]\right)^{1/2} + \Lambda \bbE[\calX_n]}
\end{align*}
Since $\mu$ and $\nu$ are $(\sigma,\beta)$-sub-Weibull, 
and since $\rho \geq 2p/\beta$, it readily follows from Jensen's inequality that 
$\bbE V_n^{2p/\beta} \leq 2$. Deduce that
\begin{equation}
\label{eq:empirical_proc_red}
\bbE\left\{I_{E_n} \left[\left|\int \varphi_n d(\mu_n-\mu)\right| + \left| \int \psi_n d(\nu_n-\nu)\right| \right]\right\}
\lesssim   \Lambda \sqrt{\bbE\left[\Delta_n^2\right]} + \Lambda \bbE\left[\calX_n\right].
\end{equation} 

{\bf Step 6: Metric Entropy Bound.} Key to bounding ${ \bbE[ \Delta_n^2]}$ is the following
upper bound on the $L^2(\mu_n)$ metric entropy of the class $\calK_{M,U}$. 
\begin{proposition}
\label{prop:covering_unbounded}
{ There exists $C_5 > 0$ such that for all $\epsilon > 0$, }
$$\log N(\epsilon, \calK_{M,U}, L^2(\mu_n)) \leq { C_5 \cdot V_n^{\frac d 4}} \epsilon^{-\frac 2  d}.$$
\end{proposition}
\begin{proof}
Using Lemma~\ref{lem:bronshtein}, we prove the following result in Appendix~\ref{app:proof_peeling}, inspired
by  Corollary 2.7.4 of \cite{vandervaart1996}.
  \begin{lemma}
  \label{lem:covering_peeling}
 For all $\epsilon > 0$,
 $$\log N(\epsilon, \calK_{M,U}, L^2(\mu_n)) \lesssim
\left(\frac 1 \epsilon\right)^{\frac d 2} \left(  \sum_{j=0}^\infty\sum_{k=1}^{m_d} \left(U_j+\diam(I_{jk}) M_j\right)^{\frac{2d}{d+4}}\mu_n(I_{jk})^{\frac d {d+4}}\right)^{\frac {4+d}{4}}.$$ 
 \end{lemma} 
In particular, Lemma~\ref{lem:covering_peeling} implies, %the following holds over the event $E_n$,
\begin{align*}
\log N(\epsilon, \calK_{M,U}, L^2(\mu_n))  
% &\lesssim \left(\frac 1 \epsilon\right)^{\frac d 2} \left(  \sum_{j=0}^\infty\sum_{k=1}^{m_d} \left(U_j+\diam(I_{jk}) M_j\right)^{\frac{2d}{d+4}}\mu_n(I_{jk})^{\frac d {d+4}}\right)^{\frac {4+d}{4}}\\
 &\lesssim\left(\frac 1 \epsilon\right)^{\frac d 2} \left(  \sum_{j=0}^\infty  \left(U_j+3^j M_j\right)^{\frac{2d}{d+4}}\mu_n(I_{j})^{\frac d {d+4}}\right)^{\frac {4+d}{4}}.
\end{align*}
By Markov's inequality, notice that for all $j \geq 0$, 
$$\mu_n(I_{j}) \leq \mu_n(B_{3^{j-1}}^\cp) \lesssim \frac{\int \exp\left(\frac{\norm x^\beta}{2\rho\sigma^\beta}\right)d\mu_n(x)}{\exp\left(3^{(j-1)\beta}/2\rho\sigma^\beta\right)}
 \leq  V_n \exp(-c_1 3^{j\beta}),$$
%Set , and notice that the sub-Weibull assumption on $\mu$
%implies $\bbE[\sigma_n^d] \leq c_d$ for some constant $c_d > 0$. In particular, we obtain that $\mu_n$ is itself 
%sub-Weibull with a parameter depending on $\sigma_n$. 
%  
% 
so that
\begin{align*} 
\log N(\epsilon, \calK_{M,U}, L^2(\mu_n))  
 &\lesssim\left(\frac 1 \epsilon\right)^{\frac d 2} V_n^{\frac d 4}\left(  \sum_{j=0}^\infty  \left(U_j+3^j M_j\right)^{\frac{2d}{d+4}}
 \exp(-c_1 3^{j\beta})\right)^{\frac {4+d}{4}}\\
 &\lesssim \left(\frac 1 \epsilon\right)^{\frac d 2} V_n^{\frac d 4}\Bigg( 
 \sum_{j=0}^{J_n}  \left( { 
 3^{jr_1} + 3^j 3^{jr_3}}\right)^{\frac {2d}{d+4}}\exp(-c_1 3^{j\beta}) \\[-0.1in] 
 	& \quad + 
       \sum_{j=J_n+1}^{\infty}  \left( (3^j\log n)^{r_2} + 3^j(3^j\log n)^{r_4}\right)^{\frac {2d}{d+4}}  
       							\exp(-c_1 3^{j\beta}) \Bigg)^{\frac {4+d}{4}}.
\end{align*}
We deduce that there exist   constants $c_2,c_3 > 0$ such that 
\begin{align*}
\log & N(\epsilon, \calK_{M,U}, L^2(\mu_n)) \\
  &\lesssim \left(\frac 1 \epsilon\right)^{\frac d 2} V_n^{\frac d 4} \left[ 
  \sum_{j=0}^{J_n} 3^{jc_2}\exp(-c_3 3^{j\beta}) + 
  (\log n)^{c_2}\sum_{j=J_n+1}^{\infty} 3^{jc_2 }   \exp(-c_3 3^{j\beta})\right]^{\frac{4+d}{4}}.
 \end{align*}
Notice that $\sum_{j=0}^\infty 3^{jc_2}\exp(-c_3 3^{j\beta}) < \infty$, thus the first summation on the right-hand side of the above display 
is finite. For the second summation, notice that there exists $J_0 > 0$ such that for all $j \geq J_0$, 
$3^{jc_2} \leq \exp(c_4 3^{j\beta})$ where $c_4 = c_3/2$. Thus, since $J_n \asymp \log\log n$, we obtain 
\begin{align*}
\log & N(\epsilon, \calK_{M,U}, L^2(\mu_n))  \\
  &\lesssim \left(\frac 1 \epsilon\right)^{\frac d 2}  V_n^{\frac d 4}\left[1
   +   (\log n)^{c_2}\sum_{j=J_n+1}^{\infty}  \exp(- c_4  3^{j\beta})\right]^{\frac{4+d}{4}}\\
  &\lesssim \left(\frac 1 \epsilon\right)^{\frac d 2} V_n^{\frac d 4} \left[1
   +   (\log n)^{c_2}  \exp\Big(- c_4 (3^{(J_n+1)\beta}+1)\Big)\right]^{\frac{4+d}{4}}
  \lesssim \left(\frac 1 \epsilon\right)^{\frac d 2}  V_n^{\frac d 4}. 
 \end{align*} 
This proves Proposition~\ref{prop:covering_unbounded}.
\end{proof}

{\bf Step 7: Chaining.}
Equipped with Proposition~\ref{prop:covering_unbounded}, 
we are now in a position to bound the expected (squared) 
supremum of the empirical process indexed by $\calK_{M,U}$.
We begin by noting the following.
\begin{lemma}
\label{lem:emp_proc_variance}
It holds that
\begin{align*}
\bbE\left[\left(\sup_{f \in \calK_{M,U}} \int f d(\mu_n-\mu)\right)^2\right] 
\lesssim \frac{(\log n)^{2r_4}}{n}+ \bbE\left[\sup_{f \in \calK_{M,U}} \int f d(\mu_n-\mu)\right]^2.
 \end{align*}
 \end{lemma}

By combining Lemma~\ref{lem:emp_proc_variance} with Lemma~\ref{lem:chaining} and Proposition~\ref{prop:covering_unbounded}, 
we deduce that for all $\tau > 0$, 
\begin{align} 
\nonumber 
\bbE\left[\left(\sup_{f \in \calK_{M,U}} \int f d(\mu_n-\mu)\right)^2\right] 
\lesssim \frac{(\log n)^{2r_4}}{n} + \left(\tau  + \frac {\bbE  V_n^{ \frac d 4}} {\sqrt n} \int_\tau^\infty  \left(\frac {1}{\epsilon}\right)^{\frac d 4} d\epsilon\right)^2.
 \end{align}
Since $\mu$ and $\nu$ are $(\sigma,\beta)$-sub-Weibull, and since 
$\rho \geq d/4$, we again have by Jensen's inequality that
$\bbE[V_{i,n}^{d/4}] \leq 2$
for both $i=1,2$, implying that  
$\bbE[ V_n^{d/4} ]\leq c_5$. 
Choosing $\tau \asymp n^{-2/d}$ in the above display thus leads to a bound scaling at the rate $n^{-{4}/d}$. 
Upon repeating the same argument for $\nu_n$, we obtain
\begin{equation}
\label{eq:main_Detlan_bound}
\sqrt{\bbE[  \Delta_n^2]} \lesssim n^{-2/d}.
\end{equation}
{\bf Step 8: Conclusion.}
Let $(\varphi_0,\psi_0) \in \Phi_c(\mu,\nu)$ be a pair of optimal Kantorovich potentials between $\mu$
and $\nu$. It follows similarly as in the proof of Theorem~\ref{thm:ub_general} that
$$\calT_c(\mu_n,\nu_n) - \calT_c(\mu,\nu) \geq \int \varphi_0d(\mu_n-\mu) + \int \psi_0 d(\nu_n-\nu) =: \Gamma_n.$$
Combine the above display with equations~\eqref{eq:reduction_unbded},
\eqref{eq:empirical_proc_red} and~\eqref{eq:main_Detlan_bound}  to deduce
\begin{align}
\label{eq:conclusion_unbounded}
\nonumber \bbE \big| &\calT_c(\mu_n,\nu_n) - \calT_c(\mu,\nu)\big| \\
\nonumber &\leq \bbE|\Gamma_n| + \bbE\left[ \int \varphi_n d(\mu_n-\mu)\right] +  \bbE\left[ \int \psi_n d(\nu_n-\nu)\right]\\
\nonumber  &\lesssim \bbE|\Gamma_n| + \Lambda\Big\{ \bbE[\calX_n] + {\sqrt{\bbE [\Delta_n^2]}} \Big\} + \bbE\left[I_{E_n^\cp} \int \varphi_n d(\mu_n-\mu)\right] +  
 							\bbE\left[ I_{E_n^\cp}\int \psi_n d(\nu_n-\nu)\right]\\
\nonumber  &\lesssim \bbE|\Gamma_n| +\Lambda\Big\{ \bbE[\calX_n] + 
\sqrt{\bbE [\Delta_n^2]} \Big\} + \sqrt{\bbP(E_n^\cp) \bbE[ \widebar R_n^2
] } \\
 &\lesssim \bbE|\Gamma_n| + \Lambda\Big\{ \bbE[\calX_n] + n^{-2/d} \Big\} + n^{-1/2} (\log n)^{c_1}. 
\end{align}
The quantities $\Gamma_n$ and $\calX_n$ are simple to bound in expectation, as we now show.
\begin{lemma}
\label{lem:root_n_unbounded}
For any $\epsilon > 0$, $\bbE |\Gamma_n|\vee \bbE[\calX_n]\lesssim n^{\epsilon - \frac 1 2}.$
\end{lemma}
Since $d \geq 5$, combining equation~\eqref{eq:conclusion_unbounded} with Lemma~\ref{lem:root_n_unbounded} leads to the claim.\qed 

\subsection{Proof of Theorem~\ref{thm:coupling_quantitative}} 
\label{sec:proof_coupling_quantitative}
Fix $x \in \bbR^d$. 
If $\partial^c \varphi(x)$ is empty, then there is nothing to show, thus
suppose otherwise. Choose $y_x \in \partial^c \varphi(x)$.
Let $K_p =2(3 \kappa^2)^{1/(p-1)}$, and notice that we may assume
\begin{equation}
\label{eq:y_x_assm}
\norm{y_x} \geq 4( K_p+1)(\norm x + 1),
\end{equation}   
as otherwise we are done. 
In particular,  this assumption implies $\norm{ y_x-x} \geq 4$, thus the following ball is non-empty
$$U = \left\{u \in \bbR^d: \norm{u-y_x} \leq  \norm{x-y_x} - 1 \right\}.$$
%If $\supp(\mu)\cap S = \emptyset$, then it must hold that
%$\sup_{\norm{x-y}\leq 2} \left[\log\left(\frac 1 {\mu(B_y)}\right)\right]^{\frac p {\alpha(p-1)}}$,
%and the claim follows, thus assume $\supp(\mu)\cap S \neq \emptyset$, and
Furthermore, define $\xi = (y_x-x)/\norm{y_x-x}$, and note that
the ball $S:= B_{x + 2\xi}$ of radius 1 centered at $x+2\xi$ is contained in $U$.

If $\partial^c\varphi(S) = \emptyset$, then the condition $\supp(\pi) \subseteq \partial^c\varphi$ implies $\mu(S) = 0$, 
in which case the right-hand side of equation~\eqref{eq:coupling_quantitative_part1} is infinite and the claim is trivial. 
Thus, assume otherwise, and 
pick $u \in S$ for which $\partial^c \varphi(u)$ is nonempty.  Notice that $\norm{x-u} \leq 3$.
Furthermore, let $y_u \in \partial^c\varphi(u)$ be arbitrary.
Since $\varphi$ is $c$-concave, the set $\partial^c\varphi$ is $c$-cyclically monotone
by Lemma~\ref{lem:kantorovich_background}(iii). In particular,  
%Combining this fact with condition~\ref{assm:unbded} and with the fact that $\norm{y_x-x} \geq 1$, we obtain
$$c(x,y_x) - c(u, y_x)  \leq c(x,y_u) -c(u,y_u) .$$
Thus, using condition~\ref{assm:unbded_growth},
\begin{align*}
c(x,y_u) - c(u, y_u)
 &\geq \omega(\norm{x-y_x}) - \omega(\norm{u-y_x}) &
 (\text{Since }\norm{x-y_x} \wedge \norm{u-y_x} \geq 1) \\
 &\geq  \omega(\norm{x-y_x}) - \omega (\norm{x-y_x}-1) & (\text{Since } u \in U \text{ and } \omega
 \text{ is increasing)}\\ % for $z > 1$})\\
 &\geq \omega'(\norm{x-y_x}-1) & \text{(By convexity of } \omega) \\
 &\geq \frac 1 \kappa(\norm{x-y_x}-1)^{p-1} & (\text{By condition } \ref{assm:unbded_growth})\\
 &\geq \frac{1}{\kappa 2^{p-1}}\norm{x-y_x}^{p-1}.& (\text{Since }  \norm{x-y_x} \geq 4)
\end{align*}
Now, conditions~\ref{assm:global} and~\ref{assm:unbded_growth} imply that $h(z) \leq \kappa$ for
all $z \in B_{0,1}$. Furthermore, the preceding display combined with equation~\eqref{eq:y_x_assm} implies that $c(x,y_u) > \kappa$, 
whence $\norm{x-y_u} \geq 1$. We may thus again apply conditions~\ref{assm:global} and~\ref{assm:unbded_growth} to obtain,
\begin{align*}
c(x,y_u) &- c(u, y_u)
 \leq \langle \nabla h(x-y_u), x-u\rangle \leq  \omega'(\norm{x-y_u}) \norm{x-u}   \leq 3\kappa \norm{x-y_u}^{p-1}.
% = \left\langle x-u, \frac{\tilde u-y_u}{\|\tilde u-y_u\|}\right\rangle \omega'(\norm{\tilde u-y_u}) 
% \leq \kappa\norm{ x-u} \norm{\tilde u-y_u}^{p-1} \\
% &~~~~~\leq 2^{p-2}\kappa\norm{ x-u} \left(\norm{\tilde u-x}^{p-1} + \norm{x-y_u}^{p-1}\right)
%% &\leq 2^{p-2}\kappa\left(\norm{ x-u}^p  + \norm{x-u}\norm{x-y_u}^{p-1}\right) \\
%\leq  3^p 2^{p-2}\kappa\left(1  +  \norm{x-y_u}^{p-1}\right). 
%% &\leq K_p^{p-1} \left(1 +  \norm{x-y_u}^{p-1}\right),
\end{align*}
%Recalling that $K_p^{p-1} = 3^p 2^{2p-3}\kappa^2$, t
We deduce that
$\norm{x-y_x} \leq K_p  \norm{x-y_u},$
whence,
\begin{align*} 
\norm{y_x} 
 &\leq  K_p \norm{x-y_u} + \norm x  
%\leq K_p (\norm x+ \norm{y_u}) + \norm x \\ 
 \leq  K_p \norm{y_u} + 
 		 \norm x (K_p+ 1) 
 \leq   K_p \norm{y_u} + 
 		 \frac 1 4 \norm{y_x},
\end{align*}
where the last inequality is due to equation~\eqref{eq:y_x_assm}. 
We thus have
$\norm{y_u}\geq C\norm{y_x}$
for a constant $C > 0$ depending only on $d,p,\kappa$. It follows that,
$$\partial^c\varphi(S) \subseteq \left\{v \in \bbR^d: \norm v \geq C \norm{y_x} \right\}.$$
Given $Y \sim \nu$, we deduce from the sub-Weibull condition on $\nu$ that
$$\nu(\partial^c \varphi(S)) \leq \bbP\left(\norm Y \geq C\norm{y_x}\right) 
 \lesssim \exp\left(-\frac{C^\beta \norm{y_x}^{\beta}}{2\sigma^\beta}\right).$$
On the other hand,  using the fact that
$\supp(\pi)\subseteq\partial^c \varphi$, one has  
\begin{align*}
\nu(\partial^c \varphi(S)) = \pi(\bbR^d\times \partial^c \varphi(S)) \geq \pi(S\times\partial^c \varphi(S))
% = \mu(S) - \pi(S\times (\partial^c \varphi(S))^\cp)  
= \mu(S),
\end{align*}
so that, \begin{equation}
\label{eq:pf_displacement_first_claim}
\exp\left(-\frac{C^\beta\norm{y_x}^{\beta}}{2\sigma^{\beta}}\right) \gtrsim \nu(\partial^c \varphi(S))
\geq  \mu(S) \geq \inf_{y:\norm{x-y}\leq 2} \mu(B_y).
\end{equation}
The first claim follows. 
To prove the second claim, 
recall from the definition of $(\gamma,b)$-super-Gaussianity that
for all $y\in \bbR^d$ such that $\norm{x-y}\leq 2$, 
\begin{align*}
\begin{multlined}
\mu(B_{y}) 
%\geq \frac {b}{\sqrt{2\pi\gamma^2}} \int_{B_y} \exp(-\|x\|^2/2\gamma^2)dx \\ \qquad\qquad
\geq \frac {b}{\sqrt{2\pi\gamma^2}} \calL(B_y) \inf_{z \in B_y} \exp(-\|z\|^2/2\gamma^2)
\geq C_1 \exp(-\|x\|^2/C_1),
%c_1 \calL(B_{y})f(0) \inf_{z \in B_y}\exp(-2\gamma_1\norm z^2)
%\geq C_1 \exp(-8\gamma_1\norm x^2),
\end{multlined}
\end{align*}
for a constant $C_1 > 0$ depending on $d,b,\gamma$. 
By Lemma~\ref{lem:kantorovich_background}(iv), since $\calT_c(\mu,\nu) < \infty$, any optimal coupling $\pi$ 
between $\mu$ and $\nu$ lies in the support of a $c$-concave  potential $\varphi$. We may 
therefore apply equation~\eqref{eq:pf_displacement_first_claim} to deduce that, for some constant $C' > 0$, any such coupling satisfies
$$\norm{y} \leq C' \sigma (\norm x + 1)^{\frac{2}{\beta}}, \quad \pi\text{-a.e. } (x,y).$$
Furthermore, since $\mu$ is absolutely continuous with respect to the Lebesgue measure,
notice that the conditions of \cite{gangbo1996}, Theorem 1.2, 
are satisfied under conditions~\ref{assm:global} and~\ref{assm:unbded_growth} and the strict convexity
of $h$. Therefore, there exists a unique optimal transport map $T$ from $\mu$ to $\nu$, so that the measure
$\pi$ in the above display may be taken to be $(Id,T)_\# \mu$. The claim follows.
\qed

\section{Lower Bounds}
\label{sec:lower_bounds}
In this section, we derive two lower bounds which imply that the rates of convergence derived in Sections~\ref{sec:upper_bounds_compact}
and~\ref{sec:upper_bounds_unbounded} are typically unimprovable. In Section~\ref{sec:lower_bounds_empirical}, we obtain
lower bounds on the rate of convergence of the empirical optimal transport cost, 
while in Section~\ref{sec:lower_bounds_minimax}, we derive a minimax lower bound which implies that,
up to polylogarithmic factors, no estimator
of $\calT_c(\mu,\nu)$ can achieve a faster rate of convergence than the empirical estimator uniformly over all pairs
of measures $\mu,\nu$.
 
In order to state our lower bounds, we   
require an assumption on the maximal H\"older exponent $\alpha \in (0,2]$ achievable by the cost $h$. 
To state such an assumption, recall that our upper bounds were based, for instance, on the condition
$\Lambda = 1\vee \norm h_{\calC^\alpha(\calZ_1)} < \infty$, for some 
$\alpha\in(0,2]$, which in particular implies that for all $z,z_0 \in \calZ$,
$$h(z) - h(z_0) \leq 
\begin{cases}
\Lambda \norm{z-z_0}^\alpha, & \alpha \leq  1 \\
\langle \nabla h(z_0), z-z_0\rangle + \Lambda \norm{z-z_0}^\alpha, & \alpha > 1
\end{cases}
.$$ 
We shall assume the following dual condition throughout this section.
 \begin{enumerate}[leftmargin=1.65cm,listparindent=-\leftmargin,label=(\textbf{H4})]   %[label=\textbf{(A1($\delta$)}]      %[\textbf{A1($\delta$)}]
\item \label{assm:lb}  
$\calX$ and $\calY$ are convex sets with nonempty interior, 
and are such that $h$ is differentiable over $\calZ=\calX-\calY$.
Furthermore, there exist
$\lambda > 0$,  $\alpha \in (0,2]$,
and $z_0 = x_0-y_0 \in \calZ$ such that $x_0 \in \mathrm{int}(\calX)$, $y_0 \in \mathrm{int}(\calY)$,
and for all $z \in \calZ$, 
$$h(z) - h(z_0) 
\geq \begin{cases}
\lambda \norm{z-z_0}^\alpha, & \alpha \leq  1 \\
\langle \nabla h(z_0), z-z_0\rangle + \lambda \norm{z-z_0}^\alpha, & \alpha > 1
\end{cases}
.$$

\end{enumerate} 
Notice that condition~\ref{assm:lb} implies that $\calX$ and $\calY$ have positive Lebesgue measure over $\bbR^d$.
The presence of this assumption can be anticipated from the fact that 
empirical optimal transport costs may achieve improved rates of convergence
when $\calX$ and $\calY$ have intrinsic dimension less than $d$ \citep{weed2019}. 
It is straightforward to verify that conditions~\ref{assm:global}
and~\ref{assm:lb} are satisfied by the 
cost $h(x) = \norm x^p$ when  
$1 \leq  p \leq 2$  with $\alpha=p$ and $z_0 = 0$. These conditions are also satisfied for $2 < p < \infty$
and $\alpha=2$, whenever there exists a neighborhood of zero which is not contained in $\calZ$.  
We also note that condition~\ref{assm:lb} is satisfied with $\alpha=2$ by any 
differentiable and $(2\lambda)$-strongly convex  function $h$ over $\calZ$.

Finally, we assume throughout this section
that $X_i$ is independent of $Y_j$ for all $1 \leq i,j \leq n$. 
Though this condition is not needed to derive our upper bounds, we do not preclude 
the possibility that they may be sharpened under particular dependence structures between
the samples from $\mu$ and $\nu$.
\subsection{Lower Bounds for the Empirical Optimal Transport Cost}
 \label{sec:lower_bounds_empirical}

 We begin with the following lower bound on the rate of convergence of the empirical optimal transport cost.

\begin{proposition} 
\label{prop:lb_empirical}
Assume conditions  \ref{assm:global} and \ref{assm:lb}. 
Then, 
$$\sup_{\substack{\mu\in\calP(\calX) \\ \nu\in\calP(\calY)}} \bbE_{\mu,\nu} \big| \calT_c(\mu_n,\nu_n) - \calT_c(\mu,\nu) \big| 
 \gtrsim \lambda n^{-\alpha/d}.$$
\end{proposition}
Proposition~\ref{prop:lb_empirical} implies that the upper bounds in Theorems~\ref{thm:ub_general}, \ref{thm:main_unbounded},
and Corollaries thereafter cannot generally be improved, provided that the H\"older exponent $\alpha$ therein 
is chosen maximally in the sense of condition~\ref{assm:lb}. As we now show, our lower bound is constructive, and 
is typically achieved by absolutely continuous measures differing by a location translation, as in Example~\ref{sec:example_location}.
 \begin{proof} 
We prove the claim assuming $\alpha \in (1,2]$, and an analogous
argument may be used when $\alpha \in (0,1]$. 
Since $x_0 \in \mathrm{int}(\calX)$ and $y_0 \in \mathrm{int} (\calY)$, there exists $\epsilon > 0$
such that $\calX_0:=B_{x_0,\epsilon} \subseteq \calX$ and 
$\calY_0 := B_{y_0, \epsilon} \subseteq \calY$. Define the measures
$$\mu = \frac{\calL|_{\calX_0}}{\calL(\calX_0)}, \quad \nu = \frac{\calL|_{\calY_0}}{\calL(\calY_0)},$$
%
% admit positive Lebesgue measure on $\bbR^d$, there exist compactly-supported
%measures $\mu\in \calP(\calX_0)$, $\nu \in \calP(\calY_0)$, both of which
%are absolutely continuous with respect to the Lebesgue measure on $\bbR^d$, 
where recall that $\calL$ is the Lebesgue measure on $\bbR^d$. 
By construction, $\nu = {T_0}_\# \mu$ where $T_{0}(x) = x+z_0$. 
Since $h$ is convex, it follows by the same argument as in Example~\ref{sec:example_location}
that $T_0$ is an optimal transport map from $\mu$ to $\nu$. 

Let $\gamma_n$ denote an optimal coupling between $\mu_n$ and $\nu$ with respect to the cost $c$. 
Then, by condition~\ref{assm:lb},
\begin{align*}
\calT_c(\mu_n, \nu) - \calT_c(\mu,\nu) 
 &= \int \Big[ c(x,y) - c(x,T_0(x)) \Big] d \gamma_n(x, y) \\
 &= \int \Big[ h(y-x) - h(z_0) \Big] d \gamma_n(x, y) \\
 &\geq \int \Big[ \langle \nabla h(z_0), y-x-z_0\rangle + \lambda \norm{y-x-z_0}^\alpha\Big] d \gamma_n(x, y) \\
 &= \int \Big[ \langle \nabla h(z_0), y-x\rangle + \lambda \norm{y-x}^\alpha\Big] d \pi_n(x, y),
% &= \int \left[ \langle \nabla h(x- T_0(x)), y-T_0(x)\rangle + \frac 1 2 (y-T_0(x))^\top \nabla^2 h(x-t_{x,y}) (y-T_0(x))\right] d \gamma_n(x, y),
\end{align*}
where $\pi_n = (Id, T_0^{-1})_\# \gamma_n \in \Pi(\mu_n,\mu)$. It follows that 
\begin{align*}
\calT_c(\mu_n, \nu) - \calT_c(\mu,\nu) 
 &\geq \int  \langle \nabla h(z_0), y-x\rangle d \pi_n(x, y) + \lambda W_\alpha^\alpha(\mu_n,\mu)\\
 &\geq \int  \langle \nabla h(z_0), \cdot \rangle d (\mu-\mu_n) + \lambda W_1^\alpha(\mu_n,\mu). 
\end{align*}
The first order term on the final line of the above display clearly has mean zero, whence 
\begin{align*}
\bbE \Big[\calT_c(\mu_n, \nu) - \calT_c(\mu,\nu) \Big]
 &\geq \lambda \bbE W_1^\alpha(\mu_n,\mu)\geq \lambda \big[\bbE W_1(\mu_n,\mu)\big]^\alpha
 \gtrsim \lambda n^{-\alpha/d},
\end{align*}
where the final inequality follows from Proposition 2.1 of  \cite{dudley1969}, 
due to the absolute continuity of $\mu$ with respect to the Lebesgue measure on $\bbR^d$. 
Finally, since $h$ is continuous and $\calZ_0=\calX_0 - \calY_0$  is compact, 
$h$ is bounded over $\calZ_0$. Thus, by Lemma~\ref{lem:kantorovich_background}(ii), 
there exists a pair of Kantorovich potentials $(\varphi_n, \psi_n)$ such that
$\calT_c(\mu_n,\nu) = J_{\mu_n,\nu}(\varphi_n,\psi_n)$, whence
$$\calT_c(\mu_n,\nu_n) \geq J_{\mu_n,\nu_n}(\varphi_n,\psi_n) = \calT_c(\mu_n,\nu) + \int \psi_n d(\nu_n-\nu).$$
Since the random variables $X_1, \dots, X_n$ are independent of $Y_1, \dots, Y_n$, 
$\psi_n$ is also independent of $Y_1, \dots, Y_n$, whence  
$$\bbE \left[ \int \psi_n d(\nu_n-\nu) \bigg| X_1, \dots, X_n\right] = 0.$$
It readily follows that $\bbE \calT_c(\mu_n, \nu_n) \geq \bbE \calT_c(\mu_n,\nu)$, so that
$$\bbE \big|\calT_c(\mu_n,\nu_n) - \calT_c(\mu,\nu)\big|
 \geq \bbE \big[\calT_c(\mu_n,\nu_n) - \calT_c(\mu,\nu)\big]
 \geq \bbE \big[\calT_c(\mu_n,\nu) - \calT_c(\mu,\nu)\big] \gtrsim \lambda n^{-\alpha/d}.$$
 The claim follows.
\end{proof}

\subsection{Minimax Lower Bounds}
\label{sec:lower_bounds_minimax}
We next turn to deriving a minimax lower bound on the rate of estimating the optimal transport
cost between two probability measures. Unlike Proposition~\ref{prop:lb_empirical}, 
our next result will require both condition~\ref{assm:lb}
and the smoothness condition~\ref{assm:bded_smoothness} from Section~\ref{sec:upper_bounds_compact}. 
\begin{theorem} %[Minimax Lower Bound]
\label{thm:minimax}
 Assume conditions  \ref{assm:global}, \ref{assm:bded_smoothness} and \ref{assm:lb}. 
Then, there exists a constant $C > 0$ depending on $\lambda,\Lambda,d,\calX,\calY,\alpha$ such that
$$\inf_{\hat \calT_n}\sup_{\substack{\mu \in \calP(\calX) \\ \nu \in \calP(\calY)}} \bbE_{\mu,\nu}\big| \hat \calT_n - \calT_c(\mu,\nu)\big| \geq C (n \log n)^{-\alpha/d},$$
where the infimum is over all Borel-measurable functions $\hat \calT_n$ of $X_1, \dots, X_n$ and $Y_1, \dots, Y_n$. 
\end{theorem}
Theorem~\ref{thm:minimax} shows that the convergence
rates exhibited throughout this paper for the empirical optimal transport cost estimator
cannot be improved by any other estimator uniformly over all pairs
of measures in $\calP(\calX) \times \calP(\calY)$, up to a polylogarithmic factor. 
Minimax lower bounds scaling at the rate $(n\log n)^{-1/d}$ have previously
been established for the problem of estimating $p$-Wasserstein distances
by~\cite{niles-weed2019} (Theorem~11), and we build upon their techniques to prove Theorem~\ref{thm:minimax}.
The proof is deferred to Appendix~\ref{app:proof_minimax}.

\appendix

%{\LARGE\bf Appendix}

%This supplementary material consists of Appendices~\ref{app:proofs_compact}--\ref{app:proof_minimax}, which
%respectively contain all omitted proofs from Sections~2--4 of the main text.

\section{Omitted Proofs from Section~\ref{sec:upper_bounds_compact}}
\label{app:proofs_compact}

\subsection{Proof of Corollary~\ref{cor:lp_norms}}
\label{app:lr_norms}
Throughout the proof, $C > 0$ denotes a constant depending only on $d,p,r$, 
whose value may change from line to line. 

The proof is elementary, but tedious. To prove the first claim, it suffices to show that 
$\norm\cdot_{\ell_r}^p \in \calC^{2\wedge r \wedge p}(B_{0,2})$. 
It is clear that $\norm\cdot_{\ell_r}$
and $\norm\cdot_{\ell_1}^p$ are Lipschitz for any $r,p \geq 1$, 
thus it suffices to assume $p,r > 1$. 
In this case,  
$\norm\cdot_{\ell_r}^p$ is differentiable,  
and for all $l=1,\dots, d$, 
\begin{align}
\label{eq:gradient_lr_norm}
\frac{\partial \norm x_{\ell_r}^p }{\partial x_l}
% = \frac{\partial }{\partial x_l}\left(\sum_{j=1}^d |x_j|^r\right)^{\frac p r} 
% = \frac p r \Big( r x_l |x_l|^{r-2}\Big) \left(\sum_{j=1}^d |x_j|^r \right)^{\frac p r -1}
 =  p    x_l |x_l|^{r-2} \norm x^{p-r}_{\ell_r}.
\end{align}
%It is easy to see from the above display that when 
Next, we show that $\frac{\partial \norm \cdot_{\ell_r}^p }{\partial x_l}$ is 
 H\"older continuous over $B_{0,2}$ 
with suitable exponent, uniformly in $l$. 
Let $x,y \in B_{0,2}$, and assume without loss of generality that $\norm x_{\ell_r} \leq \norm y_{\ell_r}$. Then,
\begin{align*}
\Big| x_l |x_l|^{r-2}& \norm x^{p-r}_{\ell_r} - y_l |y_l|^{r-2} \norm y^{p-r}_{\ell_r}\Big| \\
 &\leq  \left| x_l |x_l|^{r-2} \norm x^{p-r}_{\ell_r} - x_l |x_l|^{r-2} \norm y^{p-r}_{\ell_r}\right|+
    \left| x_l |x_l|^{r-2} \norm y^{p-r}_{\ell_r} - y_l |y_l|^{r-2} \norm y^{p-r}_{\ell_r}\right| \\
 &= |x_l|^{r-1} \left|\norm x^{p-r}_{\ell_r} -  \norm y^{p-r}_{\ell_r}\right|+
    \norm y^{p-r}_{\ell_r} \left| x_l |x_l|^{r-2}  - y_l |y_l|^{r-2} \right|.
\end{align*}
For the first term, notice that 
\begin{align*}
|x_l|^{r-1} \left|\norm x^{p-r}_{\ell_r} -  \norm y^{p-r}_{\ell_r}\right|
  &\leq \norm x_{\ell_r}^{r-1} \left(\norm y^{p-r}_{\ell_r} - \norm x^{p-r}_{\ell_r}\right) \\
 &\leq  \norm y^{p-1}_{\ell_r} - \norm x^{p-1}_{\ell_r} \leq C \norm {y -  x}^{1\wedge (p-1)}.
\end{align*}
%for a constant $C_1 = C_1(d,p,r) > 0$. 
Furthermore, letting $\epsilon_x = \text{sgn}(x_l)$ and $\epsilon_y = \text{sgn}(y_l)$, we have, 
\begin{align}
\label{eq:pf_lr_norm_step}
\nonumber
 \norm y^{p-r}_{\ell_r} & \left| x_l |x_l|^{r-2}  - y_l |y_l|^{r-2} \right|  \\
\nonumber
 &= \norm y^{p-r}_{\ell_r} \left| \epsilon_x |x_l|^{r-1}  - \epsilon_y|y_l|^{r-1} \right| \\
\nonumber
 &\leq  \norm y^{p-r}_{\ell_r}\left(\left|\epsilon_x |x_l|^{r-1} - \epsilon_x |y_l|^{r-1}\right| + \left|\epsilon_x|y_l|^{r-1} - \epsilon_y|y_l|^{r-1}
\right| \right) \\
\nonumber
 &\leq  \norm y^{p-r}_{\ell_r}\left(\left| |x_l|^{r-1} - |y_l|^{r-1}\right| + |y_l|^{r-1} \left|\epsilon_x- \epsilon_y\right| \right) \\
 &\leq  \norm y^{p-r}_{\ell_r}\left(\left| |x_l|^{r-1} - |y_l|^{r-1}\right| + 2|x_l-y_l|^{r-1}   \right).
% &\leq  \norm y^{p-r}_{\ell_r} \left| |x_l|^{r-1} - |y_l|^{r-1}\right| + \|y\|_{\ell_r}^{p-1} \left|\epsilon_x- \epsilon_y\right|.
% &\leq  \norm y^{p-r}_{\ell_r} \left(|x_l|\vee|y_l|\right)^{r-p} \left| |x_l|^{r-1} - |y_l|^{r-1}\right|  + \|y\|_{\ell_r}^{p-1} \left|\epsilon_x- \epsilon_y\right|  \\
% &\leq  \norm y^{p-r}_{\ell_r} |x_l  -  y_l|^{r-p} |x_l  -  y_l|^{1\wedge(p-1)}  + \|y\|_{\ell_r}^{p-1} \left|\epsilon_x- \epsilon_y\right|  \\
 \end{align}
 We now consider several cases. If $p \geq r$, then we readily obtain
 \begin{align*}
 \norm y^{p-r}_{\ell_r} \left| x_l |x_l|^{r-2}  - y_l |y_l|^{r-2} \right| 
 \leq C \left( |x_l-y_l|^{1\wedge(r-1)} +2|x_l-y_l|^{r-1}\right) 
 \leq C   |x_l-y_l|^{1\wedge(r-1)}. 
 \end{align*}
If instead $p < r \leq 2$, then from equation~\eqref{eq:pf_lr_norm_step}, we obtain
 \begin{align*}
 \norm y^{p-r}_{\ell_r} \left| x_l |x_l|^{r-2}  - y_l |y_l|^{r-2} \right| 
 &\leq 3 \norm y^{p-r}_{\ell_r}|x_l-y_l|^{r-1}  \\
 &\leq  3\norm y^{p-r}_{\ell_r}(|x_l| + |y_l|)^{r-p} |x_l-y_l|^{p-1}\leq C|x_l-y_l|^{p-1}.
 \end{align*}
Finally, if $p < r$ and $r > 2$, we have from equation~\eqref{eq:pf_lr_norm_step},
 \begin{align*}
 \norm y^{p-r}_{\ell_r} \left| x_l |x_l|^{r-2}  - y_l |y_l|^{r-2} \right| 
 &\leq   \norm y^{p-r}_{\ell_r}\left( (r-1)(|x_l| \vee |y_l|)^{r-2} |x_l-y_l| + 2|x_l-y_l|^{r-1}   \right) \\
 &\leq  C \norm y^{p-r}_{\ell_r} ( |x_l| + |y_l|)^{r-2}  |x_l-y_l|  \\
 &\leq  C\norm y^{p-r}_{\ell_r}(|x_l| + |y_l|)^{r-p} |x_l-y_l|^{p-1}\leq C|x_l-y_l|^{p-1}.
 \end{align*}
%
% for a constant $C_2 = C_2(d,p,r) > 0$. 
The preceding displays readily imply that for all~$l$, 
$\partial \norm\cdot_{\ell_r}^p/\partial x_l \in \calC^{1 \wedge (r-1)\wedge (p-1) }(B_{0,2})$, 
implying that  $\norm\cdot_{\ell_r}^p \in \calC^{2\wedge r \wedge p}(B_{0,2})$. 
The first claim now follows by applying Theorem~\ref{thm:ub_general}.  To prove
the second claim, notice that when $r \geq 2$, 
the map in equation~\eqref{eq:gradient_lr_norm} is differentiable with respect to $x_l$ over any open subset of $B_{0,2}$
which does not contain the origin, with derivative 
\begin{align*}
\frac{\partial^2 \norm x_{\ell_r}^p}{\partial x_l^2}
 = p\frac{\partial }{\partial x_l}     x_l |x_l|^{r-2} \norm x^{p-r}_{\ell_r} 
 %&= p  (p-r) |x_l|^{2r-2} \norm x^{p-2r}_{\ell_r} +  (r-1)|x_l|^{r-2}\\
% &= p  (p-r) |x_l|^{2r-2} \norm x^{p-2r}_{\ell_r} +  (r-1)|x_l|^{r-2} \norm x_{\ell_r}^{p-r}\\
 &= p  \norm x_{\ell_r}^{p-2r} |x_l|^{r-2} \left[(p-r)  |x_l|^{r}  +  (r-1) \norm x_{\ell_r}^{r}\right].
\end{align*}
%Altogether, 
%\begin{align*}
%\frac{\partial^2 \norm x_{\ell_r}^p}{\partial x_k \partial x_l}
% &= p  \norm x_{\ell_r}^{p-2r} |x_l|^{r-2} \Big[(p-r)  x_lx_k|x_l|^{r-2}  + I(l=k)  (r-1) \norm x_{\ell_r}^{r}\Big]\\
%\end{align*}
Recall that for all positive semidefinite matrices $A \in \bbR^{d\times d}$, the 1-Schatten norm of $A$ is equal to its trace, 
so that
$\norm{A}_{\infty} \lesssim \mathrm{tr}(A)$. Thus, for any $\epsilon > 0$,
\begin{align*}
\sup_{x \in B_{0,2}\setminus B_{0,\epsilon}} \big\|\nabla^2 \norm x_{\ell_r}^p\big\|_{\infty} 
 \lesssim\sup_{x \in B_{0,2}\setminus B_{0,\epsilon}} \sum_{l=1}^d  \left|\frac{\partial^2 \norm{x}_{\ell_r}^p}{\partial x_l^2}\right|
 &\lesssim\sup_{x \in B_{0,2}\setminus B_{0,\epsilon}} \sum_{l=1}^d  \norm x_{\ell_r}^{p-r} |x_l|^{r-2}  < \infty.
\end{align*}
It readily follows that $\norm\cdot_{\ell_r}^p \in \calC^2(B_{0,2}\setminus B_{0,\epsilon}^\circ)$, with H\"older norm depending only on $d,r,p,\epsilon$. 
Now, since $\calX, \calY$ are convex by condition~\ref{assm:S1}, so is the set $\calZ = \calX - \calY$, 
and since $\calX$ and $\calY$ are closed and disjoint, there must 
exist $\epsilon > 0$ such that $B_{0,\epsilon} \cap \calZ = \emptyset$. 
Choose any convex open set $\calZ_1$ containing $\calZ$, and contained
in $B_{0,2} \setminus B_{0,\epsilon/2}^\circ$, to deduce that 
condition~\ref{assm:bded_smoothness} holds with $\alpha = 2$, and with $\Lambda$
depending only on $d,p,r,\calX, \calY$. The claim follows.
\qed 

\subsection{Proof of Lemma~\ref{lem:semi_concave}}
\label{app:semi_convex_compact}
The proof is analogous to that of Proposition~C.2 of~\citep{gangbo1996}, and is included for completeness.
We prove the claim for $\tilde\varphi_n$, noting that a symmetric argument can be used for the map
$\tilde\psi_n$. 
Define the modified cost function 
$$h_\Lambda : z \in \calZ \mapsto h(z) - \frac{\Lambda}{2}\norm z^2.$$
By condition~\ref{assm:bded_smoothness} with $\alpha=2$, $\nabla h$ is $\Lambda$-Lipschitz over $\calZ$,
implying that for all $z_1,z_2 \in \calZ$,
\begin{align*}
\langle \nabla h_\Lambda(z_1) - \nabla h_\Lambda(z_2), z_1-z_2\rangle
 = \langle \nabla h(z_1) - \nabla h(z_2), z_1-z_2\rangle - \Lambda \norm{z_1-z_2}^2 \leq 0.
\end{align*}
It follows that $-\nabla h_\Lambda$ is monotone, whence $h_\Lambda$ is concave
(\cite{hiriart-urruty2004}, Theorem 4.1.4).
Now, notice that for all $x \in \calX$,
\begin{align*}
\tilde\varphi_n(x)
 &= \inf_{y \in \calY} \left\{c(x,y) - \psi_n(y) \right\} - \frac{\Lambda}{2} \norm x^2  \\
 &= \inf_{y \in \calY} \left\{h(x-y) -\frac{\Lambda}{2} \Big[\norm{x-y}^2 - \norm y^2 + 2\langle x, y\rangle\Big] - \psi_n(y) \right\}\\
 &= \inf_{y \in \calY} \left\{h_\Lambda (x-y) +\frac{\Lambda}{2} \Big[\norm y^2 - 2\langle x, y\rangle\Big] - \psi_n(y) \right\}.
\end{align*}  
By concavity of $h_\Lambda$, the last line of the above display  
is an infimum of concave functions of~$x$. It follows that $\tilde \varphi_n$ is concave. 
To prove that $\tilde \varphi_n$ is Lipschitz, let $x \in \calX$
and let $(y_k) \subseteq \calY$ be a sequence such that 
$$\tilde\varphi_n(x) \geq c(x,y_k) - \psi_n(y_k) -\frac{\Lambda}{2}\norm x^2 - k^{-1}.$$
Then,  for all $x' \in \calX$ and $k\geq 1$,
\begin{align*}
\tilde\varphi_n(x') - \tilde\varphi_n(x)
 &\leq \left[c(x',y_k) - \psi_n(y_k) - \frac{\Lambda}{2}\|x'\|^2\right]- 
  	  \left[c(x,y_k) - \psi_n(y_k) -\frac{\Lambda}{2}\norm{x}^2 \right] + k^{-1}\\
 &= h(x'-y_k)   - h(x-y_k)  
  -\frac{\Lambda}{2}\left[\|x'\|^2-\|x\|^2 \right] + k^{-1}\\
 &\leq \left(\sup_{z \in \calZ} \norm{\nabla h(z)} \right) \|x'-x\|
  - \frac{\Lambda}{2}(\|x'\|-\|x\|)(\|x'\|+\|x\|)   + k^{-1} \\
 &\leq \left(\sup_{z \in \calZ} \norm{\nabla h(z)} +  \Lambda \right) \|x'-x\| + k^{-1}
 \leq 2\Lambda \|x'-x\| + k^{-1}.
\end{align*}  
Since $k$ is arbitrary, the Lipschitz property follows upon repeating a symmetric argument to upper bound $\tilde\varphi_n(x) - \tilde\varphi_n(x').$ 
Finally, since %$|f_n| \leq 1$, it is a simple observation that 
$|\varphi_n| \leq 1$, %so that 
$|\tilde\varphi_n| \leq 2\Lambda$ 
as $\Lambda \geq 1$. 
%The claim follows.
\qed

\subsection{Proof of Lemma~\ref{lem:cost_convolution}}
\label{app:proof_lem_cost_conv}

%We shall make use of the following standard result.
%\begin{lemma}
%\label{lem:gaussian_moments}
%Let $X \sim N(\mu, \sigma^2I_d)$ for some $\mu \in B_{0,M}$ for some $M > 0$. Then, for all $p \geq 1$, there exists a constant
%$C_p > 0$ depending on $M$ such that
%$$\bbE \norm X^p \leq C \sigma^p.$$
%\end{lemma}
%{\color{revs}
%\sout{ 
%We prove the claim in the case $\alpha \in (1,2)$, noting that a similar argument may be 
%used when $\alpha=1$ (in which case $h$ need not be differentiable).}}
%\com{$\alpha=1$ is now in case 3 of the main proof, not case 2}
To prove the first part, recall that condition~\ref{assm:bded_smoothness} implies
$h \in \calC^\alpha(\calZ_1)$ with $1 < \alpha < 2$, and $\Lambda \geq \norm h_{\calC^\alpha(\calZ_1)}$.
For any $z \in \calZ$, let $A_z:=\{u \in \bbR^d: z - u \in \calZ_1\}$. 
Since $\calZ$ is compact and $\calZ_1$ is open, there exists $\epsilon > 0$ such that for all $z \in \calZ$, $B_{0, \epsilon} \subseteq A_z$.
In particular, if $\sigma < \epsilon$, then $u \in A_z$ for all $z \in \calZ$ and $u$ in the support of $K_\sigma$.

Moreover, we have by a first-order Taylor expansion
that for all $z \in \calZ$ and $u \in A_z$,
$$h(z-u) - h(z) = - \langle \nabla h(z - t u), u\rangle, $$ 
for some $t \in (0,1)$. By convexity of $\calZ_1$, we have $z-tu \in \calZ_1$. It follows that
$$|h(z-u) - h(z) + \langle \nabla h(z), u\rangle|
\leq | \langle \nabla h(z) - \nabla h(z - t u), u\rangle |
%\leq  \norm{\nabla h(z)-\nabla h(z - t u)} \norm u 
\leq \Lambda \norm u^\alpha.$$ 

Finally, the fact that $K$ is even implies that $\int u K_\sigma(u) du = 0$.
Combining these facts, we obtain
\begin{align}
\label{eq:conv_decomp}
\nonumber
\left| h_\sigma(z) - h(z) \right| 
 &= \left| \int  \big[h(z-u) - h(z) \big] K_\sigma( u)du \right| \\
 &\leq \left| \int \big[h(z-u) - h(z) + \langle \nabla h(z), u \rangle \big] K_\sigma( u)du \right| + \left| \int  \langle \nabla h(z), u \rangle K_\sigma( u)du\right| \\
 & \leq \Lambda \int \norm u^\alpha K_\sigma(u) du 
 \leq \Lambda \sigma^{\alpha}\,,
\end{align}
since the support of $K_\sigma$ lies in $B_{0, \sigma}$.
This proves the first claim.

To prove the second part, it is easy to see that the cost $h_\sigma$ is convex, even, and lower semi-continuous by assumption on $h$, 
thus $h_\sigma$ satisfies assumption~\ref{assm:global}. 
Now, let $\tilde \calZ_1$ be an open set such that $\calZ \subseteq \tilde \calZ_1$ and such that
$\mathrm{cl}(\tilde \calZ_1) \subseteq \calZ_1$. After possibly decreasing the value of $\epsilon > 0$, 
we may again ensure that $B_{z,\epsilon} \subseteq \calZ_1$ for all $z \in \tilde \calZ_1$. 
We shall now prove that $h_\sigma$ satisfies assumption~\ref{assm:bded_smoothness}
with the H\"older norm $\norm{h_\sigma}_{\calC^2(\tilde \calZ_1)} \leq C \Lambda \sigma^{\alpha-2}$ as long as $\sigma < \epsilon$.
That $h_\sigma \leq 1$ on $\tilde \calZ_1$ is immediate, so it suffices to show that $h_\sigma$ has the requisite H\"older norm.

Define for any given $z \in \calZ$
and all $u \in \bbR^d$,
$$\tilde h(u) = h(u) - h(z) - \langle \nabla h(z), u-z\rangle, \quad \tilde h_\sigma = \tilde h\star K_\sigma.$$
As before, for any $z \in \tilde \calZ_1$
and any $u \in \bbR^d$ such that $\norm{u-z} \leq \epsilon$, 
we have $u \in \calZ_1$, whence a first-order Taylor expansion leads to
$$|\tilde h(u)| \leq \Lambda \norm{z-u}^{\alpha}.$$
%On the other hand, one has that for all $u \in \bbR^d$,
%$$\tilde h(u) \leq 2 +  \norm{\nabla h(z)} \norm{u-z} \leq 2 + \Lambda \norm{u-z}.$$
%On the other hand, notice that %for any multi-index $\beta \in \bbN^{2\times d}$, we have
%\begin{align*}
%\|\nabla h_\sigma(z) - h_\sigma(w) \|
% &= \|\nabla \tilde h_\sigma(z) - \tilde h_\sigma(w) \| \\
% &\leq  \int |\tilde h(y)| \left\| \nabla K_\sigma(\|z-u\|) - \nabla K_\sigma(\|w-u\|)\right\| du \\
% &\leq \frac{\Lambda}{\sigma^2}   \int \|z-u\|^\alpha \left\| (z-u) K_\sigma(\|z-u\|) - (w-u) K_\sigma(\|w-u\|)\right\| du \\
%\end{align*}
%
We thus obtain for all $z \in \tilde \calZ_1$,
\begin{align*}
\|\nabla^2 h_\sigma(z)\|_\infty
 &= \|\nabla^2  \tilde h_\sigma(z)\|_\infty  \\
 &\leq \int |\tilde h(u)| \|\nabla^2  K_\sigma(z-u)\|_\infty du \\
 %&\leq \Lambda \int \|z-u\|^\alpha \|\nabla^2  K_\sigma(\norm{z-u})\|_\infty du \\
 & = \sigma^{-d-2} \int |\tilde h(u)| \|\nabla^2 K((z-u)/\sigma)\|_\infty du \\
 & \leq \Lambda \sigma^{-d-2} \int \norm{z-u}^\alpha \|\nabla^2 K((z-u)/\sigma)\|_\infty du \\
 & = \Lambda \sigma^{\alpha -2} \int \norm u^\alpha \|\nabla^2 K(u)\|_\infty du 
  \leq C \Lambda\sigma^{\alpha - 2}\,,
\end{align*}
for some constant $C$ depending only on $K$.
This proves the second claim.\qed

\section{Omitted Proofs from Section~\ref{sec:upper_bounds_unbounded}}
\label{app:proofs_unbounded}
\subsection{Proof of Lemma~\ref{lem:semi_concave_general}}
\label{app:semi_concave_general} 
Under condition~\ref{assm:unbded_smoothness},
recall that for all $R > 0$,  $\norm h_{\calC^2(B_{0,R})} \leq \Lambda R^p$.
Set
$$R = \sup\{\norm{x-y}: x \in B_{0,r}, \  y \in \partial^c \varphi(B_{0,r})\},$$
and let $\Lambda_r = \Lambda R^p$.
It then follows by the same argument as in the proof of Lemma~\ref{lem:semi_concave} that the map
$$h_{\Lambda_r} : z\in B_{0,R} \mapsto   h(z)- \frac  { \Lambda_r} 2 \norm z^2$$
is concave.
Now, the assumptions on $c$ in Lemma~\ref{lem:kantorovich_background}(iv) are satisfied under conditions~\ref{assm:global}, \ref{assm:unbded_smoothness}, 
and under the assumption of superlinearity of $h$, 
thus the assumption of local boundedness on $\varphi$ ensures that
$\partial^c \varphi(x)$ is nonempty for all $x \in B_{0,r}$, and that $\varphi$ admits the representation
%Now, since $\varphi$ is $c$-concave, it holds that for all $x \in B_{0,r}$, 
$$\varphi(x) = \inf_{y \in \partial^c \varphi(B_{0,r})} \Big\{c(x,y) - \varphi^c(y)\Big\}.$$
It follows that
\begin{align*}
\phi(x)
 &= \inf_{y \in \partial^c\varphi(B_{0,r})} \left\{c(x,y) - \varphi^c(y) \right\} - \frac{ \Lambda_r}{2} \norm x^2  \\
 &= \inf_{y \in \partial^c\varphi(B_{0,r})} \left\{h_{\Lambda_r} (x-y) +\frac{\Lambda_r}{2} \Big[\norm y^2 - 2\langle x, y\rangle\Big] - \varphi^c(y) \right\}.
\end{align*}  
Notice that $\norm{x-y} \leq R$ for all $x,y$ appearing in the 
infimum of the final
line in the above display, thus $h_{\Lambda_r}$ is 
defined and concave therein.
Similarly as in Lemma~\ref{lem:semi_concave}, the last line of the above display  
is thus  an infimum of concave functions of $x$, implying that $\phi$ is concave. 
To prove that $\phi$ is Lipschitz, let $x \in B_{0,r}$
and choose a sequence $(y_k) \subseteq \calY$ such that 
$$\phi(x) \geq c(x,y_k) - \varphi^c(y_k) -\frac{\Lambda_r}{2}\norm x^2 - k^{-1}.$$
Then,  for all $x' \in B_{0,r}$ and $k\geq 1$,
\begin{align*}
\phi(x') - \phi(x)
 &\leq \left[c(x',y_k) - \varphi^c(y_k) - \frac{\Lambda_r}{2}\|x'\|^2\right]- 
  	  \left[c(x,y_k) - \varphi^c (y_k) -\frac{\Lambda_r}{2}\norm{x}^2 \right] + k^{-1}\\
 &= h(x'-y_k)   - h(x-y_k)  
  -\frac{\Lambda_r}{2}\left[\|x'\|^2-\|x\|^2 \right] + k^{-1}\\
 &\leq \left(\sup_{z \in B_{0,R}} \norm{\nabla h(z)} \right) \|x'-x\|
  + \frac{\Lambda_r}{2}(\|x'\|-\|x\|)(\|x'\|+\|x\|)   + k^{-1} \\
 &\leq \left(\sup_{z \in B_{0,R}} \norm{\nabla h(z)} +  r\Lambda_r \right) \|x'-x\| + k^{-1}\\
 &\leq 2r{\Lambda_r} \|x'-x\| + k^{-1}.
\end{align*} 
The claim readily follows.
\qed

\subsection{Proof of Lemma~\ref{lem:extension_potentials}}
Part (i) is immediate by definition of $c$-conjugate.
For part~(ii), note that for any $x \in \supp(\mu_n)$, 
\begin{equation}
\label{eq:lem_extension_step}
\varphi_n(x) = \inf_{y \in \bbR^d}\Big\{c(x,y) - \eta_n(y) \Big\}
 \geq \inf_{y \in \bbR^d}\left\{ c(x,y) - \Big[ c(x,y) - f_n(x)\Big]\right\} = f_n(x).
\end{equation}
Similarly, for any $y \in \supp(\nu_n)$, since $(f_n, g_n) \in \Phi_c(\mu_n, \nu_n)$,
\begin{align*}
\psi_n(y) & = \inf_{x \in \bbR^d} \Big\{c(x, y) - \varphi_n(x)\Big\} \\
& \geq \inf_{x \in \bbR^d} \Big\{c(x, y) - \big[c(x, y) - \eta_n(y)\big] \Big\} \\
& = \eta_n(y) \\
& = \inf_{x \in \supp(\mu_n)} \Big\{c(x, y) - f_n(x)\Big\} \wedge \widebar R_n \\
& \geq g_n(y) \wedge \widebar R_n = g_n(y)\,,
\end{align*}
where the final equality uses that $g_n$ maps into $[0, \widebar R_n]$.
Since $\mu_n$ and $\nu_n$ are finitely supported, 
either of the above inequalities is strict if and only~if 
\begin{equation*}
\int \varphi_n d\mu_n + \int \psi_n d\nu_n > \int f_nd\mu_n + \int g_n d \nu_n= 
\calT_c(\mu_n,\nu_n),
\end{equation*} 
 in violation of the optimality of $(f_n, g_n)$.
 Therefore $\varphi_n$ and $\psi_n$ agree with $f_n$ and $g_n$ on $\supp(\mu_n)$ and $\supp(\nu_n)$, and 
%$$\psi_n(y) = \inf_{x \in \bbR^d}\Big\{c(x,y) - \varphi_n(x) \Big\}
% \geq \inf_{x \in \bbR^d}\left\{ c(x,y) - \Big[ c(x,y) - \eta_n(y)\Big]\right\} = g_n(y).$$
% and 
 \begin{equation}
\label{eq:pf_optimality_extension}
\calT_c(\mu_n,\nu_n) = \int \varphi_n d\mu_n + \int \psi_n \nu_n.
\end{equation}
%(iii) This property is well-known from Ruschendorf (see also Theorem 2.7 of Gangbo and McCann).

To prove part (iii), note that $\eta_n$ is nonnegative
over $\bbR^d$,
since $f_n$ is nonpositive over  $\supp(\mu_n)$. 
Therefore, for any $x \in \bbR^d$,  
$$\varphi_n(x) = 
\inf_{y \in \bbR^d} \Big\{ c(x,y) - \eta_n(y)\Big\} \leq h(0)  - \eta_n(x) \leq 0.$$
Furthermore, since $\eta_n$ is bounded above by $\widebar R_n$,
\begin{equation}
\label{eq:varphi_bound}
\varphi_n(x) \geq \inf_{y \in \bbR^d} \Big\{c(x,y) - \widebar R_n\Big\} \geq -\widebar R_n,
\end{equation}
Thus, $|\varphi_n(x)| \leq \widebar R_n$.
Similarly, since $\varphi_n$ is nonpositive, $\psi_n$ is nonnegative, and for all $y \in \bbR^d$,
$$\psi_n(y) = \inf_{x \in \bbR^d} \Big\{ c(x,y) - \varphi_n(x)\Big\} \leq h(0) - \varphi_n(y) \leq \widebar R_n,$$
where we used equation~\eqref{eq:varphi_bound}.
Thus, $|\psi_n(x)| \leq \widebar R_n$ as well.

Finally, to prove part (iv), equation~\eqref{eq:pf_optimality_extension} and the primal definition of $\calT_c(\mu_n,\nu_n) < \infty$ imply 
$$\int \big[c(x,y) - \varphi_n(x) - \psi_n(y)\big] d\pi_n(x,y) = 0.$$
By part (i), the integrand of the above display is nonnegative, thus 
$$c(x,y) = \varphi_n(x) + \psi_n(y), \quad \text{for all } (x,y) \in \supp(\pi_n).$$
Since $\varphi_n$ and $\psi_n$ are bounded by part (iii), 
it must then follow from Lemma~\ref{lem:kantorovich_background}(iv) 
that for any $(x,y)$ satisfying the above display, $(x,y) \in \partial^c \varphi_n(x)$
and $(y,x) \in \partial^c \psi_n(y).$  
\qed

\subsection{Proof of Lemma~\ref{lem:anticonc}}
\label{sec:proof_lem_anticonc}
Let $\calB$ denote the set of all balls in $\bbR^d$.
Recall that $\calB$ has Vapnik-Chervonenkis dimension $d+2$, 
thus the Vapnik-Chervonenkis inequality~\citep{vapnik1968} implies
that for all $u > 0$,
\begin{equation}
\label{eq:vc}
\bbP\left(\sup_{B \in \calB} |\mu_n(B)-\mu(B)|\geq u\right)\lesssim n^{d+2}\exp\left(-\frac{n u^2}{32}\right).
\end{equation} 
By the assumption of $(\gamma,b)$-super-Gaussianity,  %$(\gamma_1,\gamma_2)$-regularity and Lemma~\ref{lem:c1c2_reg}, 
%$\mu$ admits a density
%$f$ satisfying
%$$f(x) \geq c_1 f(0) \exp\big(-(1+\gamma_1)\norm x^2\big), \quad x \in \bbR^d.$$
we have for all $\norm{y-x}\leq 2$,
$$\mu(B_y) \gtrsim 
\int_{B_y} \exp\big(-\norm u^2/(2\gamma^2)\big)du \gtrsim
 \exp\big(-\norm x^2/\gamma^2\big),$$
so that, for all $0 \leq j \leq J_n$, 
$$\inf_{x \in I_{j}}\inf_{\norm{x-y}\leq 2}\mu(B_y) \geq C_1 \exp(-\ell_j^2/\gamma^2).$$ 
Thus setting $u=C_1 \exp(-\ell_j^2/\gamma^2)/2$ in equation~\eqref{eq:vc} for all $0 \leq j \leq J_n$, and applying a union bound, leads to 
\begin{equation}
\label{eq:anticonc_hpb}
\bbP\left(A_n^c\right) 
  \lesssim n^{d+2} J_n\exp\left\{-\frac{C_1^2}{128} n  \exp(-2\ell_{J_n}^2/\gamma^2)\right\}
  = n^{d+2} J_n \exp\left\{-\frac{C_1^2\sqrt n}{128}\right\} \lesssim \frac 1 n.
\end{equation}
The claim follows.
\qed

\subsection{Proof of Proposition~\ref{prop:phi_bound_superdiff}}  
The proof proceeds using a similar argument as that of Proposition~C.4 of~\cite{gangbo1996}.
Under conditions~\ref{assm:global} and~\ref{assm:unbded_growth}, 
it follows from Lemma~\ref{lem:kantorovich_background}(iv)
that  $\partial^c \varphi(x)$ is nonempty for all $x \in B_{r/2}$. For any $y \in \partial^c \varphi(x)$, we have
$$\varphi(x)  = c(x, y) - \varphi^c(y).$$ 
Let $v = x - y$. If $\norm{v}\leq r$ there is nothing to prove, so assume otherwise, 
and define $\xi = 1-\frac r {2\norm{v}}$. Our assumption implies that $\xi \in [1/2,1]$. Furthermore, define
$$u = x + (\xi-1)v = x - \frac r 2 \left( \frac{v}{\norm {v}}\right).$$
Then, the penultimate display leads to
$$h(v) - h(\xi v) = c(x, y) - c(u, y) 
= c(x, y) -\varphi^c( y) - [c(u, y) - \varphi^c(y)]
\leq \varphi(x) - \varphi(u)\leq 2R.$$
This fact, together with the convexity and differentiability of $h$ away from zero, under condition~\ref{assm:unbded_growth}, implies
$$\frac r 2 \langle\nabla h(\xi v),  v/\norm{ v}\rangle \leq 2R.$$
On the other hand, by condition~\ref{assm:unbded_growth} we have $h(0)=0$, thus by convexity of $h$,
%\begin{align*}
%\frac{h(\xi v)}{\norm{\xi  v}} 
% = \frac{ h(\xi v)-1}{\norm{\xi v}} \leq \left\langle \nabla h(\xi v), \frac{\xi v}{\norm{\xi v}}-\frac v {\xi\norm v^2}\right\rangle 
% \leq \frac{4R}{r} - \frac {\omega'(\xi\norm v)} {\xi}\leq \frac{4R}{r}.
%% &=  \left\langle \omega'(\xi\norm v) \frac{v}{\norm v}, \frac{\xi v}{\norm{\xi v}}-\frac v {\xi\norm v^2}\right\rangle\\
%% &= \omega'(\xi\norm v) \left( 1-\frac 1 {\xi\norm v}\right)\\ 
%% &\geq \frac 1 2 kappa \norm{\xi v}^{p-1}\\ 
%\end{align*}
\begin{align*}
 \frac{ h(\xi v)}{\norm{\xi v}} \leq \left\langle \nabla h(\xi v), \frac{\xi v}{\norm{\xi v}}\right\rangle 
 \leq \frac{4R}{r}. 
\end{align*}
In particular, since $h(z)\gtrsim \kappa^{-1} \norm z^p$ for all $\norm z \geq 2$
under condition~\ref{assm:unbded_growth}, we have
$\norm{\xi  v}^{p-1}  \lesssim 4R/r,$
thus since $\xi\geq 1/2$ and $r \geq 1$, 
$\norm{v}^{p-1}  \lesssim R,$
and hence 
$$\norm{y}^{p-1} \lesssim \norm{ x}^{p-1} + R.$$
The claim follows.
\qed

 \subsection{Proof of Lemma~\ref{lem:covering_peeling}}
 \label{app:proof_peeling}
Let $\bar M_{jk} = M_j$ and $\bar U_{jk} = U_j$
for all $j \geq 0$ and $k=1, \dots, m_d$. 
 Fix an enumeration $D_1, D_2, \dots$ (resp. $\bar M_1, \bar M_2, \dots$
 and $\bar U_1, \bar U_2, \dots$) of the set $\{I_{jk}: j \geq 0, 1 \leq k \leq m_d\}$
 (resp. $(\bar M_{jk})$, $(\bar U_{jk})$). 
 Given a sequence $(a_j)_{j=1}^\infty$ of positive real numbers, let
 $p_j = N(\epsilon a_j, \calF_{\bar M_j, \bar U_j}(D_{j}), L^\infty)$  and let 
 $f_{j,1}, \dots, f_{j,p_j}$ be a $\epsilon a_j$-cover for $\calF_{\bar M_j,\bar U_j}(D_j)$
 in $L^\infty$. By Lemma~\ref{lem:bronshtein}, we have
 $$\log p_j \lesssim \left(\frac{\bar U_j+\diam(D_j)\bar M_j}{\epsilon a_j}\right)^{\frac d 2}.$$
% \lesssim \left(\frac{\bar U_j}{\epsilon a_j}\right)^{\frac d 2} + \sqrt{\lambda(D_j)}\left(\frac{\bar M_j}{\epsilon a_j}\right)^{\frac d 2},$$
 Now, it can be directly verified that the set
 $$\left\{\sum_{j=1}^\infty f_{j,k_j}:  k_j\in\{1, \dots, p_j\}, j\geq 0\right\}$$
 forms an $\epsilon \left(\sum_{j=1}^\infty a_j^2 \mu_n(D_j)\right)^{1/2}$-cover
 of $\calK_{M,U}$ in $L^2(\mu_n)$, which is of size $\prod_{j=1}^\infty p_j$. Thus,
 $$\log N\left(\epsilon\left[  \sum_{j=1}^\infty a_j^2 \mu_n(D_j)\right]^{1/2}, \calK_{M,U}, L^2(\mu_n)\right) 
 \leq \sum_{j=1}^\infty \log p_j  
    \lesssim \sum_{j=1}^\infty \left(\frac{\bar U_j+\diam(D_j)\bar M_j}{\epsilon a_j}\right)^{\frac d 2}.$$
    Now, set
$$a_j %=  \Big[\big(\bar U_j+\diam(D_j)\bar M_j\big)^{\frac d 2}/\mu_n(D_j)\Big]^{\frac {2}{d   + 4}}
 = \big(\bar U_j+\diam(D_j)\bar M_j\big)^{\frac d {d+4}}\mu_n(D_j)^{-\frac {2}{d   + 4}}.$$ 
Then
$$\sum_{j=1}^\infty \left(\frac{\bar U_j+\diam(D_j)\bar M_j}{ a_j}\right)^{\frac d 2}
 = \sum_{j=1}^\infty \left(\bar U_j+\diam(D_j)\bar M_j\right)^{\frac{2d}{d+4}}\mu_n(D_j)^{\frac d {d+4}},$$
and,
\begin{align*}
\sum_{j=1}^\infty a_j^2 \mu_n(D_j) 
% &= \sum_{j=1}^\infty\big(\bar U_j+\diam(D_j)\bar M_j\big)^{\frac {2d} {d+4}}\mu_n(D_j)^{\frac {d+2}{d   + 4}} \\
 &\leq \sum_{j=1}^\infty \left(\bar U_j+\diam(D_j)\bar M_j\right)^{\frac{2d}{d+4}}\mu_n(D_j)^{\frac d {d+4}}.
\end{align*}
 We deduce that for all $\epsilon > 0$,
\begin{align*}
\log N\left(\epsilon, \calK_{M,U}, L^2(\mu_n)\right) 
 &\lesssim \left(\sum_{j=1}^\infty a_j^2 \mu_n(D_j)\right)^{\frac d 4}\sum_{j=1}^\infty \left(\frac{\bar U_j+\diam(D_j)\bar M_j}{ \epsilon a_j}\right)^{\frac d 2}\\
 &\lesssim \left(\frac 1 \epsilon\right)^{\frac d 2} \left(  \sum_{j=1}^\infty \left(\bar U_j+\diam(D_j)\bar M_j\right)^{\frac{2d}{d+4}}\mu_n(D_j)^{\frac d {d+4}}\right)^{\frac {4+d}{4}}.
\end{align*}
The claim follows.
\qed 
 
\subsection{Proof of Lemma~\ref{lem:emp_proc_variance}}
%To prove the claim, it suffices to show that
%$$\calV := \Var\left[ \sup_{f \in \calK_{M,U}} \int f d(\mu_n-\mu)\right] \lesssim \frac{\log n}{n}.$$
To prove the claim, it suffices to show that  the variance
of the supremum of the empirical process is of the order $(\log n)^{2r_4}/n$. By \citet[Theorem 11.1]{boucheron2013}, it holds that
\begin{align*}
\Var\left[ \sup_{f \in \calK_{M,U}} \int f d(\mu_n-\mu)\right]
% &\leq \frac 2 {n^2} \bbE \left[\sup_{f \in \calK_{M,U}} \sum_{i=1}^n (f(X_i) - \bbE f(X_i))^2\right]  \\
 &\leq \frac 1 {n^2} \sum_{i=1}^n  \bbE\left[  \sup_{f \in \calK_{M,U}} \big(f(X_i) - \bbE f(X_i)\big)^2\right] \\
 &\lesssim \frac 1 n \bbE\left[ \sup_{f \in \calK_{M,U}}f^2(X_1)\right] \\ 
 &= \frac 1 n \sum_{j=0}^\infty \int_{L_j} \left(\sup_{f \in \calK_{M,U}}f^2(x)\right) d\mu(x)\\
 &\leq \frac 1 n \sum_{j=0}^\infty U_j^2 \mu(L_j)\\ 
 &\lesssim \frac 1 n \sum_{j=0}^\infty (3^j \log n)^{2r_4} \exp(-c_1 3^{  j \beta}) \\
 &\lesssim \frac {(\log n)^{2r_4}} n.
\end{align*}
The claim readily follows.\qed 

%Let $X_1', \dots, X_n' \sim \mu$ be an i.i.d. sample independent of $X_1, \dots, X_n$. 
%Let 
%$\mu_{ni}' = \frac 1 n \left(\delta_{X_i}'+ \sum_{j\neq i} \delta_{X_j} \right)$, and
%$$R_n = \sup_{f \in \calK_{M,U} } \int f d(\mu_n-\mu), \quad R_{ni} = \sup_{f \in \calK_{M,U} }\int fd(\mu_{ni}'-\mu), \quad 
%i=1, \dots, n.$$
%Then, by the Efron-Stein inequality, 
%\begin{align*}
%\Var[R_n] \leq \sum_{i=1}^n \bbE(R_n-R_{ni})_+^2
%\end{align*}
%
%Notice that for all $f \in \calK_{M,U}$, and $X \sim \mu$, 
%$$\|f\|_{L^2(\mu)}^2 = \sum_{j=0}^\infty \int_{L_j} |f(x)|^2 d\mu(x) 
%\lesssim \sum_{j=0}^\infty U_j \bbP(\|X\| \geq 3^j) \lesssim (\log n)^{r_4} \sum_{j=0}^\infty 3^{r_4 j} e^{-3^{\frac j \beta}}
%\lesssim (\log n)^{r_4}.$$
%
%\qed 

\subsection{Proof of Lemma~\ref{lem:root_n_unbounded}}
\label{pf:root_n_unbounded}
We begin with $\bbE|\Gamma_n|$. 
Since the quantity $\int \varphi_0 d(\mu_n - \mu)$ remains unchanged if a constant is added to the map $\varphi_0$, 
there is no loss of generality in assuming $\varphi_0(0)=0$. By Theorem~\ref{thm:coupling_quantitative} and Lemma~\ref{lem:semi_concave_general} it must then follow that $\varphi_0(x) \lesssim  1 + \norm{x}^q$ for a sufficiently
large exponent $q\geq 1$, implying that
$$\bbE[\varphi_0(X)^2] \lesssim 1 + \bbE[\norm X^{2q}] \leq C,$$
where the final inequality holds because $\mu$ is $(\sigma,\beta)$-sub-Weibull, and thus admits $(2q)$-th moment
bounded above by a constant depending only on $q$, $\sigma$ and $\beta$, for all $q \geq 1$.
Therefore, by Markov's inequality, 
$$\bbE\left| \int \varphi_0 d(\mu_n-\mu)  \right|
 = \int_0^\infty \bbP\left(\left|\int \varphi_0 d(\mu_n-\mu)\right| \geq u\right)du
 \leq n^{-1/2} + \int_{n^{-\frac 1 2}}^\infty \frac{C}{nu^2}du \lesssim \frac 1 {\sqrt n}.$$
Applying a similar argument to $\psi_0$ leads to $\bbE | \Gamma_n| \lesssim n^{-1/2}$. 

Turning to $\calX_n$, notice that $|\xi_n(x)| \lesssim (\log n \norm x)^{q'}$ for all $x \in \bbR^d$, for
a sufficiently large constant $q' > 0$, thus it follows similarly as before that
$\bbE[\calX_n] \lesssim (\log n)^{q'} n^{-1/2} \lesssim n^{\epsilon - \frac 1 2 }$ for any $\epsilon > 0$. 
\qed 

\subsection{Proof of Corollary~\ref{cor:norm_p_unbounded}}
\label{app:pf_norm_p_unbounded}
The claim is straightforward when $p \geq 2$.
To prove the claim when $p \in (1,2)$, abbreviate $h_p(x) = \norm x^p$, and for all $\epsilon \in [0,1]$ define the cost 
$$h_{p,\epsilon}(x) = \Big( \norm x^2 + \epsilon^{\frac 2 p}\Big)^{\frac p 2} - \epsilon.$$
\begin{lemma}
\label{lem:approximation_unbded}
We have for all $p \in (1,2)$ and all $\epsilon\in[0,1]$, 
\begin{enumerate}
\item $\norm{h_{p,\epsilon} - h_p}_{L^\infty} \leq 2\epsilon.$
\item $h_{p,\epsilon}$ satisfies condition~\ref{assm:unbded_smoothness}
with $\norm h_{\calC^2(B_{0,r})} \leq \Lambda_\epsilon r^p$ for all $r \geq 1$, where $\Lambda_\epsilon = c_1\epsilon^{1-\frac 2 p}$
for a universal constant $c_1 > 0$. Furthermore, $h_{p,\epsilon}$ satisfies 
condition~\ref{assm:unbded_growth} with $\kappa = 2^{\frac p 2 -1}p$. 
%\item $\norm{c_{p,\epsilon}}_\infty \lesssim \epsilon^{1-\frac 2 p} \norm x^p.$
%$$\nabla^2 c_{p,\lambda}(x) \preceq 2\lambda^{1-\frac 2 p}  I_d.$$
\end{enumerate}
\end{lemma}
Lemma~\ref{lem:approximation_unbded}(i) implies
\begin{align*}
\bbE \big|\calT_{h_p}(\mu_n,\nu_n) - \calT_{h_p}(\mu,\nu)\big|
 &\leq \bbE \big|\calT_{h_{p,\epsilon}}(\mu_n,\nu_n) - \calT_{h_{p,\epsilon}}(\mu,\nu)\big| + 4\epsilon,
\end{align*}
which together with Lemma~\ref{lem:approximation_unbded}(ii) and
Theorem~\ref{thm:main_unbounded}
imply
\begin{equation*}
\bbE \big|\calT_{h_p}(\mu_n,\nu_n) - \calT_{h_p}(\mu,\nu)\big|
 \lesssim  \epsilon^{1-\frac 2 p} n^{-\frac 2 d} + \epsilon.
\end{equation*}
The right-hand side is minimized by choosing $\epsilon \asymp n^{-p/d}$, leading to the claim. It thus remains to prove
Lemma~\ref{lem:approximation_unbded}. 

\subsubsection{Proof of Lemma~\ref{lem:approximation_unbded}}
Notice  that for all $x \in \bbR^d$,
\begin{align*}
\left| h_{p,\epsilon}(x) - h_p (x)\right|
  = \left| \Big( \norm x^2 + \epsilon^{\frac 2 p}\Big)^{\frac p 2} - \epsilon - \norm x^p  \right|
  \leq  \Big( \norm x^2 + \epsilon^{\frac 2 p}\Big)^{\frac p 2} - \norm x^p 
 +   \epsilon  \leq  2\epsilon,
 \end{align*}
thus part (i) follows. To prove part (ii), choose the function
$\omega(z) = (z^2 + \epsilon^{2/p})^{p/2}  - \epsilon$. We have 
$h(0)= 0$, and for all $z >1$, 
$$\omega'(z) = p (z^2 + \epsilon^{2/p})^{\frac p 2 -1} z,$$
so that 
%\leq p z^{p-1},$$
%and since $z > 1 \geq \epsilon$,
%$\omega'(z) \geq  p(2z^2)^{\frac p 2 -1} z = C_a z^{p-1}$.
$h_{p,\epsilon}$   satisfies condition~\ref{assm:unbded_growth} with 
$\kappa = 2^{1-\frac p 2}p$. 
It remains to prove the H\"older estimate. 
Clearly, $h_{p,\epsilon} \in \calC^2_{\mathrm{loc}}(\bbR^d)$ for all
$\epsilon > 0$, and
%Now, 
%$$\big\|\nabla h_{p,\epsilon}(x)\big\|
% \lesssim  p \norm x \Big( \norm x^2 + \epsilon^{\frac 2 p}\Big)^{\frac p 2 -1}
% \lesssim \begin{cases}
% 	\epsilon^{1-\frac 1 p} , & \norm x \leq \epsilon^{\frac 1 p} \\
% 	\norm x^{p-1}, & \norm x > \epsilon^{\frac 1 p}.
% \end{cases}$$
%Further,  
$$\nabla^2 h_{p,\epsilon}(x) = p (p-2)   \Big( \norm x^2 + \epsilon^{\frac 2 p}\Big)^{\frac p 2 -2} xx^\top
 +  p \Big( \norm x^2 + \epsilon^{\frac 2 p}\Big)^{\frac p 2 -1} I_d.$$
Therefore, 
\begin{align*}
\norm{\nabla^2 h_{p,\epsilon}(x)}_{\text{op}} 
 &\lesssim  \Big( \norm x^2 + \epsilon^{\frac 2 p}\Big)^{\frac p 2 -2} \norm x^2
 +    \Big( \norm x^2 + \epsilon^{\frac 2 p}\Big)^{\frac p 2 -1}  
 \lesssim \begin{cases}
 	\epsilon^{1-\frac 2 p} , & \norm x \leq \epsilon^{\frac 1 p}, \\
 	\norm x^{p-2}, & \norm x > \epsilon^{\frac 1 p}.
 \end{cases}
\end{align*} 
We thus easily deduce that for all $r \geq 1$,
$$\norm{h_\epsilon}_{\calC^2(B_{0,r})} 
\lesssim r^p + \epsilon^{1-\frac 2 p} \leq\epsilon^{1-\frac 2 p} r^p,$$
and the claim follows. 
\qed

\subsection{On the Super-Gaussianity Assumption}
\label{app:super_gaussian}
We close this Appendix with a simple characterization of super-Gaussianity  which
was stated in Section~\ref{sec:upper_bounds_unbounded}. 
Recall that we say a measure $\mu$ is $(\gamma,b)$-super-Gaussian if  
$\mu(B_x) \geq b \cdot \bbP(Z \in B_x)$ for any $x \in \bbR^d$, where $Z \sim N(0,\gamma^2)$. 
Furthermore, we say that $\mu$ admits a $(\gamma_1,\gamma_2)$-regular density~\citep{polyanskiy2016}
for some $\gamma_1,\gamma_2 > 0$ 
if $\mu$ admits a density $f$ with respect to the Lebesgue measure such that 
%\com{I added this appendix as part of our answer to comment \#2 (the lemma is a reformulation of a result we used to have in Section 3,
%which used to be proven in App B.2)}
$\log f$ is differentiable and satisfies
$$\norm{\nabla \log f(x)} \leq \gamma_1 \norm x + \gamma_2,\quad
\text{for all } x \in \bbR^d.$$
\begin{lemma}
\label{lem:c1c2_reg}
Assume $\mu$ admits a $(\gamma_1,\gamma_2)$-regular density. Then, there exist
constants $\sigma,b > 0$ such that $\mu$ is $(\sigma,b)$-super-Gaussian.
%$c_1 > 0$ depending on $\gamma_2$ such that
%$$f(x) \geq c_1 f(0)\exp\left(-(1+\gamma_1)\norm x^2\right),$$ 
%for all $x \in \bbR^d$. %One may take $c_1 = \exp(-\gamma_2^2)$.
%%$$\sup_{\norm{x-y}\leq 1} \mu(B_y) \geq \exp($$
\end{lemma}

\subsubsection{Proof of Lemma~\ref{lem:c1c2_reg}}
\label{app:proof_c1c2_reg}
By a first-order Taylor expansion, we have for all $x \in \bbR^d$,
$$|\log f(x)  -\log f(0)| = | \nabla \log f(\tilde x)^\top x|,$$
for some $\|\tilde x\| \leq \norm x$. Therefore, 
$$|\log f(x)  -\log f(0)| \leq \|\nabla \log f(\tilde x)\| \|x\|\leq \gamma_1 \norm x^2 + \gamma_2 \norm x,$$
which entails
$$f(x) \geq \exp\Big\{ \log f(0) - \gamma_1 \norm x^2 - \gamma_2 \norm x\Big\}
 = f(0)\exp\Big\{ - \gamma_1 \norm x^2 - \gamma_2 \norm x\Big\}.$$
If $\norm x \geq \gamma_2$, the above display is bounded below by
$f(0)\exp(-(1+\gamma_1)\norm x^2)$, while if $\norm x \leq \gamma_2$, 
it is bounded below by $f(0)\exp(-\gamma_1\norm x^2 + \gamma_2^2)$. 
In either case, $f$ is bounded below by a constant multiple of
the $N(0,\gamma_1^{-1})$ density, thus the claim readily follows. \qed

\section{Omitted Proofs from Section~\ref{sec:lower_bounds}}
\label{app:proof_minimax}
\subsection{Proof of Theorem~\ref{thm:minimax}}
Throughout this section, we respectively denote the $\chi^2$-divergence and the Total Variation distance between
two probability measures $P \ll Q$ by 
$$\chi^2(P, Q) = \int\left(\frac{dP}{dQ}-1\right)^2 dQ, \quad 
\TV(P,Q) = \frac 1 2 \int \left|\frac {dP}{dQ} - 1\right| dQ.$$
%Furthermore, the symmetric group over $[m] = \{1, \dots, m\}$ is denoted $S_m$ for all integers $m \geq 1$.
Furthermore, similarly as in the proof of Proposition~\ref{prop:lb_empirical}, we write
$T_0(z) = z+z_0$, where $z_0$ is defined in condition~\ref{assm:lb}.

Our proof of Theorem~\ref{thm:minimax} follows similarly as  that of \cite{niles-weed2019}, Theorem 11, which establishes a minimax
lower bound for estimating $p$-Wasserstein distances.
Our key extension of their proof technique is contained in the following result. 
\begin{proposition}
\label{prop:w-tv-chi}
Assume the same conditions as Theorem \ref{thm:minimax}. 
Given an integer $m \geq 1$, let $u$ be the uniform distribution on $[m]$. Then, there exist 
universal constants $C_1, C_2> 0$, 
a constant $C_{\lambda,\Lambda} > 0$ depending on $\lambda,\Lambda,\alpha$, and a random
function $F:[m] \to \calX$, such that for any distribution $q$ on $[m]$, we have
\[
\begin{multlined}[\linewidth]
C_1 \lambda m^{-\alpha/d}\TV(q, u) - C_{\lambda,\Lambda} \sqrt{\frac{\chi^2(q,u)}{m}}   \leq \calT_c\big(F_\# q, (T_0 \circ F)_\# u\big) - h(z_0)  
 \\
\leq 
C_2 \Lambda m^{-\alpha/d} \big(\chi^2(q, u)\big)^{\frac \alpha d} \TV(q, u)^{1 - \frac {2\alpha} d}
+ C_{\lambda,\Lambda} \sqrt{\frac{\chi^2(q,u)}{m}},
\end{multlined} 
\]
with probability at least $.9$.
\end{proposition}
\begin{proof}
We prove the claim for $\alpha \in (1,2]$. An analogous
argument may be used to prove the claim when $\alpha \in (0,1]$. 
Recall the notation of condition~\ref{assm:lb}.
Similarly as in the proof of Proposition~\ref{prop:lb_empirical},
there exists $\gamma > 0$ such that $\calX_0 = B_{x_0,\gamma} \subseteq \calX$ and such that $\calY_0 = T_0(\calX_0) \subseteq \calY$, 
where $T_0(z) = z + z_0$. 
%
%
% $x_0 \in \mathrm{int}(\calX)$ and $y_0 \in \mathrm{int} (\calY)$, 
%such that $\calX_0:=B_{x_0,\epsilon} \subseteq \calX$ and 
%$\calY_0 := B_{y_0, \epsilon} \subseteq \calY$. Define the measures
%$$\mu = \frac{\calL|_{\calX_0}}{\calL(\calX_0)}, \quad \nu = \frac{\calL|_{\calY_0}}{\calL(\calY_0)},$$ 
%and notice by construction that $\nu = T_{0} \# \mu$ where $T_{0}(x) = x+z_0$. 
%Since $h$ is convex, it follows by the same argument as in Example~\ref{sec:example_location}
%
%The assumption that $\calX$ admits nonempty interior implies that 
Now, it is a straightforward observation that $N(\epsilon, \calX_0, \norm\cdot ) \geq c' \epsilon^{-d}$,
for all $\epsilon \in (0,1)$ and for a constant $c' > 0$ depending only on $d,\gamma$, 
which implies that the $\epsilon$-packing number of $\calX$ under $\norm\cdot$ is also greater than
$c'\epsilon^{-d}$~(\cite{wainwright2019}, Lemma 5.5). 
Therefore, there exists a set $\calG_m = \{x_1, \dots, x_m\} \subseteq \calX_0$ 
such that $\norm{x_i - x_j} \gtrsim m^{-1/d}$ for all $i \neq j$. 
We let $F$ be selected uniformly at random from the set of bijections $S_m$
from $[m]=\{1,\dots,m\}$ to $\calG_m$.

We begin by proving the lower bound. 
Let $\tilde \pi$ denote an optimal coupling between $F_\# q$ and $(T_{0} \circ F)_\# u$,
and let $ \pi = (Id, T_0^{-1})_\# \tilde \pi \in \Pi(F_\# q,F_\# u)$. 
We then have,
\begin{align}
\label{eq:pf_minimax_taylor}
\nonumber
\calT_c(F_\# q, (T_0 \circ F)_\# u) - h(z_0)
 &= \int\Big[h(y-x) - h(z_0)\Big] d\tilde \pi(x, y) \\
\nonumber
 &= \int \Big[ h(y-x+ z_0) - h(z_0)\Big] d\pi(x, y) \\
 &\geq \int \langle \nabla h(z_0), y-x \rangle d \pi(x,y) +  \lambda  \int \norm{x-y}^\alpha d\pi(x,y),
 \end{align}
by condition~\ref{assm:lb}. 
We next bound the term $\Gamma_m = \int \langle \nabla h(z_0), y-x \rangle d \pi(x,y)$.
Notice that \begin{align*}
\bbE_F [\Gamma_m]
 &= \left\langle \nabla h(z_0), \frac 1 {m!} \sum_{G\in S_m} \sum_{j=1}^m (u(j) - q(j)) G(j)\right\rangle \\
 &= \left\langle \nabla h(z_0), \sum_{j=1}^m (u(j) - q(j)) \left(\frac 1 {m!}\sum_{G\in S_m} G(j)\right)\right\rangle.
\end{align*}
%where the summation of $F$ is taken over all bijections from $[m]$ to $\calG_m$. 
The quantity $\frac 1 {m!}\sum_{G\in S_m} G(j)$ takes on the same
value for all $j=1, \dots, m$, thus %since $u$ and $q$ have the same mean, 
we deduce from the above display that $\bbE_F[\Gamma_m] = 0$. 
Similarly, notice that for any $1 \leq j\neq k \leq m$, 
\begin{align*}
\bbE_F\Big[ \langle \nabla h(z_0), F(j)\rangle \langle  \nabla h(z_0), F(k)\rangle\Big] &=
 \frac 1 {m!} \sum_{x \in \calG_m} \sum_{\substack{y \in \calG_m \\ y\neq x}}
   \sum_{\substack{G \in S_m\\ G(j) = x\\ G(k)=y}} \langle \nabla h(z_0), x\rangle \langle \nabla h(z_0),y\rangle   \\
   &= \frac 1 {m(m-1)}
 \sum_{x\neq y} \langle \nabla h(z_0), x\rangle\langle \nabla h(z_0), y\rangle =:M_{1,1},
 \end{align*}
which is again constant in $j,k$. It follows that 
\begin{align*}
\bbE_F &\left[ \sum_{j\neq k}  \langle \nabla h(z_0), F(j)\rangle\langle \nabla h(z_0), F(k)\rangle(q(j)-u(j))(q(k)-u(k))\right] \\
 &= M_{1,1} \sum_{j\neq k}  (q(j)-u(j)) (q(k) - u(k)) 
% &= M_{1,1} \sum_{j=1}^m (q(j)-u(j)) (u(j) - q(j)) \\
 = -M_{1,1} \sum_{j=1}^m \left(q(j) - \frac 1 m \right)^2 = -\frac{M_{1,1}}{m} \chi^2(q,u),
\end{align*}
whence, letting $M_2 := \bbE_F \Big[ \langle \nabla h(z_0), F(j)\rangle^2\Big]$, which itself is again constant
in $j$, we obtain 
\begin{align*}
\Var_F[\Gamma_n]
 &{=} \bbE_F\left[ \left(\sum_{j=1}^m \langle \nabla h(z_0), F(j)\rangle(q(j) - u(j))\right)^2\right] \\
 &= \sum_{j=1}^m \bbE_F \Big[ \langle \nabla h(z_0), F(j)\rangle^2\Big] (q(j) - u(j))^2 -\frac{M_{1,1}}{m} \chi^2(q,u)
% = M_2 \sum_{j=1}^m \left(q(j) - \frac 1 m \right)^2
 = \frac{M_2-M_{1,1}}{m}\chi^2(q,u).
\end{align*}
Therefore, by Markov's inequality, there exists a constant $C_{\lambda,\Lambda} > 0$ depending only on $M_2$, and hence
only on $\lambda,\Lambda, \alpha$, such that
\begin{equation}
\label{eq:pf_minimax_Gamma_n_Markov}
\bbP\left(|\Gamma_n| \geq C_{\lambda,\Lambda}\sqrt{\frac{\chi^2(q,u)}{m}}\right) \leq .025.
\end{equation}
% 
%
%\begin{align*} 
%\bbE_F&\left[\sum_{j,k=1}^m \langle \nabla h(z_0), F(j)\rangle\langle \nabla h(z_0), F(k)\rangle (q(j) - u(j))(q(k) - u(k))\right]\\
% &= \sum_{j\neq k} \bbE[\langle \nabla h(z_0), F(j)\rangle\langle \nabla h(z_0), F(k)\rangle] (q(j)-u(j))(q(k)-u(k))
% +  \sum_{j=1}^m \bbE[\langle \nabla h(z_0), Z\rangle^2] (q(j)-u(j))^2 \\
% &= -\sum_{j=1}^m \bbE[\langle \nabla h(z_0), Z\rangle]^2 (q(j)-u(j))^2
% +  \sum_{j=1}^m \bbE[\langle \nabla h(z_0), Z\rangle^2] (q(j)-u(j))^2 \\
% &= \Var_Z[\langle \nabla h(z_0), Z\rangle] \sum_{j=1}^m  \left(q(j)-\frac 1 m \right)^2\\
% &= \frac{\Var_Z[\langle \nabla h(z_0), Z\rangle] }{m^2} \chi^2(q,u).
%\end{align*}
Thus, returning to equation~\eqref{eq:pf_minimax_taylor}, and
recalling that for all $x, y \in \calG_m$, $\norm{x - y} \gtrsim m^{-1/d} I(x \neq y)$,
we deduce that for some $C_1 > 0$,  with probability at least $.975$,
\begin{align}
\label{eq:lower_bound_highprob}
\nonumber
\calT_c(F_\# q, (T_0 \circ F)_\# u) - h(z_0)
  &\geq C_1 \lambda m^{-\frac \alpha d} \bbP_{\pi}(X \neq Y)+\Gamma_m \\
\nonumber
  &\geq C_1\lambda m^{-\frac \alpha d} \TV(F_\# q,F_\# u) +    \Gamma_m \\
%\nonumber
%  &=    C_1\lambda m^{-\frac \alpha d} \TV(q,u) +    \Gamma_m \\
  &\geq C_1\lambda m^{-\frac \alpha d} \TV(q,u) - C_{\lambda,\Lambda}\sqrt{\frac{\chi^2(q,u)}{m}}.
\end{align}
We now prove the upper bound of the claim. 
Unlike before, we now let $\pi$ denote an optimal coupling between $F_\# q$ and $F_\# u$,
and $\tilde \pi = (Id, T_0)_\# \pi \in \Pi(F_\# q, (T_0\circ F)_\# u)$ a possibly suboptimal coupling.
%  an optimal coupling between $F_\# q$ and $(T_{0} \circ F)_\# u$,
%and let $ \pi = (Id, T_0^{-1})_\# \tilde \pi$ denote a coupling between $F_\# q$ and $F_\# u$. 
By assumption~\ref{assm:bded_smoothness}, we then have
\begin{align}
\label{eq:pf_minimax_ub}
\nonumber
\calT_c(F_\# q, (T_0 \circ F)_\# u) - h(z_0)
 &\leq \int\Big[h(y-x) - h(z_0)\Big] d\tilde \pi(x, y) \\
\nonumber
 &= \int \Big[ h(y-x+ z_0) - h(z_0)\Big] d\pi(x, y) \\
\nonumber
 &\leq \int \langle \nabla h(z_0), y-x \rangle d \pi(x,y) +  \Lambda  \int \norm{x-y}^\alpha d\pi(x,y)\\
 &= \Gamma_m +  \Lambda W_\alpha^\alpha(F_\# q, F_\# u).
 \end{align}
Now, by~\cite{niles-weed2019}, Proposition 9, there exists a constant $C_2 > 0$ such that
$$  W_\alpha^\alpha(F_\# u, F_\# q) % \leq C_{d} 3^{-2k_0} \TV(q, u) 
\leq C_2 m^{-\alpha/d} \big( \chi^2(q, u)\big)^{\alpha/d} \TV(q, u)^{1 - \frac {2\alpha}{d}}.$$
After possibly modifying   $C_2$, it follows from Markov's inequality that
$$\bbP\Big(   W_\alpha^\alpha(F_\# u, F_\# q)
\leq C_2 \Lambda m^{-\alpha/d}\big( \chi^2(q, u)\big)^{\alpha/d} \TV(q, u)^{1 - \frac {2\alpha}{d}}\Big) \geq .975,$$
so that, together with equations~\eqref{eq:pf_minimax_Gamma_n_Markov} and~\eqref{eq:pf_minimax_ub}, we have with probability at least $.95$, 
$$  \calT_c(F_\# u, (T_0 \circ F)_\# u)  - h(z_0)
\leq C_{\lambda,\Lambda} \sqrt{\frac{\chi^2(q,u)}{m}} + C_2 \Lambda m^{-\alpha/d}\big( \chi^2(q, u)\big)^{\alpha/d} \TV(q, u)^{1 - \frac {2\alpha}{d}}.$$
Combining this fact with equation \eqref{eq:lower_bound_highprob} and a union bound leads to the claim.
 \end{proof}
We now prove the main Theorem. In what follows, let
$\calD_m$ denote the set of probability distributions $q$ on $[m]$
satisfying $\chi^2(q, u) \leq 9$. Also, given $\delta > 0$, let 
$\calD_{m,\delta}^-$ denote the subset of distributions in $\calD_m$
satisfying $\TV(q, u) \leq \delta$, and 
by $\calD_m^+$ the subset of $\calD_m$ satisfying $\TV(q, u) \geq 1/4$.
Furthermore, set 
$$\Delta_m = \frac{C_1 \lambda m^{-\alpha/d}}{16},$$
and $\delta = \left(\frac{C_1\lambda}{288 \Lambda C_2}\right)^{\frac 1 {1 - \frac{2\alpha}{d}}}$. 
Since $d \geq 5 > 2\alpha$, we may assume that $m$ is large enough to satisfy
$$\Delta_m \geq  2C_{\lambda,\Lambda} \sqrt{\frac{9}{m}},$$
Then, by Proposition \ref{prop:w-tv-chi}, for all  $q \in \calD_{m,\delta}^-$, we have with probability at least $.9$,
\begin{align*}
\calT_c\big(F_\# q, (T_0 \circ F)_\# u\big) - h(z_0) 
 &\leq 
C_2 \Lambda m^{-\alpha/d} \big(\chi^2(q, u)\big)^{\frac \alpha d} \TV(q, u)^{1 - \frac {2\alpha} d}
+ C_{\lambda,\Lambda} \sqrt{\frac{\chi^2(q,u)}{m}} \\
 &\leq   \frac{C_1\lambda C_2 \Lambda m^{-\frac \alpha d} 9^{\frac \alpha d}}{288 \Lambda C_2}
+ C_{\lambda,\Lambda} \sqrt{\frac{9}{m}} 
 \leq   \frac{\Delta_m}{2 }
+ C_{\lambda,\Lambda} \sqrt{\frac{9}{m}} \leq \Delta_m.
\end{align*}
Similarly, for all $q \in \calD_m^+$, we have with probability at least .9,
\begin{align*}
\calT_c\big(F_\# q, (T_0 \circ F)_\# u\big) - h(z_0) 
 &\geq 
 C_1 \lambda m^{-\alpha/d} \TV(q, u) - C_{\lambda,\Lambda} \sqrt{\frac{\chi^2(q,u)}{m}}\\
 &\geq 
 \frac{C_1}{4} \lambda m^{-\alpha/d}  - C_{\lambda,\Lambda} \sqrt{\frac{9}{m}} 
 \geq 4\Delta_m  - C_{\lambda,\Lambda} \sqrt{\frac{9}{m}}\geq 3\Delta_m.
\end{align*}
%
%$$\inf_{q \in \calD_{m,\delta}^-}\bbP_F\Big(\calT_c(F_\# q, (T_0\circ F)_\# u) - h(z_0) \leq \Delta_m \Big) \geq .9.$$
%Similarly, we have 
%$$\inf_{q \in \calD_m^+}\bbP_F \Big(\calT_c (F_\# q, (T_0\circ F)_\# u) - h(z_0) \leq 3\Delta_m\Big) \geq .9,$$
%for a suitable choice of $\delta,c_1$ (have to do this bookkeeping eventually). 
%Now, for any estimator $\hat W_n$, define the test
%$$\psi = \psi(X_1, \dots, X_n) = I\Big( \hat W_n\big(F(X_1), \dots, F(X_n); F(Y_1), \dots, F(Y_n)\big) \leq 2\Delta_d\Big),$$
%for a sample $X_1, \dots, X_n$ from an unknown distribution on $[m]$ and for a sample $Y_1, \dots, Y_n \sim u$.
Now, for any given estimator $\hat \calT_n$ based on the independent
samples $X_1, \dots, X_n$ and $Y_1, \dots, Y_n$, 
define the event $A = \{|\hat \calT_n - \calT_c(F_\# q, (T_0\circ F)_\# u)| \geq \Delta_m\}$. 
We have by Markov's inequality,
\begin{align}
\nonumber
\sup_{\substack{\mu \in \calP(\calX) \\ \nu \in \calP(\calY)}} \bbE_{\mu,\nu}\big| \hat \calT_n - \calT_c(\mu,\nu)\big|
 &{\geq} \Delta_m \sup_{\substack{\mu \in \calP(\calX) \\ \nu \in \calP(\calY)}} 
 		\bbP_{\mu,\nu}\Big(|\hat \calT_n - \calT_c(\mu,\nu)| \geq \Delta_m\Big) \\
 &{\geq} \frac {\Delta_m} 2 \left\{ \sup_{q \in \calD_{m,\delta}^-} \bbE_F \bbP_{F_\# q,(T_0\circ F)_\# u}[A] + 
 						 \sup_{q \in \calD_{m}^+} \bbE_F \bbP_{F_\# q, (T_0\circ F)_\# u}[A]\right\}.
\label{eq:minimax_reduction} 						 
\end{align}

Notice that for all $q \in \calD_{m,\delta}^-$, we have
\begin{align*}
\bbE_F &\bbP_{F_\# q,(T_0\circ F)_\# u}[A] \\
 &\geq \bbE_F \bbP_{F_\# q,(T_0\circ F)_\# u}\Big(\hat \calT_n \geq h(z_0) + 2\Delta_m \text{ and } \calT_c(F_\# q, (T_0\circ F)_\# u) \leq h(z_0) + \Delta_m\Big)\\
 &\geq \bbE_F \bbP_{F_\# q,(T_0\circ F)_\# u}\Big( \hat \calT_n \geq h(z_0) + 2\Delta_m \Big) - 
       \bbP_F \Big(\calT_c(F_\# q, (T_0\circ F)_\# u) > h(z_0) + \Delta_m\Big)\\
 &\geq \bbE_F \bbP_{F_\# q,(T_0\circ F)_\# u}\Big( \hat \calT_n\geq h(z_0) + 2\Delta_m \Big) -   .1.
       \end{align*}
Similarly, for all $q \in \calD_m^+$, 
 \begin{align*}
\bbE_F& \bbP_{F_\# q,(T_0\circ F)_\# u}[A] \\
 &\geq \bbE_F \bbP_{F_\# q,(T_0\circ F)_\# u}\Big( \hat \calT_n \leq h(z_0) + 2\Delta_m \text{ and } \calT_c(F_\# q, (T_v\circ F)_\# u) \geq h(z_0) + 3\Delta_m\Big)\\
&\geq \bbE_F \bbP_{F_\# q,(T_0\circ F)_\# u}\Big(\hat \calT_n \leq h(z_0) + 2\Delta_m \Big) -  .1.
\end{align*}      
Returning to equation~\eqref{eq:minimax_reduction}, we thus have,
\begin{align*}
\inf_{\hat \calT_n}\sup_{\substack{\mu \in \calP(\calX) \\ \nu \in \calP(\calY)}}   \bbE_{\mu,\nu}\big| \hat \calT_n - \calT_c(\mu,\nu)\big|
\geq \frac {\Delta_m} 2 \inf_\psi\left\{ \sup_{q \in \calD_{m,\delta}^-} \bbP_q(\psi=1) +
 										   \sup_{q \in \calD_{m}^+} \bbP_q(\psi=0)  - 0.2\right\},
\end{align*}
where the infimum is over all tests based 
on the samples $X_1, \dots, X_n$ and $Y_1, \dots, Y_n$. By Proposition 10
of \cite{niles-weed2019}, the infimum on the right-hand side of the above display is bounded below by a constant
if $m\asymp n \log n$. The claim then follows by definition of $\Delta_m$. \qed

  \hspace{0.14in}
  
%\section*{Acknowledgements}
{\bf Acknowledgements.} The authors would like to thank an anonymous referee for comments
which significantly improved the quality of this paper.
TM was partially supported by the Natural Sciences and Engineering Research Council
of Canada, through a PGS D scholarship. JNW gratefully acknowledges the support of National Science Foundation grant DMS-2015291.

\bibliographystyle{imsart-nameyear} % Style BST file (imsart-number.bst or imsart-nameyear.bst)
\bibliography{manuscript_aoap}       % Bibliography file (usually '*.bib')

%% or include bibliography directly:
% \begin{thebibliography}{}
% \bibitem{b1}
% \end{thebibliography}

\end{document}